\documentclass[a4paper]{amsart}

\pdfoutput=1 

\usepackage{cmbright}
\usepackage{mathrsfs}
\usepackage{xcolor}
\usepackage[utf8]{inputenc}
\usepackage[T1]{fontenc}
\usepackage[english]{babel}
\usepackage{amsfonts}
\usepackage{amsthm}
\usepackage{amssymb}
\usepackage{mathcomp}
\usepackage{mathrsfs}
\usepackage[bb=boondox]{mathalfa}
\usepackage{graphicx}
\usepackage{epstopdf}
\usepackage[centertags]{amsmath}
\usepackage[pdftex]{hyperref}
\usepackage{vmargin}
\usepackage{tikz}
\usepackage{tikz-cd}
\usepackage{dsfont}
\usepackage{cases}
\usepackage{mathtools}
\usepackage{shuffle}
\usepackage{adjustbox}

\theoremstyle{plain}
\newtheorem{theorem}{Theorem}
\newtheorem*{theorem*}{Theorem}
\newtheorem{lemma}{Lemma}
\newtheorem{proposition}{Proposition}
\newtheorem{corollary}{Corollary}

\theoremstyle{definition}
\newtheorem{definition}{Definition}

\theoremstyle{remark}
\newtheorem{remark}{\sc Remark}
\newtheorem*{notation}{\sc Notation}

\newcommand{\cat}{\mathrm{cat}}
\newcommand{\lax}{\mathrm{lax}}
\newcommand{\oplax}{\mathrm{oplax}}
\newcommand{\strong}{\mathrm{strong}}
\newcommand{\strict}{\mathrm{strict}}
\newcommand{\pl}{\mathrm{pl}}
\newcommand{\Comonads}{\mathsf{Comonads}}
\newcommand{\Monads}{\mathsf{Monads}}
\newcommand{\Cats}{\mathsf{Cats}}
\newcommand{\Functors}{\mathsf{Functors}}
\newcommand{\Mor}{\mathsf{Mor}}
\newcommand{\Nat}{\mathsf{Nat}}
\newcommand{\sk}{\mathrm{sk}}
\newcommand{\TwoFun}{\text{2-Fun}}
\newcommand{\Ob}{\mathrm{Ob}}
\newcommand{\tr}{\mathrm{tr}}

\newcommand{\categ}[1]{\mathsf{#1}}
\newcommand{\set}[1]{\mathrm{#1}}

\newcommand{\catoperad}[1]{\mathsf{#1}}
\newcommand{\algebra}[1]{\mathrm{#1}}
\newcommand{\coalgebra}[1]{\mathrm{#1}}

\newcommand{\catcog}[2]{\Cog_{\categ{#1}}\left(#2\right)}
\newcommand{\catalg}[2]{\Alg_{\categ{#1}}\left(#2\right)}
\newcommand{\tcatalg}[2]{\mathsf{ALG}_{\categ{#1}}\left(#2\right)}
\newcommand{\tcatcog}[2]{\mathsf{COG}_{\categ{#1}}\left(#2\right)}

\newcommand{\Fun}[3]{\mathrm{Fun}^{#1}\left(#2,#3\right)}

\newcommand{\forget}{\mathrm{U}}

\newcommand{\Operad}{\mathsf{Operad}}

\newcommand{\mbk}{\mathbb{K}}

\newcommand{\id}{\mathrm{Id}}

\newcommand{\Hom}[3]{\mathrm{hom}_{#1}\left(#2 , #3 \right)}

\newcommand{\II}{\mathbb{1}}

\newcommand{\ra}{\rightarrow}

\newcommand{\Cog}{\mathsf{Cog}}

\newcommand{\Alg}{\mathsf{Alg}}

\newcommand{\Set}{\mathsf{Set}}

\newcommand{\colim}{\mathrm{colim}}

\newcommand{\op}{\mathrm{op}}

\newcommand{\catEnd}{\mathsf{End}}

\newcommand{\itemt}{\item[$\triangleright$]}

\newcommand{\poubelle}[1]{}

\def\corollatwo{\tikz[baseline=.1ex]{
\draw (1ex,0ex) -- (1ex,1ex);
\fill (1ex,1ex) circle (1.3pt);
\draw (1ex,1ex) -- (0ex,2ex);
\draw (1ex,1ex) -- (2ex,2ex);}
}

\def\corollazero{\tikz[baseline=.1ex]{
\draw (0,0ex) -- (0ex,1ex);
\fill (0ex,1ex) circle (1.7pt);}
}

\def\corollan{\tikz[baseline=.1ex]{
\draw (3ex,0ex) -- (3ex,1ex);
\fill (3ex,1ex) circle (1.3pt);
\draw (3ex,1ex) -- (0ex,2ex);
\draw (3ex,1ex) -- (1ex,2ex);
\draw (3ex,1ex) -- (2ex,2ex);
\draw (3ex,1ex) -- (6ex,2ex);
\draw (4ex,2ex) node {$\cdots$};}
}

\def\treeassleft{\tikz[baseline=.1ex]{
\draw (2ex,0ex) -- (2ex,1ex);
\fill (2ex,1ex) circle (1.3pt);
\draw (2ex,1ex) -- (0ex,3ex);
\draw (2ex,1ex) -- (3ex,2ex);
\fill (1ex,2ex) circle (1.3pt);
\draw (1ex,2ex) -- (2ex,3ex);}
}

\def\treeassright{\tikz[baseline=.1ex]{
\draw (1ex,0ex) -- (1ex,1ex);
\fill (1ex,1ex) circle (1.3pt);
\draw (1ex,1ex) -- (0ex,2ex);
\draw (1ex,1ex) -- (3ex,3ex);
\fill (2ex,2ex) circle (1.3pt);
\draw (2ex,2ex) -- (1ex,3ex);}
}

\def\treeunitleft{\tikz[baseline=.1ex]{
\draw (1ex,0ex) -- (1ex,1ex);
\fill (1ex,1ex) circle (1.3pt);
\draw (1ex,1ex) -- (0ex,2ex);
\draw (1ex,1ex) -- (2ex,2ex);
\fill (0ex,2ex) circle (1.7pt);}
}

\def\treeunitright{\tikz[baseline=.1ex]{
\draw (1ex,0ex) -- (1ex,1ex);
\fill (1ex,1ex) circle (1.3pt);
\draw (1ex,1ex) -- (0ex,2ex);
\draw (1ex,1ex) -- (2ex,2ex);
\fill (2ex,2ex) circle (1.7pt);}
}

\def\treepentagonna{\tikz[baseline=.1ex]{
\draw (3ex,0ex) -- (3ex,1ex);
\fill (3ex,1ex) circle (1.3pt);
\fill (2ex,2ex) circle (1.3pt);
\fill (1ex,3ex) circle (1.3pt);
\draw (3ex,1ex) -- (0ex,4ex);
\draw (3ex,1ex) -- (0ex,4ex);
\draw (3ex,1ex) -- (4ex,2ex);
\draw (2ex,2ex) -- (3ex,3ex);
\draw (1ex,3ex) -- (2ex,4ex);}
}

\def\treepentagonnb{\tikz[baseline=.1ex]{
\draw (2.5ex,0ex) -- (2.5ex,1ex);
\fill (2.5ex,1ex) circle (1.3pt);
\draw (2.5ex,1ex) -- (1ex,2ex);
\draw (2.5ex,1ex) -- (4ex,2ex);
\draw (1ex,2ex) -- (0ex,3ex);
\fill (1ex,2ex) circle (1.3pt);
\draw (1ex,2ex) -- (2ex,3ex);
\draw (4ex,2ex) -- (3ex,3ex);
\fill (4ex,2ex) circle (1.3pt);
\draw (4ex,2ex) -- (5ex,3ex);}
}

\def\treepentagonnc{\tikz[baseline=.1ex]{
\draw (2ex,0ex) -- (2ex,1ex);
\fill (2ex,1ex) circle (1.3pt);
\draw (2ex,1ex) -- (0ex,3ex);
\fill (1ex,2ex) circle (1.3pt);
\draw (2ex,1ex) -- (3ex,2ex);
\fill (2ex,3ex) circle (1.3pt);
\draw (1ex,2ex) -- (3ex,4ex);
\draw (2ex,3ex) -- (1ex,4ex);}
}

\def\treepentagonnd{\tikz[baseline=.1ex]{
\draw (1ex,0ex) -- (1ex,1ex);
\fill (1ex,1ex) circle (1.3pt);
\draw (1ex,1ex) -- (0ex,2ex);
\draw (1ex,1ex) -- (3ex,3ex);
\fill (2ex,2ex) circle (1.3pt);
\draw (2ex,2ex) -- (0ex,4ex);
\fill (1ex,3ex) circle (1.3pt);
\draw (1ex,3ex) -- (2ex,4ex);}
}

\def\treepentagonne{\tikz[baseline=.1ex]{
\draw (1ex,0ex) -- (1ex,1ex);
\fill (1ex,1ex) circle (1.3pt);
\draw (1ex,1ex) -- (0ex,2ex);
\draw (1ex,1ex) -- (4ex,4ex);
\fill (2ex,2ex) circle (1.3pt);
\draw (2ex,2ex) -- (1ex,3ex);
\fill (3ex,3ex) circle (1.3pt);
\draw (3ex,3ex) -- (2ex,4ex);}
}

\def\treetrianglea{\tikz[baseline=.1ex]{
\draw (2ex,0ex) -- (2ex,1ex);
\fill (2ex,1ex) circle (1.3pt);
\draw (2ex,1ex) -- (0ex,3ex);
\draw (2ex,1ex) -- (3ex,2ex);
\fill (1ex,2ex) circle (1.3pt);
\draw (1ex,2ex) -- (2ex,3ex);
\fill (2ex,3ex) circle (1.7pt);
}
}

\def\treetriangleb{\tikz[baseline=.1ex]{
\draw (1ex,0ex) -- (1ex,1ex);
\fill (1ex,1ex) circle (1.3pt);
\draw (1ex,1ex) -- (0ex,2ex);
\draw (1ex,1ex) -- (3ex,3ex);
\fill (2ex,2ex) circle (1.3pt);
\draw (2ex,2ex) -- (1ex,3ex);
\fill (1ex,3ex) circle (1.7pt);}}

\def\treepentagonnmoda{\tikz[baseline=.1ex]{
\draw (3ex,0ex) -- (3ex,1ex);
\fill (2.5ex,0.7ex)rectangle(3.5ex,1.3ex);
\fill (2ex,2ex) circle (1.3pt);
\fill (1ex,3ex) circle (1.3pt);
\draw (3ex,1ex) -- (0ex,4ex);
\draw (3ex,1ex) -- (0ex,4ex);
\draw (3ex,1ex) -- (4ex,2ex);
\draw (2ex,2ex) -- (3ex,3ex);
\draw (1ex,3ex) -- (2ex,4ex);}
}

\def\treepentagonnmodb{\tikz[baseline=.1ex]{
\draw (2.5ex,0ex) -- (2.5ex,1ex);
\fill (2ex,0.7ex) rectangle (3ex,1.3ex);
\draw (2.5ex,1ex) -- (1ex,2ex);
\draw (2.5ex,1ex) -- (4ex,2ex);
\draw (1ex,2ex) -- (0ex,3ex);
\fill (1ex,2ex) circle (1.3pt);
\draw (1ex,2ex) -- (2ex,3ex);
\draw (4ex,2ex) -- (3ex,3ex);
\fill (3.5ex,1.7ex) rectangle (4.5ex,2.3ex);
\draw (4ex,2ex) -- (5ex,3ex);}
}

\def\treepentagonnmodc{\tikz[baseline=.1ex]{
\draw (2ex,0ex) -- (2ex,1ex);
\fill (1.5ex,0.7ex)rectangle(2.5ex,1.3ex);
\draw (2ex,1ex) -- (0ex,3ex);
\fill (1ex,2ex) circle (1.3pt);
\draw (2ex,1ex) -- (3ex,2ex);
\fill (2ex,3ex) circle (1.3pt);
\draw (1ex,2ex) -- (3ex,4ex);
\draw (2ex,3ex) -- (1ex,4ex);}
}

\def\treepentagonnmodd{\tikz[baseline=.1ex]{
\draw (1ex,0ex) -- (1ex,1ex);
\fill (0.5ex,0.7ex)rectangle(1.5ex,1.3ex);
\draw (1ex,1ex) -- (0ex,2ex);
\draw (1ex,1ex) -- (3ex,3ex);
\fill (1.5ex,1.7ex)rectangle(2.5ex,2.3ex);
\draw (2ex,2ex) -- (0ex,4ex);
\fill (1ex,3ex) circle (1.3pt);
\draw (1ex,3ex) -- (2ex,4ex);}
}

\def\treepentagonnmode{\tikz[baseline=.1ex]{
\draw (1ex,0ex) -- (1ex,1ex);
\fill (0.5ex,0.7ex)rectangle(1.5ex,1.3ex);
\draw (1ex,1ex) -- (0ex,2ex);
\draw (1ex,1ex) -- (4ex,4ex);
\fill (1.5ex,1.7ex)rectangle(2.5ex,2.3ex);
\draw (2ex,2ex) -- (1ex,3ex);
\fill (2.5ex,2.7ex)rectangle(3.5ex,3.3ex);
\draw (3ex,3ex) -- (2ex,4ex);}
}

\def\corollamod{\tikz[baseline=.1ex]{
\draw (1ex,0ex) -- (1ex,1ex);
\fill (0.5ex,0.7ex)rectangle(1.5ex,1.3ex);
\draw (1ex,1ex) -- (0ex,2ex);
\draw (1ex,1ex) -- (2ex,2ex);}
}

\def\treeassleftmod{\tikz[baseline=.1ex]{
\draw (2ex,0ex) -- (2ex,1ex);
\fill (1.5ex,0.7ex)rectangle(2.5ex,1.3ex);
\fill (1ex,2ex) circle (1.3pt);
\draw (2ex,1ex) -- (0ex,3ex);
\draw (1ex,2ex) -- (2ex,3ex);
\draw (2ex,1ex) -- (3ex,2ex);}
}

\def\treeassrightmod{\tikz[baseline=.1ex]{
\draw (1ex,0ex) -- (1ex,1ex);
\fill (0.5ex,0.7ex)rectangle(1.5ex,1.3ex);
\fill (1.5ex,1.7ex)rectangle(2.5ex,2.3ex);
\draw (1ex,1ex) -- (0ex,2ex);
\draw (1ex,1ex) -- (3ex,3ex);
\draw (2ex,2ex) -- (1ex,3ex);}
}

\def\treeunitmodule{\tikz[baseline=.1ex]{
\draw (1ex,0ex) -- (1ex,1ex);
\fill (0.5ex,0.7ex)rectangle(1.5ex,1.3ex);
\fill (0ex,2ex) circle (1.7pt);
\draw (1ex,1ex) -- (0ex,2ex);
\draw (1ex,1ex) -- (2ex,2ex);}
}

\def\treeunitmoduleasslefta{\tikz[baseline=.1ex]{
\draw (2ex,0ex) -- (2ex,1ex);
\fill (1.5ex,0.7ex)rectangle(2.5ex,1.3ex);
\draw (2ex,1ex) -- (0ex,3ex);
\draw (1ex,2ex) -- (2ex,3ex);
\fill (1ex,2ex) circle (1.3pt);
\fill (0ex,3ex) circle (1.7pt);
\draw (2ex,1ex) -- (3ex,2ex);}
}

\def\treeunitmoduleassleftb{\tikz[baseline=.1ex]{
\draw (1ex,0ex) -- (1ex,1ex);
\fill (0.5ex,0.7ex)rectangle(1.5ex,1.3ex);
\fill (1.5ex,1.7ex)rectangle(2.5ex,2.3ex);
\fill (0ex,2ex) circle (1.7pt);
\draw (1ex,1ex) -- (0ex,2ex);
\draw (1ex,1ex) -- (3ex,3ex);
\draw (2ex,2ex) -- (1ex,3ex);}
}

\def\treeunitmoduleassrighta{\tikz[baseline=.1ex]{
\draw (2ex,0ex) -- (2ex,1ex);
\fill (1.5ex,0.7ex)rectangle(2.5ex,1.3ex);
\draw (2ex,1ex) -- (0ex,3ex);
\draw (1ex,2ex) -- (2ex,3ex);
\fill (1ex,2ex) circle (1.3pt);
\fill (2ex,3ex) circle (1.7pt);
\draw (2ex,1ex) -- (3ex,2ex);}
}

\def\treeunitmoduleassrightb{\tikz[baseline=.1ex]{
\draw (1ex,0ex) -- (1ex,1ex);
\fill (0.5ex,0.7ex)rectangle(1.5ex,1.3ex);
\fill (1.5ex,1.7ex)rectangle(2.5ex,2.3ex);
\fill (1ex,3ex) circle (1.7pt);
\draw (1ex,1ex) -- (0ex,2ex);
\draw (1ex,1ex) -- (3ex,3ex);
\draw (2ex,2ex) -- (1ex,3ex);}
}

\def\corollatwocomm{\tikz[baseline=.1ex]{
\draw (1ex,0ex) -- (1ex,1ex);
\fill (1ex,1ex) circle (1.3pt);
\draw (1ex,1ex) -- (0ex,2ex);
\draw (1ex,1ex) -- (2ex,2ex);
\draw (0ex,2ex) -- (2ex,3ex);
\draw (2ex,2ex) -- (0ex,3ex);}
}

\def\corollacommassa{\tikz[baseline=.1ex]{
\draw (1ex,0ex) -- (1ex,1ex);
\fill (1ex,1ex) circle (1.3pt);
\draw (1ex,1ex) -- (0ex,2ex);
\draw (1ex,1ex) -- (2ex,2ex);
\draw (0ex,2ex) -- (2ex,3ex);
\draw (2ex,2ex) -- (0ex,3ex);
\fill (2ex,3ex) circle (1.3pt);
\draw (2ex,3ex) -- (1ex,4ex);
\draw (2ex,3ex) -- (3ex,4ex);}
}

\def\corollacommassb{\tikz[baseline=.1ex]{
\draw (2ex,0ex) -- (2ex,1ex);
\fill (2ex,1ex) circle (1.3pt);
\draw (2ex,1ex) -- (0ex,3ex);
\draw (2ex,1ex) -- (3ex,2ex);
\fill (1ex,2ex) circle (1.3pt);
\draw (1ex,2ex) -- (3ex,4ex);
\draw (3ex,2ex) -- (3ex,3ex);
\draw (3ex,3ex) -- (0ex,4ex);
\draw (0ex,3ex) -- (2ex,4ex);}
}

\def\corollacommassc{\tikz[baseline=.1ex]{
\draw (2ex,0ex) -- (2ex,1ex);
\fill (2ex,1ex) circle (1.3pt);
\draw (2ex,1ex) -- (0ex,3ex);
\draw (2ex,1ex) -- (3ex,2ex);
\fill (1ex,2ex) circle (1.3pt);
\draw (1ex,2ex) -- (2ex,3ex);
\draw (0ex,3ex) -- (2ex,4ex);
\draw (2ex,3ex) -- (0ex,4ex);}
}

\def\corollacommassd{\tikz[baseline=.1ex]{
\draw (1ex,0ex) -- (1ex,1ex);
\fill (1ex,1ex) circle (1.3pt);
\draw (1ex,1ex) -- (0ex,2ex);
\draw (1ex,1ex) -- (3ex,3ex);
\fill (2ex,2ex) circle (1.3pt);
\draw (2ex,2ex) -- (0ex,4ex);
\draw (0ex,2ex) -- (2ex,4ex);}
}

\def\corollacommasse{\tikz[baseline=.1ex]{
\draw (1ex,0ex) -- (1ex,1ex);
\fill (1ex,1ex) circle (1.3pt);
\draw (1ex,1ex) -- (0ex,2ex);
\draw (1ex,1ex) -- (3ex,3ex);
\fill (2ex,2ex) circle (1.3pt);
\draw (2ex,2ex) -- (1.5ex,2.5ex);
\draw (1.5ex,2.5ex) -- (3ex,4ex);
\draw (3ex,3ex) -- (0ex,4ex);
\draw (0ex,2ex) -- (0ex,3ex);
\draw (0ex,3ex) -- (2ex,4ex);}
}

\def\corollatwodoublecomm{\tikz[baseline=.1ex]{
\draw (1ex,0ex) -- (1ex,1ex);
\fill (1ex,1ex) circle (1.3pt);
\draw (1ex,1ex) -- (0ex,2ex);
\draw (1ex,1ex) -- (2ex,2ex);
\draw (0ex,2ex) -- (2ex,3ex);
\draw (2ex,3ex) -- (0ex,4ex);
\draw (2ex,2ex) -- (0ex,3ex);
\draw (0ex,3ex) -- (2ex,4ex);}
}

\title{Mapping coalgebras I\\Comonads}
\author{Brice Le Grignou}
\email{bricelegrignou "at" gmail.com}

\date{\today}

\begin{document}

\begin{abstract}
    In this article we describe properties of the 2-functor from the 2-category of comonads to the 2-category of functors that sends a comonad to its forgetful functor. This allows us to describe contexts where algebras over a monad are enriched tensored and cotensored over coalgebras over a comonad.
\end{abstract}
\maketitle

\setcounter{tocdepth}{1}
\tableofcontents

\section*{Introduction}

This is the first of a series of articles about categories enriched over a category describing some notion of a coalgebra. This article is devoted to coalgebras over a comonad. The next article will deal with coalgebras over an operad and the last article will focus on coalgebras in the context of chain complexes.

Let us start with Sweedler's theory (\cite{Sweedler69}, \cite{AnelJoyal13}). For any two differential graded coassociative coalgebras
$(\coalgebra V, w_V, \tau_V)$ and $(\coalgebra W,w_W,\tau_W)$, the tensor product $V \otimes W$ inherits
the structure of a coassociative coalgebra as follows
\[
	V \otimes W \xrightarrow{w_V \otimes w_W} V \otimes V \otimes W \otimes W
	\simeq V \otimes W \otimes V \otimes W .
\]
This gives a symmetric monoidal structure on the category of differential graded coalgebras.
Moreover, from the existence of cofree coalgebras \cite{Anel14}, one can show that
this monoidal structure is closed. Besides, for any differential graded associative algebra
$(\algebra A,m)$ and any differential graded coalgebra $(\coalgebra V,w)$, the mapping chain complex
$[\coalgebra V,\algebra A]$ has the canonical structure of an algebra
\[
	[\coalgebra V, \algebra A] \otimes [\coalgebra V, \algebra A] \hookrightarrow
	[\coalgebra V \otimes \coalgebra V, \algebra A \otimes \algebra A]
	\xrightarrow{[w,m]} [\coalgebra V,\algebra A].
\]
It is usually called the convolution algebra of $\coalgebra V$ and $\algebra A$.
From the existence of free algebras and cofree coalgebras,
one can build a left adjoint $V \boxtimes -$ to the functor $[\coalgebra V, -]$
and a left adjoint $\{- , \algebra A\}$ to the functor $[- , \algebra A]$.
These three bifunctors
\begin{align*}
	[-,-] :& \mathsf{Coalgebras}^{\op} \times \mathsf{Algebras} \to \mathsf{Algebras}
	\\
	- \boxtimes - : & \mathsf{Coalgebras} \times \mathsf{Algebras} \to \mathsf{Algebras}
	\\
	\{-,-\} : & \mathsf{Algebras}^\op \times \mathsf{Algebras} \to \mathsf{Coalgebras}
\end{align*}
make the category of algebras tensored, cotensored and enriched
over the category of coalgebras.

This may be reinterpreted in the language of monads and comonads since associative dg algebras and coassociative dg coalgebras are respectively algebras over a monad $M$ and coalgebras over a comonad $Q$ on chain complexes (see \cite{Anel14} for a description of $Q$). Indeed, the structure of a monoidal category on coalgebras is related to the structure of a Hopf comonad on $Q$, that is the data of a natural map
$$
    Q(X) \otimes Q(Y) \to Q(X \otimes Y) 
$$
satisfying dual conditions as those of a Hopf monad (see \cite{Moerdijk02}). Moreover, the lifting of the mapping chain complex bifunctor $X,Y \mapsto [X,Y]$ to a bifunctor from $\mathsf{Coalgebras}^{\op} \times \mathsf{Algebras}$ to $\mathsf{Algebras}$ is related to the structure of a Hopf module monad on $M$ with respect to the Hopf comonad $Q$, that is the data of a natural map
$$
    M([X,Y]) \to [Q(X),M(Y)]
$$
satisfying again some conditions. All these relations between monoidal structures on categories of algebras and coalgebras and structures on monads and comonads are encoded in a single 2-functor from the 2-category $\Comonads$ of categories with comonads and oplax morphisms to the 2-category $\Functors$ made up of functors between two categories.

\begin{theorem*}
The construction that sends a category $\categ C$ with a comonad $Q$ to the forgetful functor
$U_Q$ from $Q$-coalgebras to $\categ C$
induces a 2-functor from the two category $\Comonads$ to the 2-category $\Functors$ that is strictly fully faithful and preserves strict finite products.
\end{theorem*}

This theorem shows in particular that these relations between monoidal structures on categories of algebras and coalgebras and structures on monads and comonads are bijections. For instance, a monoidal structure on the category of $M$-algebras that lifts that of chain complexes is equivalent to the structure of a Hopf monad on $M$ (see again \cite{Moerdijk02}). Moreover, the existence of the two bifunctors $- \boxtimes -$ and $\{-, - \}$ described above is then just a consequence of Johnstone's adjoint lifiting theorem (\cite{Johnstone75}).

\subsection*{Layout}

In the first section, we describe a strict version of a monoidal symmetric monoidal category that we call a monoidal context. In the second section, we describe the monoidal context of comonads and relate it to the monoidal context of functors. In the third section, we apply this relation and the adjoint lifting theorem to describe mapping coalgebras.

\subsection*{Acknowledgement}

I would like to thank Damien Lejay and Mathieu Anel for inspiring discussions. The idea of this work came when
reading the book \cite{AnelJoyal13} by Mathieu Anel and Andr\'e Joyal. This work was supported in part by
the NWO grant of Ieke Moerdijk.

\subsection*{Universes}

Let us consider three universes $\mathcal U \in \mathcal{V} \in \mathcal W$.
A set is called $\mathcal U$-small if it is an element of $\mathcal U$ and
it is called $\mathcal U$-large if it is a subset of $\mathcal U$. The notion
of smallness and largeness are defined similarly for the other universes.
We thus have a hierarchy of sizes of sets
$$
\mathcal{U}\text{ small sets}
\subset \mathcal{U}\text{ large sets}
\subset \mathcal{V}\text{ small sets}
\subset \mathcal{V}\text{ large sets}
\subset \mathcal{W}\text{ small sets}
\subset \mathcal{W}\text{ large sets}.
$$
Besides, a $\mathcal U$-category is a category whose set of objects
is $\mathcal{U}$-large and whose hom sets are all $\mathcal{U}$-small.
Such a $\mathcal U$-category is called $\mathcal U$-small if its set
of objects is $\mathcal U$-small. We have similar notions of
$\mathcal V$ and $\mathcal W$.

Finally, we will use the following aliases:
\begin{itemize}
    \itemt a set, also called small set will be a
    $\mathcal{W}$-small set;
    \itemt a large set will be a
    $\mathcal{W}$-large set;
    \itemt a category will be a $\mathcal W$-category;
    \itemt a small category will be a $\mathcal W$-small category.
\end{itemize}

\subsection*{Some notations}

\begin{itemize}
    \itemt For any natural integer $n$, the set of permutations of the set
    $$
\underline n = \{1, \ldots, n\}
    $$
    will be denoted $\mathbb S_n$.
    \itemt For any natural integer $n$, any permutation $\sigma \in \mathbb S_n$
    and any category $\categ C$,
    we denote $\sigma^\ast$ the following functor
    $$
    \categ C^n = \categ{Fun}(\underline{n}, \categ C)
    \xrightarrow{- \circ \sigma}
    \categ{Fun}(\underline{n}, \categ C)
    = \categ C^n .
    $$
    \itemt For a comonad $Q$ on a category $\categ E$, the induced comonadic adjunction relating $Q$-coalgebras
    to $\categ E$ will be denoted $\forget_Q \dashv L^Q$.
    
    \itemt For a monad $M$, we will sometimes denote
    $T_M \dashv \forget^M$ the induced monadic adjunction relating $M$-algebras
    to $\categ E$. Otherwise, if the context is clear, we will use the same notation for the monad
    $M$ and the left adjoint functor from the ground category to the category of algebras.
\end{itemize}


\section{Monoidal context}

In this section, we describe a semi-strict notion of a symmetric monoidal 2-category that we call a monoidal context. Then, we describe some usual algebraic notions in a monoidal context.

We refer to \cite{DAY199799} and \cite{SchommerPries09} for broader notions of symmetric monoidal 2-categories. In particular, the PhD thesis of Chris Schommer-Pries describes in details a "full" version of a symmetric monoidal bicategory.

Let us fix a universe (except in the last subsection); for instance $\mathcal U$.
Then $\mathcal U$-small sets will be called small sets or just sets
and
$\mathcal U$-large sets will be called large sets.

\subsection{Strict 2-categories}

\subsubsection{The definition of a 2-category}

\begin{definition}
 A strict 2-category $\categ C$ is a category enriched in categories, that is the data of
\begin{itemize}
 \itemt a large set of objects $\Ob(\categ C)$;
 \itemt for any two objects $X,Y$, a small category $\categ C(X,Y)$;
 \itemt a unital associative composition $\categ C (Y,Z) \times \categ C(X,Z) \to \categ C (Y,Z)$
 whose units are elements $1_X \in \categ C(X,X)$.
\end{itemize}
A strict 2-functor $F$ between two strict 2-categories $\categ C$
and $\categ D$ is a morphism
of categories enriched in categories , that is the data of
\begin{itemize}
 \itemt a function $F(-) : \Ob(\categ C) \to \Ob(\categ D)$;
  \itemt functors $F(X,Y): \categ C(X,Y) \to \categ D(F(X),F(Y))$ that commute with the composition and send
  units to units.
\end{itemize}
\end{definition}

\begin{definition}
    A strict 2-category is called small if its set of objects is small.
    The data of small strict 2-categories and strict 2-functors
    form the category $\categ{2-Cats}$.
\end{definition}

\begin{remark}
From now on, we will call strict 2-functors just 2-functors.
\end{remark}

\begin{definition}
 Given a strict 2-category $\categ C$ and two objects $X,Y \in \Ob(\categ C)$,
 \begin{itemize}
     \itemt an object of $\categ C(X,Y)$
     will be called a morphism of $\categ C$;
     \itemt a morphism in the category $\categ C(X,Y)$ will be called a 2-morphism of $\categ C$.
 \end{itemize}
\end{definition}

\begin{definition}
Let $\categ C$, be a strict 2-category and $X,Y,Z$ be three objects.
\begin{itemize}
    \itemt Given morphisms $f_1, f_2 : X \to Y$ and $g_1,g_2 : Y \to Z$ and 2-morphisms $a : f_1 \to f_2$ and 
    $b : g_1 \to g_2$, one can compose $a$ and $b$ using the functor
    $$
    \categ C(Y,Z) \times \categ C (X,Y)
    \to \categ C (X,Z)
    $$
    to obtain a 2-morphism from $g_1 \circ f_1$ to $g_2 \circ f_2$ that is denoted $b \circ_h a$ and is called the horizontal composition of $a$ and $b$.
    \itemt Given morphisms $f_1, f_2, f_3 : X \to Y$ and 2-morphisms $a_1 : f_1 \to f_2$ and $a_2 : f_2 \to f_3$, one can compose $a_1$ and $a_2$ as morphisms in the category $\categ C (X,Y)$ to obtain a 2-morphism from $f_1$ to $f_3$, denoted $a_2 \circ_v a_1$ and called the vertical composition of $a_1$ and $a_2$.
\end{itemize}
\end{definition}

\begin{notation}
Let us consider a 2-morphism $a$ of $\categ C$, that is a morphism in $\categ C(X,Y)$ for two objects $X,Y$. For any morphism $f : Y \to Z$, we will usually denote the horizontal composition $\id_f \circ_h a$ as
 $$
    f \circ_h a .
 $$
 Similarly, for any morphism $g : Z \to X$, we will denote $a \circ_h \id_g$ as $a \circ_h g$. 
\end{notation}

\begin{definition}
An isomorphism in a strict 2-category $\categ C$ is an isomorphism of the underlying category (that is the skeleton of $\categ C$ denoted $\sk(\categ C))$. A 2-isomorphism is an invertible 2-morphism. Finally, an equivalence in $\categ C$ is a morphism $f: X \to Y$ so that there exists a morphism
$g : Y \to X$ and 2-isomorphisms
\begin{align*}
    a : \id_{X} \simeq g \circ f;
    \\
    b : \id_{Y} \simeq f \circ g.
\end{align*}
Then, $g$ is a pseudo-inverse of $f$.
\end{definition}

\begin{definition}
An adjunction in $\categ C$ is the data of two objects $X,Y$ together with morphisms
\begin{align*}
    l : X \to Y
    \\
    r : Y \to X
\end{align*}
and 2-morphisms $\eta : \id_X \to rl$
and $\epsilon : lr \to \id_Y$ so that
\begin{align*}
    (r \circ_h \epsilon) \circ_v (\eta \circ_h r)
    = \id_R;
    \\
    (\epsilon \circ_h l) \circ_v (l \circ_h \eta)
    = \id_R.
\end{align*}
\end{definition}

\begin{definition}
An adjoint equivalence is an adjunction whose unit and counit are isomorphisms.
\end{definition}

\begin{proposition}
If $f: X \to Y$ is an equivalence of a strict 2-category with pseudo-inverse $g$, then, $f$ and $g$ are part of an adjoint equivalence.
\end{proposition}

\begin{proof}
 Let us consider two 2-isomorphisms
 \begin{align*}
     \eta : \id_Y \simeq fg;
     \\
     \zeta : gf \simeq \id_X .
 \end{align*}
 We can notice that the two maps $gf \circ_h \zeta$ and $\zeta \circ_h gf$ from $gfgf$ to $gf$ are equal.
 If we define the 2-isomorphism $\epsilon :gf \simeq \id_X$ as the composition
 $$
    gf \xrightarrow{ gf \circ_h \zeta^{-1} = \zeta^{-1} \circ_h gf} gfgf
    \xrightarrow{g \circ_h \eta^{-1} \circ_h f}
    gf
    \xrightarrow{\zeta^{-1}}
    \id_X,
 $$
 then, the tuple $(f,g,\eta,\epsilon)$ is an adjunction.
\end{proof}

\subsubsection{The strict 2-category of strict 2-categories}

Categories are organised into a 2-category. Similarly, 2-categories form a 3-category. But in the same way one can consider the category of categories, one can also consider a 2-category of strict 2-categories. This implies that we will not consider mapping categories up to equivalences but up to isomorphisms.

\begin{definition}
Categories form a coreflexive full subcategory of strict 2-categories.
We denote $\sk$ the corresponding idempotent comonads
that sends a strict 2-category to its underlying category (that is with the same objects and morphisms) called its skeleton.
\end{definition}

\begin{definition}
 Let $\categ{Mor}$ be the category
 with two objects $0$ and $1$ and a non trivial
 morphism from $0$ to $1$. Let $\categ{Nat}$
 be the strict 2-category with two objects $0$
 and $1$ and so that
 $$
 \begin{cases}
      \categ{Nat}(0,1) = \categ{Mor};
      \\
      \categ{Nat}(0,0) =  \categ{Nat}(1,1)= \ast;
      \\
      \categ{Nat}(1,0) = \emptyset.
 \end{cases}
 $$
\end{definition}

Given two strict 2-categories $\categ C, \categ D$, their product $\categ C\times \categ D$
is the strict 2-categories whose
\begin{itemize}
 \itemt set of objects is $\Ob(\categ C\times \categ D) = \Ob(\categ C)\times \Ob(\categ D)$;
 \itemt categories of morphisms are
 \[
 	\categ C\times \categ D ((X,X'),(Y,Y')) = \categ C (X,Y)\times \categ D(X',Y') .
 \]
\end{itemize}

\begin{proposition}
 The category $\categ{2-Cats}$ of small
 strict 2-categories is a cartesian closed monoidal category.
\end{proposition}

\begin{proof}
 It suffices to show that for any small 2-category $\categ C$, the endofunctor 
$- \times \categ C$ of $\categ{2-Cats}$ has a right adjoint.
It is given by the functor $\TwoFun(\categ C,-)$ that sends a strict 2-category
$\categ D$ to the 2-category whose
\begin{itemize}
 \itemt objects are 2-functor from $\categ C \to \categ D$;
 \itemt morphisms are 2-functors from $\Mor \times \categ C$  to $\categ D$;
\itemt 2-morphisms  are 2-functors from $\Nat \times \categ C$ to $\categ D$;
\end{itemize}
\end{proof}

\begin{remark}
 This product of strict 2-categories or even of bicategories does not have good homotopical properties. For
 instance, the product $\Mor \times \Mor$ is the strictly
 commutative square instead of the square commutative up to
 a natural isomorphism.
\end{remark}

\begin{definition}
 Given two 2-functors $F,G$ from $\categ C$
 to $\categ D$, a strict natural transformation from $F$ to $G$ is just a morphism from $F$ to $G$ in the category $\TwoFun(\categ C,\categ D)$.
\end{definition}

\begin{proposition}
A strict natural transformation $A$ from $F$ to $G$ (that share the same source $\categ C$ and the same target $\categ D$) is equivalent to the data
of morphisms in $\categ D$
$$
    A(X) : F(X) \to G(X)
$$
for any object $X \in \categ C$ so that
\begin{itemize}
    \itemt for any morphism $f : X \to Y$ in $\categ C$, the following square diagram commutes in $\categ D$
    $$
    \begin{tikzcd}
         F(X) \ar[r,"A(X)"]
         \ar[d,"F(f)"]
         & G(X)
         \ar[d,"G(f)"]
         \\
         F(Y) \ar[r,"A(Y)"']
         & G(Y) ;
    \end{tikzcd}
    $$
    \itemt for any 2-morphism 
    $a : f \to g$ in $\categ C$ where $f,g: X \to Y$, we have
$$
    A(Y) \circ_h F(a) = G(a) \circ_h A(X) .
$$
\end{itemize}
In other words, this is a natural transformation
from the functor $\sk(F)$ to the functor $\sk(G)$
that satisfies the last condition.
\end{proposition}

\begin{proof}
 Straightforward.
\end{proof}

\begin{corollary}
    Two strict natural transformations
    $A,A': F \to G$ are equal if and only if the underlying 
    natural tranformations from $\sk(F)$ to $sk(G)$ are equal.
\end{corollary}

\begin{proof}
    This just follows from the fact that 
    a strict natural transformations from $F$ to $G$ is
    the data of a natural transformation from
    $\sk(F)$ to $sk(G)$ that satisfies an additional condition.
\end{proof}

\begin{definition}
Given two 2-functors $F,G$ from $\categ C$
 to $\categ D$ and two strict natural transformations $A$ and $A'$ from $F$ to $G$, a modification from $A$ to $A'$ is a 2-morphism from $A$ to $A'$ in the 2-category $\TwoFun(\categ C,\categ D)$.
\end{definition}

\begin{remark}
 We will work in a framework strict enough to avoid modifications.
\end{remark}

\begin{definition}
A 2-functor $F : \categ{C} \to \categ D$ is strictly fully faithful if for any objects $X,Y \in \Ob(\categ C)$, the functor
$$
F(X,Y) : \categ C(X,Y) \to \categ D(F(X),F(Y))
$$
is an isomorphism of categories.
\end{definition}

\begin{definition}
A 2-functor $F : \categ{C} \to \categ D$ is strictly essentially surjective if the underlying functor $\sk(F)$ is essentially surjective.
\end{definition}

\begin{definition}
A 2-functor $F : \categ{C} \to \categ D$ is an iso-equivalence if there exists a 2-functor
$G : \categ D \to \categ C$ and
strict natural isomorphisms
\begin{align*}
    \id_{\categ C} &\simeq G \circ F;
    \\
    \id_{\categ D} &\simeq F \circ G.
\end{align*}
\end{definition}

\begin{remark}
In other words
$F$ is an iso-equivalence if it is an equivalence in the 2-category made up of 2-categories, 2-functors and strict natural transformations.
\end{remark}

\begin{proposition}
A 2-functor $F$ is an iso-equivalence if and only if it is strictly fully faithful and strictly essentially surjective.
\end{proposition}

\begin{proof}
 This follows from the same arguments as for categories.
\end{proof}

\subsection{Monoidal context}

We give here the definition of a monoidal context, which could also be called an almost strict symmetric monoidal strict 2-category. Indeed, this is the data of a strict 2-category together with a symmetric monoidal structure which is as strict as symmetric monoidal structures on categories.

Any 2-category is equivalent to a strict 2-category. But this is not true for 3-categories. Likewise we do not expect our notion of monoidal context to encompass all symmetric monoidal 2-categories.

\begin{definition}
 A monoidal context is the data a strict 2-category $\categ C$
 together with
\begin{itemize}
 \itemt a 2-functor $-\otimes - : \categ C \times \categ C \to \categ C$;
 \itemt a 2-functor $\II : \ast \to \categ C$;
  \itemt a strict natural isomorphism call the associator
  $(X \otimes Y) \otimes Z \simeq X \otimes (Y \otimes Z)$
  \itemt a strict natural isomorphism called 
  the commutator
  $$
    X \otimes Y  \simeq Y \otimes X;
  $$
  \itemt two strict natural isomorphisms called the unitors
  $$
    \II \otimes X \simeq X \simeq X \otimes \II 
  $$
\end{itemize}
that makes the skeleton $\sk (\categ C)$ a symmetric monoidal category (for instance they satisfy the pentagon identity and the triangle identity).
\end{definition}

\begin{definition}
    A monoidal context is small if its underlying  strict 2-category
    is small.
\end{definition}

We now define morphisms between monoidal contexts that we call (lax) context functors. This is an almost strict version of (lax) symmetric monoidal 2-functor.

\begin{definition}
A lax context functor between two monoidal contexts $\categ C, \categ D$ is the data of
\begin{itemize}
\itemt a 2-functor $F: \categ C \to \categ D$;
\itemt a strict natural transformation $F(X) \otimes F(Y) \to F(X \otimes Y)$;
\itemt a strict natural transformation $\II \to F(\II)$;
\end{itemize}
that makes the functor between skeletons $\sk (F)$ a lax symmetric monoidal functor.
Let $\mathsf{Context}_{\lax}$ be the category of small
monoidal contexts and lex context functors.
\end{definition}

\begin{definition}
A lax context functor between two monoidal contexts $F : \categ C \to \categ D$ whose structural strict natural transformation
 $$
F(X) \otimes F(Y) \to F(X \otimes Y), \quad \II \to F(\II)
$$
are isomorphisms (resp. identities) is called a strong
context functor (resp. a strict context functor).
\end{definition}

\begin{remark}
 A strict context functor is equivalently a 2-functor that commute with the monoidal context structures. 
\end{remark}

\begin{definition}
Given two lax context functors $F,G : \categ C
\to \categ D$, a monoidal natural transformation
between them is a strict natural transformation
$A : F \to G$ so that the induced natural transformation from $\sk(F)$ to $\sk(G)$ is monoidal. In other words, $A$ commutes with the structural
natural transformations of the two lax context functors.
\end{definition}

\subsection{Mac Lane's coherence theorem}

 Let us consider a monoidal context $\categ C$.
 In particular the skeleton $\mathrm{sk}(\categ C)$ is a
 symmetric monoidal category.

 \begin{definition}\label{definitiontreefunct}
    For any planar tree $t$ (see Appendix \ref{appendixtree}
    for a precise definition) with $n$ leaves
    whose nodes have arity 2 or 0
    one obtains a 2-functor
    $$
\otimes_{t}: \categ C^n \to \categ C
    $$
    as follows:
    \begin{itemize}
        \itemt if $t$ has no node, then $\otimes_t = \id$;
        \itemt if $t$ is the arity 0 corolla, then
        $\otimes_t$ is the 2-functor $\ast \to \categ C$
        corresponding to the object $\II$;
        \itemt if $t$ is the arity 2 corolla, then 
        $\otimes_t  = - \otimes -$;
        \itemt for any larger tree made up of an arity 2-corolla
        that supports a left tree $t_l$ and a right tree $t_r$, then
        $$
        \otimes_{t} = (- \otimes -) \circ (\otimes_{t_l} \times \otimes_{t_r}) ;
        $$
    \end{itemize}
 \end{definition}

 \begin{definition}\label{definitiontreefunct2}
    From any planar tree $t$ with $n$ leaves
    whose nodes have arity 2 or 0 and for
    any permutation $\sigma \in \mathbb{S}_n$,
    one obtains a 2-functor
    $$
\otimes_{t, \sigma}: \categ C^n \to \categ C
    $$
    as follows the composition
        $$
        \categ C^n \xrightarrow{\sigma^\ast} \categ C^n
        \xrightarrow{\otimes_t} \categ C.
        $$
 \end{definition}

 \begin{definition}
    Let $n$ be natural integer.
     The $n$-structural groupoid of the monoidal context $\categ C$
     is the subcategory of $\TwoFun(\categ C^n, \categ C)$
     whose objects are the functors $\otimes_{t, \sigma}$
     for $t$ a tree with $n$ leaves and whose nodes have arity $2$ $0$
     and for $\sigma \in \mathbb{S}_n$; the morphisms are the strict natural
     transformations generated by the following maps
     \begin{enumerate}
         \item  for any tree $t$ and any subtree of the form
         $\treeassleft{}$, let $t'$ be the tree obtained from $t$ by
         replacing the subtree $\treeassleft{}$ by $\treeassright{}$; then the
         associator of $\categ C$ gives us a strict natural isomorphism
         $$
        \otimes_{t, \sigma} \simeq \otimes_{t', \sigma};
         $$
         for any permutation $\sigma$;
         \item similarly as in the previous recipe one can replace
         subtrees of the form $\treeunitleft{}$ (resp. $\treeunitright{}$)
         by a simple edge and use
         the left unitor (resp. right unitor) to obtain
         strict natural isomorphisms;
         \item for any tree $t$ and any binary node $n$, let
         $t_l, t_r$ be the subtrees made up edges above
         respectively the left edge
         of $n$ and the right edge
         of $n$ and let $t'$ be the tree obtrained from $t$ by swapping
         $t_1$ and $t_2$. Let $\rho$ be the permutation so that
         $$
\begin{cases}
    \rho(i) = i \text{ if }i \leq a;
    \\
    \rho(a+j) = a + l_2 +j \text{ if }1 \leq  j  \leq l_1;
    \\
    \rho(a+l_1+j) = a +j \text{ if }1 \leq  j  \leq l_2;
    \\
    \rho(i) = i \text{ if }i > a+l_1+l_2.
\end{cases}
         $$
         The commutator of $\categ C$ gives us a strict natural isomorphism
         $$
        \otimes_{t, \sigma} \simeq \otimes_{t', \rho \circ \sigma};
         $$
         for any permutation $\sigma$.
     \end{enumerate}
 \end{definition}

 \begin{theorem}[Mac Lane's coherence theorem, \cite{maclane1963natural}]
    For any natural integer $n$, the $n$-structural groupoid
    of $\categ C$ is contractible.
 \end{theorem}

 \begin{proof}
     Actually Mac Lane's result deals with monoidal categories
     instead of monoidal contexts. Its extension to monoidal
     contexts is just a consequence of the fact that the $n$-structural
     groupoid of $\categ C$ is the same as that of $\sk(\categ C)$.

 \end{proof}

 \begin{remark}\cite{Kelly74}\label{propositionunitcomalg}
    In particular, in a monoidal context,
    the left unitor $\II \otimes \II \to \II$ applied to the unit
    $\II$ equals its right unitor.
    This gives $\II$ the structure of a strict unital commutative algebra
    whose unit is just the identity $\II = \II$.
   \end{remark}


\subsection{Multiple tensors}

Let $\categ C$ be a monoidal context.
One can define multiple tensors as follows.

\begin{definition}
    Let $n > 2$ be a natural integer. The 2-functor
    $$
    \otimes_n : (X_1,  \ldots, X_n) \in \categ C^n \to 
    X_1 \otimes \cdots\otimes X_n \in \categ C
    $$
    is the colimit of the diagram of 2-functors
    whose objects are the 2-functors $\otimes_t$ where $t$ is a binary tree
    with $n$ leaves and whose morphisms are those of the $n$-structural
    groupoid.
    
    Then, for any planar tree $t$ with $m$ leaves and whose nodes
    have arity $0,2$ or more than $2$, and any
    permutation $\sigma \in  \mathbb S_m$, one can define
    the 2-functors
    $$
    \otimes_t, \otimes_{t, \sigma}: \categ C^m \to \categ C
    $$
    in the same way as in Definition \ref{definitiontreefunct}
    and Definition \ref{definitiontreefunct2}.
\end{definition}

\begin{definition}
    Let $n$ be natural integer.
     The extended $n$-structural groupoid of the monoidal context $\categ C$
     is the subcategory of $\TwoFun(\categ C^n, \categ C)$
     whose objects are the functors $\otimes_{t, \sigma}$
     for $t$ a tree with $n$ leaves and whose nodes have arity $0$, $2$
     or more than $2$ and for $\sigma \in \mathbb{S}_n$; the morphisms
     are the strict natural
     transformations generated by the morphisms of the
     (non extended) $n$-structural groupoid
     and, for any permutation $\sigma$,
     for any tree $t$ and any binary subtree $s$ with $m$-leaves
     the strict natural isomorphism
     $$
    \otimes_{t, \sigma} \simeq \otimes_{t/s, \sigma}
     $$
     induced by the canonical map $\otimes_s \to \otimes_m$,
     where $t/s$ is the tree obtained from $t$ by contracting $s$
     into a $m$-corolla.
\end{definition}

\begin{corollary}
    Let $n$ be a natural integer. The extended $n$-structural
    groupoid of $\categ C$ is contractible.
\end{corollary}

\begin{proof}
    Straightforward.
\end{proof}

Finally for any lax context functor $F: \categ C \to \categ D$,
the lax structure induces
canonical strict natural transformations
    $$
    \otimes_{t, \sigma} \circ F^n \to F \circ  \otimes_{t, \sigma}
    $$
    for any tree with $n$ leaves and whose nodes have arity different
    from $1$ and for any permutation $\sigma \in \mathbb{S}_n$.
    Moreover, for any canonical isomorphism $\otimes_{t, \sigma}
    \simeq \otimes_{t', \sigma'}$,
    the following square diagram commutes
    $$
\begin{tikzcd}
    \otimes_{t, \sigma} \circ F^n
    \ar[r] \ar[d, "\simeq"]
    & F \circ  \otimes_{t, \sigma}
    \ar[d, "\simeq"]
    \\
    \otimes_{t', \sigma'} \circ F^n
    \ar[r]
    & F \circ  \otimes_{t', \sigma'} .
\end{tikzcd}
    $$

\subsection{Cartesian monoidal structures}

\begin{definition}
Let $\categ C$ be a strict 2-category and $X,Y$ be two objects.
The strict product of $X$ and $Y$ if it exists is an object $X,\times Y$
(defined up to a unique isomorphism)
 equipped with two morphisms
\begin{align*}
    X \times Y &\to X;
    \\
    X \times Y &\to Y;
\end{align*}
so that for any object $Z$ the functor
$$
    \categ C(Z, X \times Y) \to  \categ C(Z,X) \times \categ C(Z,Y)
$$
is an isomorphism of categories.
\end{definition}

\begin{definition}
A strict final element in a strict 2-category $\categ C$ is an object $\ast$ defined up to a unique isomorphism so that for any object $Z$ the functor
$$
    \categ C(Z, \ast) \to  \ast
$$
is an isomorphism of categories.
\end{definition}

\begin{definition}
A strict 2-category is said to have strict finite  products if it has strict products of any pair of objects and if it as a strict final element.
\end{definition}

\begin{proposition}
Suppose that the strict 2-category $\categ C$ has strict finite products. Then $\categ C$ gets from  the product and the final element the structure of a monoidal context.
\end{proposition}

\begin{proof}
 The associator is the canonical natural transformation
 $$
 (X \times Y) \times Z \simeq X \times (Y \times Z)
 $$
 and the unitors are the canonical natural transformations
 $$
    X \times \ast  \simeq X \simeq \ast \times X .
 $$
\end{proof}

\begin{definition}
 A cartesian monoidal context is a strict 
 2-category that has strict finite products and that is equipped with the induced structure of a  monoidal context.
\end{definition}

\begin{definition}
Let $\categ C, \categ D$ be two strict 2-categories with strict finite products. A 2-functor $F:\categ C \to \categ D$ preserve strict products if the canonical morphisms
\begin{align*}
    F(X\times Y) \to F(X) \times F(Y) ;
    \\
    F(\ast) \to \ast;
\end{align*}
are isomorphisms.
\end{definition}

\begin{proposition}
Let $\categ C, \categ D$ be two cartesian monoidal contexts.
A 2-functor $F : \categ C \to \categ D$ that preserves strict finite products has the canonical structure of a context functor.
\end{proposition}

\begin{proof}
 This structure is given by the strict natural
 transformation $F(X \times Y) \simeq F(X) \times F(Y)$.
\end{proof}

\subsection{Opposite structures}

Let $(\categ C, \otimes, \II)$ be a monoidal context. Since in some sense this monoidal context is a 3-categorical structure, there are three ways to inverse structures.
\begin{itemize}
    \itemt one can take opposite morphisms;
    \itemt or take opposite 2-morphisms;
    \itemt or take the opposite monoidal structure.
\end{itemize}

\begin{definition}
Let $\categ C^{\op}$ be the monoidal context with the same object as $\categ C$ and so that 
$$
\categ C^{\op} (X,Y) = \categ C (Y,X)
$$
for any two objects $X,Y$.
\end{definition}

\begin{definition}
Let $\categ C^{co}$ be the monoidal context with the same object as $\categ C$ and so that 
$$
\categ C^{co} (X,Y) = \categ C (X,Y)^{op}
$$
for any two objects $X,Y$.
\end{definition}

\begin{definition}
Let $\categ C^{\tr} = (\categ C, \otimes^{tr}, \II)$ (where tr stands for "transposition") be the monoidal context with the same underlying 2-category $\categ C$ but whose monoidal structure is defined by
$$
 X \otimes^{tr} Y = Y \otimes X
$$
for any two objects $X,Y$.
The associator, unitors and commutator are given by that of
the monoidal context $(\categ C, \otimes,\II)$.
\end{definition}

It is possible to apply more than one of these transformation to obtain $\categ C^{coop},\categ C^{optr}, \categ C^{cotr}$ and $\categ C^{cooptr}$.


\subsection{Some monoidal contexts}

The main example of a monoidal context
is that of categories.

\begin{definition}
    Let $\categ{Cats}$ be the strict
    2-category of small categories, functors and
    natural transformations. It has strict finite products
    and is thus a cartesian
    monoidal context.
\end{definition}

Here are two other examples of monoidal contexts.
\begin{itemize}
    \itemt The 2-category of small strict 2-categories,
    2-functors and
    strict natural transformations form a cartesian monoidal
    context.
    \itemt The 2-category of small monoidal contexts,
    lax 2-functors and
    monoidal natural transformations form a cartesian monoidal
    context.
\end{itemize}
We will not deal directly with these two cartesian
monoidal contexts but with their underlying cartesian
symmetric monoidal categories $\categ{2-Cats}$ and
$\categ{Contexts}_\lax$.

\begin{definition}
Let $\categ{Functors}$ be the strict 2-category $\TwoFun(\Mor, \categ{Cats})$, that is
\begin{itemize}
    \itemt its objects are functors $F : \categ C \to \categ D$ between small categories;
    \itemt its morphisms from $F_1$ to $F_2$ are pairs of functors $(S,T)$ so that the following square diagram of categories is strictly commutative
    \[
    \begin{tikzcd}
         \categ C_1
         \ar[r,"S"] \ar[d,"F_1"]
         & \categ C_2
         \ar[d,"F_2"]
         \\
         \categ D_1
         \ar[r,swap, "T"]
         & \categ D_2;
    \end{tikzcd}
    \]
    \itemt its 2-morphisms from $(S_1,T_1)$ to
    $(S_2,T_2)$ are natural transformations $A_S : S_1 \to S_2$ and $A_T: T_1 \to T_2$ so that
    \begin{align*}
        A_T \circ_h F_1 = F_2 \circ_h A_S .
    \end{align*}
\end{itemize}
\end{definition}

\begin{proposition}
The 2-category $\categ{Functors}$ form a cartesian monoidal context.
Moreover, the two 2-functors to categories (target and source) preserve strict finite products.
\end{proposition}

\begin{proof}
 Straightforward.
\end{proof}

One can also work with subuniverses ($\mathcal U$
and $\mathcal V$).

\begin{definition}
    A $\mathcal U$-strict 2-category $\categ C$
    is a strict 2-category whose set of objects
    is $\mathcal U$-large and so that
    for any two objects $x,y$, the category
    $\categ C(x,y)$ is $\mathcal{U}$-small.
    
    A $\mathcal U$-strict 2-category is $\mathcal U$-small
    if its set of objects is $\mathcal U$-small.
    
    One defines similarly $\mathcal V$-strict 2-categories
    and $\mathcal V$-small strict 2-categories.
\end{definition}

\begin{definition}
    Let $\categ{Cats}_{\mathcal U}$ be the
    $\mathcal V$-strict 2-category (in
    particular it is a $\mathcal W$-small)
    whose
    \begin{itemize}
        \itemt objects are $\mathcal U$-categories
        (that are in particular $\mathcal V$-small categories);
        \itemt morphisms are functors;
        \itemt 2-morphisms are natural transformations.
    \end{itemize}
    This is actually the full sub-2-category
    of the strict 2-category of small
    categories spanned by $\mathcal U$-categories.
\end{definition}

\begin{proposition}
    The small 2-category $\categ{Cats}_{\mathcal U}$
    has strict finite products and is hence a small
    cartesian monoidal context.
\end{proposition}

\begin{proof}
    It suffices to notice that the product of two
    $\mathcal U$-categories is again a $\mathcal U$-category.
    This follows from the fact that for any two elements $x,y$ of
    $\mathcal U$, the pair $(x,y)$ still belongs to $\mathcal U$.
\end{proof}


\section{Categorical operads acting on a monoidal context}

To describe algebraic structures in a monoidal context $\categ C$, one general way is
to consider actions of another monoidal context on $\categ C$, that is context functors from some monoidal context $\categ A$ to $\categ C$. We choose here to deal with the restricted framework of action of operads enriched in categories that we call categorical operads.

\subsection{Operads}

\begin{definition}
    A categorical collection $X$ is the data of a set of colours $O=\Ob(X)$
    (also called objects)
    and a functor
    $$
X : \coprod_{n \geq 0} O^n \times O \to \Cats,
    $$
    that is equivalently the data of small categories
    $$
X(c_1, \ldots, c_n;c)
    $$
    for any natural integer $n$ and any elements
    $(c_1, \ldots, c_n,c) \in O^{n+1}$.
    A morphism $f$ of categorical collections from $X$
    to $Y$ is the data of a function
    $$
f(-) : \Ob(X)\to \Ob(Y)
    $$
    together with functors 
    $$
    f(c_1, \ldots, c_n;c): X(c_1, \ldots, c_n;c) \to
    Y(f(c_1), \ldots, f(c_n);f(c)).
    $$
    This defines the category $\categ{Collections}$
    of categorical collections.
\end{definition}

\begin{definition}
    Let $O$ be a set of objects. A $O$-coloured 
    tree is the data of a tree $t$ and a function
    $$
\mathrm{col}_t : \mathrm{edges}(t) \to O.
    $$
\end{definition}

\begin{definition}
Given a categorical collection $X$,
a $\Ob(X)$ coloured tree $t$ and a node
$v$ of $t$, we denote
$$
X(v) = X(c_1, \ldots, c_n; c)
$$
where $(c_i)_{i=1}^n$ are the colours of the
ordered inputs of $v$ and
$c$ is the colour of its output.
Then, we denote
$$
X(t) = \prod_{v \in \mathrm{nodes}(t)} X(v).
$$
\end{definition}

\begin{definition}
    For any categorical collection $X$, let
    $\mathrm{T}_{\pl}(X)$ be the categorical collection
    with the same set of objects as $X$ and so that
    $$
\mathrm{T}_{\pl}(X)(c_1, \ldots, c_n; c) = \coprod_t X(t)
    $$
    where the coproduct is taken over the
    small set of equivalence classes of $\Ob(X)$-coloured trees
    $t$ whose root is coloured by $c$ and whose leaves are
    coloured by $(c_1, \ldots, c_n)$ (we respect the order on leaves).
    This defines an endofunctor of
    the category of categorical collections
    that has the canonical structure of a monad whose
    \begin{itemize}
        \itemt unit is given by the inclusion of corollas into
        all trees;
        \itemt product is given by forgetting partitions in
        partitioned trees. 
    \end{itemize}
\end{definition}

\begin{definition}
Categorical planar operads $\catoperad P$ are algebras over the
planar tree monad $\mathrm{T}_{\pl}$ in categorical collections.
Equivalently, a categorical planar operad $\catoperad P$
is a categorical collection equipped with compositions
\begin{align*}
    \catoperad P(\underline c ; c)
\times \catoperad P(\underline d ; c_i)
&\xrightarrow{\triangleleft_i}
\catoperad P(\underline c
    \triangleleft_i \underline d ; c) \\
    1_c &\in \catoperad P(c;c)
\end{align*}
that satisfy associativity and unitality conditions and where 
\begin{align*}
    \underline c &=(c_1, \ldots,c_n);
    \\
    \underline d &=(d_1, \ldots,d_m);
    \\
    \underline c
    \triangleleft_i \underline d
    &=
    (c_1, \ldots,c_{i-1},d_1, \ldots,d_m,c_{i+1},\ldots, c_n).
\end{align*}
A morphism of categorical planar operads from $\catoperad P$ to $\catoperad Q$
is a morphism of categorical collections $f$
that commutes strictly with the operadic compositions and units.
We denote $\Operad_{\cat, \pl}$ the category of
categorical planar operads.
\end{definition}

\begin{definition}
    Let $\catoperad P$ be a categorical planar operad. Then
    we denote $\catoperad P_{\mathbb{S}}$ the
    categorical planar operad with the same objects as $\catoperad P$
    and so that
    $$
    \catoperad P_{\mathbb{S}}(\underline c; c)
    =
    \coprod_{\sigma \in \mathbb S_n}
    \catoperad P(\sigma^\ast(\underline c); c) \times \{\sigma\}
    $$
    where
\begin{align*}
    \underline c &=(c_1, \ldots,c_n);
    \\
    \sigma^\ast(\underline c) &=(c_{\sigma^{-1}(1)}, \ldots,c_{\sigma^{-1}(n)}).
\end{align*}
Moreover, the composition
$$
\catoperad P_{\mathbb{S}}(\underline c; c)
 \times \catoperad P_{\mathbb{S}}(\underline d; c_i)
\xrightarrow{\triangleleft_i}
 \catoperad P_{\mathbb{S}}(\underline c \triangleleft_i \underline{d}; c)
$$
in $\catoperad P_{\mathbb{S}}$
is given by the maps
$$
\catoperad P(\sigma^\ast(\underline c);c) \times \{\sigma\}
\times
\catoperad P(\mu^\ast(\underline d);c_i) \times \{\mu\}
\xrightarrow{\triangleleft_{\sigma(i)}}
\catoperad P((\sigma^\ast(\underline c)\triangleleft_{\sigma(i)}
\mu^\ast(\underline d);c)
\times \{\rho\}
$$
where $\rho \in \mathbb{S}_{n+m-1}$ is the permutation given
by
$$
\begin{tikzcd}
    \{1, \ldots , n+m-1\}
    \ar[d, equal]
    \\
    \{1, \ldots, i-1\}
    \sqcup \{i+m, \ldots, n+m-1\}
    \sqcup \{i, \ldots, i+ m-1\}
    \ar[d, "\id{}\sqcup (a\mapsto a-m+1)\sqcup (b\mapsto b-i+1)"]
    \\
    \{1, \ldots, i-1\}
    \sqcup \{i+1, \ldots, n\}
    \sqcup \{1, \ldots, m\}
    \ar[d, equal]
    \\
    \left(\{1, \ldots, n\} - \{i\} \right)
    \sqcup \{1, \ldots, m\}
    \ar[d, "\sigma \sqcup \mu"]
    \\
    \left(\{1, \ldots, n\} - \{\sigma(i)\} \right)
    \sqcup \{1, \ldots, m\}
    \ar[d, equal]
    \\
    \{1, \ldots, \sigma(i)-1\}
    \sqcup \{\sigma(i)+1, \ldots, n\}
    \sqcup \{1, \ldots, m\}
    \ar[d, "\id{}\sqcup (a\mapsto a+m-1)\sqcup (b\mapsto b+\sigma(i)-1)"]
    \\
    \{1, \ldots, \sigma(i)-1\}
    \sqcup \{\sigma(i)+m, \ldots, n+m-1\}
    \sqcup \{\sigma(i), \ldots, \sigma(i)+ m-1\}
    \ar[d, equal]
    \\
    \{1, \ldots , n+m-1\}
\end{tikzcd}
$$
This defines an endofunctor of the category of
categorical planar operads
that has the canonical structure of a monad whose
    \begin{itemize}
        \itemt unit is given by the inclusion of identities
        into
        all permutations;
        \itemt product is given by the composition
        of permutations. 
    \end{itemize}
\end{definition}

\begin{definition}
A categorical operad $\catoperad P$ is an algebra over the monad $\__{\mathbb S}$
in the category of categorical planar operads.
In other words, this is the data of
categorical planar operad together with isomorphisms
$$
    \sigma^\ast :  \catoperad P(\underline c;c)
    \to \catoperad P(\underline c^{\sigma};c)
$$
for any permutation $\sigma \in \Sigma_n$, where
\begin{align*}
    \underline c &=(c_1, \ldots,c_n);
    \\
    \underline c^\sigma &=(c_{\sigma(1)}, \ldots,c_{\sigma(n)}) = (\sigma^{-1})^\ast(\underline{c});
\end{align*}
so that $(\mu \circ \sigma)^\ast = \sigma^\ast \circ \mu^\ast$ and that satisfy coherence conditions with respect to the operadic composition.
We denote $\Operad_{\mathrm{cat}}$ the category of
categorical operads.
\end{definition}

\begin{remark}
As for strict 2-categories,
categorical operads are actually organised
into a 3-category. It is even
enriched in 2-operads (represented for instance by categorical operads).
We will see later a lax version of this enrichment.
\end{remark}

One has a sequence of adjunctions
$$
\begin{tikzcd}
    \categ{Collections}
    \ar[rr, shift left, "\mathrm{T}_\pl"]
    && \categ{Operad}_{\mathrm{cat}, \pl}
    \ar[rr, shift left, "\__{\mathbb{S}}"] \ar[ll, shift left]
    && \categ{Operad}_{\mathrm{cat}} .
    \ar[ll, shift left]
\end{tikzcd}
$$

\begin{proposition}
    The composite adjunction is monadic. We denote $\mathrm{T} = \mathrm{T}_\pl(-)_{\mathbb S}$
    the resulting monad.
\end{proposition}

\begin{proof}
    For instance, one can use the monadicity theorem and prove
    that the forgetful functor creates coequalisers of pairs of morphisms
    of categorical
    operads $f,g :\catoperad P \to \catoperad P'$
    that are split in the category of categorical collections.
\end{proof}

\begin{definition}
    Let $\categ C$ be a small monoidal context. Then
    let $\catEnd(\categ C)$ be the categorical operad
    whose set of objects is that of $\categ C$
    and so that
    $$
    \catEnd(\categ C)(x_1,\ldots, x_n;x) 
    = \categ C(x_1 \otimes \cdots \otimes x_n , x) .
    $$
    This defines a functor
    $$
    \catEnd(-) : \categ{Context}_\lax \to \Operad_{\mathrm{cat}}.
    $$
\end{definition}

\begin{proposition}
    The functor $\catEnd(-)$ is fully faithful and
    preserves products.
\end{proposition}

\begin{proof}
    Straightforward.
\end{proof}

Actually, the essential image of the functor 
$\catEnd(-)$ is the full subcategory spanned by 
operads $\catoperad P$ so that the morphism
$\catoperad P \to \catoperad{uCom}$ is some kind of
strict Grothendieck opfibration.

\begin{definition}
    Let $\catoperad P$ be a categorical operad. Then, let $\catoperad P^{co}$
    be the categorical operad with the same objects and so that
    $$
    \catoperad P^{co}(c_1, \ldots, c_n;c) = \catoperad P(c_1, \ldots, c_n;c)^\op
    $$
    for any colours $c_1, \ldots, c_n,c$. The compositions, units and actions
    of symmetric groups are defined as in $\catoperad P$. This co construction
    induces an involutive endofunctor
    $$
    -^{co}: \Operad_\cat \to \Operad_\cat.
    $$
\end{definition}


\subsection{Algebras}

Let $\catoperad P$ be a categorical operad and let $\categ C$ be a
small monoidal context.

\begin{definition}
A $\catoperad P$-algebra $A$ in $\categ C$
is the data of a 
morphism of categorical operads
$$
A : \catoperad P \to \catEnd(\categ C) .
$$
\end{definition}

More concretely, such an algebra is the data of an object
$A_c \in \categ C$ for any colour $c \in \Ob(\catoperad P)$ together with functors
\begin{align*}
    \catoperad P (c_1, \cdots c_n ; c)
 &\to \categ C (A_{c_1} \otimes \cdots \otimes A_{c_n} , A_c)
 \\
 p &\mapsto A(p)
\end{align*}
that commute with the units, compositions and actions of the
symmetric groups.

\begin{definition}\label{def :laxmorphis}
Let $A$ and $B$ be two $\catoperad P$-algebras in $\categ C$.
A lax $\catoperad P$-morphism from $A$ to $B$ is the data of morphisms in $\categ C$
$$
f_c \in \categ C(A_c, B_c), \quad c \in \Ob(\catoperad P)
$$
and natural transformations
$$
\begin{tikzcd}
    \catoperad P(c_1, \ldots, c_n;c)
    \ar[r, "B"] \ar[d, "A"']
    & \categ C(B_{c_1}\otimes \cdots \otimes B_{c_n}; B(c))
    \ar[d, "- \triangleleft f^{\otimes n}"]
    \ar[ld, Rightarrow]
    \\
    \categ C(A_{c_1}\otimes \cdots \otimes A_{c_n}; A(c))
    \ar[r, "f \triangleleft -"']
    & \categ C(A_{c_1}\otimes \cdots \otimes A_{c_n}; B(c))
\end{tikzcd}
$$
for any colours $c_1, \ldots, c_n, c$ of $\catoperad P$,
where $f^{\otimes n}$ and $f$ actually stands for
$f_{c_1} \otimes \cdots \otimes f_{c_n}$ (or $\id_{\II}$ if $n=0$)
and $f_c$.
We also require that these natural transformations behave coherently
with respect to the units, compositions and actions of the
symmetric groups of the operads $\catoperad P$ and $\catEnd(\categ C)$
in the sense that the following diagrams commute
$$
\begin{tikzcd}
 B(1_c) \circ f
 \ar[r] \ar[d, equal]
 & f \circ A(1_c)
 \ar[d]
 \\
 f
 \ar[r, equal]
 & f
\end{tikzcd}
\quad
\begin{tikzcd}
    B(p^{\sigma}) \circ f^{\otimes n}
    \ar[r] \ar[d, equal]
    & f \circ A(p^{\sigma})
    \ar[d, equal]
    \\
    B(p) \circ f^{\otimes n} \circ \sigma^\ast
    \ar[r]
    & f \circ A(p) \circ \sigma^\ast .
\end{tikzcd}
$$
$$
\begin{tikzcd}
     B(p) \circ (\id^{\otimes i-1} \otimes B(p') \otimes \id^{\otimes n-i}) \circ f^{n+m-1}
     \ar[r] \ar[dd, equal]
     & B(p) \circ f^{\otimes n} \circ  (\id^{\otimes i-1} \otimes A(p') \otimes \id^{\otimes n-i})
     \ar[d]
     \\
     &
     f \circ A(p) \circ  (\id^{\otimes i-1} \otimes A(p') \otimes \id^{\otimes n-i})
     \ar[d, equal]
     \\
     B(p \triangleleft_i p') \circ f^{\otimes n+m-1}
     \ar[r]
     &
     f \circ
     A(p \triangleleft_i p');
\end{tikzcd}
$$
for any operations $p \in \catoperad P(c_1, \ldots, c_n;c)$
and $p' \in \catoperad P(c'_1, \ldots, c'_m;c_i)$
and any permutation $\sigma \in \mathbb S_n$.
\end{definition}

\begin{remark}\label{remarkbvtensorlax}
    The lax morphisms are actually algebras over a categorical
    operad $\catoperad P \otimes_{BV, lax} [1]$ which a "lax variation"
    of the Boardman--Vogt tensor product of $\catoperad P$ with the
    poset $[1]$.
\end{remark}

\begin{definition}
A strong $\catoperad P$-morphism (resp. strict $\catoperad P$-morphism)
is a lax $\catoperad P$-morphism $f$ so that the structural natural transformations
are invertible (resp. are identities).
\end{definition}

\begin{definition}
Let $f,g: A \to B$ be two lax $\catoperad P$-morphisms.
A $\catoperad P$ 2-morphism $a$ from $f$ to $g$ 
is the data of 2-morphisms in $\categ C$
$$
     a_c : f_c \to g_c
$$
for any colour $c \in \Ob(\catoperad P)$
so that the following diagram commutes
$$
\begin{tikzcd}
     B(p) \circ f^{\otimes n}
     \ar[r] \ar[d]
     & f \circ A(p)
     \ar[d]
     \\
     B(p) \circ g^{\otimes n}
     \ar[r]
     & g \circ A(p)
\end{tikzcd}
$$
for any element $p$ of $\catoperad P(c_1, \ldots, c_n ; c)$.
\end{definition}

\begin{definition}
Let $\tcatalg{\categ C}{\catoperad P}_{\lax}$ be the strict 2-category
whose objects are $\catoperad P$-algebras in $\categ C$, whose morphisms
are lax $\catoperad P$-morphisms, and whose 2-morphisms are
$\catoperad P$-2-morphisms.
The composition of two lax $\catoperad P$-morphisms $f : A \to A'$ and $g : A' \to A''$ is given by the composition of morphisms in $\categ C$
$$
(g \circ f)_c = g_c \circ f_c
     $$
     and the 2-morphism
     $$
      A''(p) \circ (g \circ f)^{\otimes n}
      =  A''(p) \circ g^{\otimes n} \circ f^{\otimes n}
      \to g \circ A'(p) \circ f^{\otimes n}
      \to g \circ f \circ A(p) .
     $$
     The vertical composition and the horizontal composition
     of $\catoperad P$ 2-morphisms are just given by the vertical composition
     and the horizontal composition of 2-morphisms in $\categ C$.
     Besides, one defines similarly
     the strict 2-categories 
    $\tcatalg{\categ C}{\catoperad P}_\strong$,
     and  $\tcatalg{\categ C}{\catoperad P}_\strict$
     replacing mutatis mutandix lax $\catoperad P$-morphisms
     by respectively strong ones and strict ones.
    \end{definition}


\subsection{The monoidal context of algebras}

Let $\categ C$ be a monoidal context and let $\catoperad P$ be a categorical
operad.

\begin{proposition}\label{lemmalaxmonoidalcontextoperad}
    The strict 2-category $\tcatalg{\categ C}{\catoperad P}_{\mathrm{lax}}$ has
    the structure of a monoidal context so that the forgetful 2-functor
    $$
    \tcatalg{\categ C}{\catoperad P}_{\mathrm{lax}} \to
    \categ{C}^{\Ob(\catoperad P)}
    $$
    is a strict context functor.
   \end{proposition}
   
   \begin{proof}
The monoidal unit is given by the
following $\catoperad P$-algebra obtained from the monoidal
unit of $\categ C$
$$
\catoperad P \to \catoperad{uCom} \xrightarrow{\II_{\categ C}}
\catEnd(\categ C).
$$
The tensor product $-\otimes -$ is the bifunctor that sends a pair of
$\catoperad P$-algebra $A, A'$ to the algebra $A \otimes A'$ given by
$$
\catoperad P \to \catoperad P \times \catoperad P \xrightarrow{A \times A'}
\catEnd(\categ C) \times \catEnd(\categ C)
\xrightarrow{\simeq} \catEnd(\categ C \times \categ C)
\xrightarrow{\catEnd(- \otimes -)}
\catEnd(\categ C \times \categ C)
$$
and that sends a pair of lax morphisms $f : A \to B$ and $g: A' \to B'$
to the morphism $f \otimes g$ given by the maps
$$
f_c \otimes g_c : A_c \otimes B_c \to A'_c \otimes B'_c
$$
and the lax structure
$$
   \begin{tikzcd}
        (A \otimes B)(p) \circ (f \otimes g)^{\otimes n}
        \ar[d, equal]
        \\
        \left((A(p) \circ  f^{\otimes n}) \otimes (B(p) \circ  g^{\otimes n})\right) \circ \sigma
        \ar[d]
        \\
       \left((f \circ A'(p)) \otimes (g \circ B'(p))\right) \circ \sigma
       \ar[d, equal]
       \\
       (f \otimes g)  \circ (A\otimes B)(p) ,
   \end{tikzcd}
   $$
   where $\sigma$ is the isomorphism
   $$
   A_{c_1} \otimes B_{c_1} \otimes \cdots \otimes A_{c_n} \otimes B_{c_n} \simeq
   A_{c_1}  \otimes \cdots \otimes A_{c_n} \otimes B_{c_1} \otimes \cdots \otimes B_{c_n} .
   $$
   For any two 2-morphisms $M : f \to f'$ and $N: g \to g'$
   in $\tcatalg{\categ C}{\catoperad P}_{\mathrm{lax}}$,
   one can check that the 2-morphisms
   $(M_c \otimes N_c)_{c \in \Ob(P)}$ in $\categ C$ 
   form a 2-morphism in $\tcatalg{\categ C}{\catoperad P}_{\mathrm{lax}}$.
   
   Now, for any three algebras $A,A',A''$, the associators, unitors
   and commutators in $\categ C$
   $$
   (A_c \otimes A'_c) \otimes A''_c \simeq A_c \otimes (A'_c \otimes A''_c);
    \quad A_c \otimes \II \simeq A_c \simeq \II \otimes A_c;
   \quad A_c \otimes A'_c \simeq A'_c \otimes A_c
   $$
   form natural strict morphisms of $\catoperad P$-algebras.
   These natural transformations satisfy the axioms of a
   monoidal context since their images into $\categ C^{\Ob(\catoperad P)}$ do
   and since the forgeful functor 
   $$
   \sk(\tcatalg{\categ C}{\catoperad P}_\strict) \to \mathrm{sk}(\categ C^{\Ob(\catoperad P)})
   $$
  is faithful.
   \end{proof}

   \begin{corollary}
    The strict 2-categories
    $\tcatalg{\categ C}{\catoperad P}_\strong$ and
    $\tcatalg{\categ C}{\catoperad P}_\strict$
    have canonical structures of monoidal contexts and the 2-functors
    $$
        \tcatalg{\categ C}{\catoperad P}_{\strict}
        \to
         \tcatalg{\categ C}{\catoperad P}_\strong
         \to
          \tcatalg{\categ C}{\catoperad P}_\lax
         \to
         \categ C^{\Ob(\catoperad P)}
    $$
    are strict context functors.
   \end{corollary}
   

\subsection{Naturality of algebras}

\begin{proposition}
The construction that sends a categorical operad $\catoperad P$
and a monoidal context $\categ C$
to the monoidal context $\tcatalg{\categ C}{\catoperad P}_{\lax}$
induces a functor
$$
\tcatalg{-}{-}_{\lax} : \Operad_{\mathrm{cat}}^{op} \times \mathsf{Context}_{\lax}
\to \mathsf{Context}_{\lax}.
$$
that has the canonical structure of a lax monoidal functor.
One obtains similar functors lax monoidal functors
by replacing lax morphisms
by strong ones or strict ones.
Then, one has canonical monoidal natural transformations of
lax monoidal functors
$$
        \tcatalg{-}{-}_{\strict}
        \to
         \tcatalg{-}{-}_\strong
\to
          \tcatalg{-}{-}_\lax
\to
         ((\catoperad P, \categ C) \mapsto \categ C^{\Ob(\catoperad P)}).
    $$
\end{proposition}

\begin{proof}
Given a morphism of operads $f : \catoperad Q \to \catoperad P$
and a lax context functor
$F : \categ C \to \categ D$, one gets a lax context functor 
$$
\tcatalg{\categ C}{\catoperad P}_{\lax}
\to \tcatalg{\categ D}{\catoperad Q}_{\lax}
$$
that sends an object,
    that is a morphism of operads
    $A :\catoperad P \to \catEnd(\categ C)$ to the morphism
    $$
    \catEnd(F) \circ A \circ f .
    $$
The image of lax morphisms and 2-morphisms are
defined accordingly using $f$ and $F$.
Moreover, the structure of a lax context functor
is given by the that of $F$.

Finally, the lax monoidal
structure on the functor $\tcatalg{-}{-}_\lax$
is given by the natural strict context functor
$$
\tcatalg{\categ C}{\catoperad P}_\lax \times \tcatalg{\categ D}{\catoperad Q}_\lax
\to \tcatalg{\categ C \times \categ D}{\catoperad P \times \catoperad D}_\lax
$$
that sends the pair $(A, B)$ of a $\catoperad P$-algebra $A$ and a
$\catoperad Q$-algebra $B$ to the family of objects 
$$
(A,B)_{c,c'} = (A_c, B_{c'}) \in \categ C \times \categ D
$$
indexed by $\Ob(\catoperad P) \times \Ob(\catoperad Q)$ that are
equipped with the structure of a $\catoperad P \times \catoperad Q$-algebra
given by
$$
\begin{tikzcd}
    \catoperad P(c_1, \ldots, c_n;c) \times \catoperad Q(c'_1, \ldots, c'_n;c')
    \ar[d, "A \times B"]
    \\
    \categ C(A_{c_1} \otimes \cdots \otimes A_{c_1}, A_{c})
    \times 
    \categ D(B_{c'_1} \otimes \cdots \otimes B_{c'_1}, B_{c'})
    \ar[d, equal]
    \\
    \categ C \times \categ D
    ((A_{c_1}, B_{c'_1})\otimes \cdots \otimes (A_{c_n}, B_{c'_n}), (A_{c}, B_{c'})).
\end{tikzcd}
$$
\end{proof}

\begin{remark}
Again, by only considering a functor (and not a 3-functor here), we are discarding a lot of higher information.
\end{remark}

\begin{corollary}
    Let $F: \categ C \to \categ D$ be a strong (resp. strict) context
    functor and let $f: \catoperad P \to \catoperad Q$ be a morphism of operads.
    Then, the resulting 2-functor
    $$
\tcatalg{\categ C}{\catoperad Q} \to \tcatalg{\categ D}{\catoperad P}
    $$
    is a strong (resp. strict) context functor.
\end{corollary}

\begin{proposition}\label{prop : pullbackalg}
Let $\catoperad P$ be a categorical operad
and let $F: \categ C \to \categ D$ be a strictly fully faithful
strong context functor. Then, the square
$$
\begin{tikzcd}
     \tcatalg{\categ{C}}{\catoperad P}_{\lax}
     \ar[r] \ar[d]
     & \tcatalg{\categ{D}}{\catoperad P}_{\lax}
     \ar[d]
     \\
 \categ C^{\Ob(\catoperad P)}
 \ar[r]
 & \categ D^{\Ob(\catoperad P)} 
\end{tikzcd}
$$
is a pullback square in the category of strict 2-categories.
In particular, the 2-functor $\tcatalg{\categ{C}}{\catoperad P}_{\lax}
     \to \tcatalg{\categ{D}}{\catoperad P}_{\lax}$
is strictly fully faithful.
\end{proposition}

\begin{proof}
 It amounts to prove that the underlying square of sets given by objects
 is a pullback (which means that a $\catoperad P$-algebra in $\categ C$ is the same
 thing as objects $A = (A_c)_{c \in \Ob(\catoperad P)}$ in $\categ C$ together
 with the structure of a $\catoperad P$-algebra on $F(A)$) and that the 2-functor
 $\tcatalg{\categ{C}}{\catoperad P}_{\lax}
     \to \tcatalg{\categ{D}}{\catoperad P}_{\lax}$ is strictly fully faithful.
\end{proof}

\begin{corollary}
 Let $F: \categ C \to \categ D$ be a strong context functor
 that is also an iso-equivalence.
 Then, for any categorical operad $\catoperad P$, the 2-functor
 $$
 \tcatalg{\categ C}{\catoperad P}_\lax \to \tcatalg{\categ D}{\catoperad P}_\lax
 $$
 is an isoequivalence.
\end{corollary}

\begin{proof}
We already know that this 2-functor is strictly fully faithful.
It suffices then to show that it is strictly essentially surjective.
Let us consider a $\catoperad P$-algebra $B$ in $\categ D$.
Since $F$ is strictly essentially surjective, let us consider objects
$(A_c)_{c \in \Ob(\catoperad P)}$ in $\categ C$ together with isomorphisms
$$
    F(A_c) \simeq B_c ,\  c \in \Ob(\catoperad P).
$$
Using these isomorphisms, one can build a structure of a $\catoperad P$-algebra
$A'$ on the objects $(F(A_c))_{c \in \Ob(\catoperad P)}$ so that these isomorphisms
form an isomorphism of $\catoperad P$-algebras $A'\simeq B$ in $\categ D$.
By Proposition \ref{prop : pullbackalg}, this new $\catoperad P$-algebra $A'$ is the
image of a $\catoperad P$-algebra in $\categ C$.
\end{proof}

\subsection{A long remark about categorical operads}

One has actually a functor
$$
\mathcal A_-(-)_\lax:
\Operad_\cat \times \Operad_\cat^\op \to \Operad_\cat
$$
so that the following diagram commutes
$$
\begin{tikzcd}
    \categ{Context}_\lax \times \Operad_\cat^\op
    \ar[rr, "{\tcatalg{-}{-}_\lax}"]\ar[d, hookrightarrow]
    && \categ{Context}_\lax
    \ar[d, hookrightarrow]
    \\
    \Operad_\cat \times \Operad_\cat^\op 
    \ar[rr, "{\mathcal A_-(-)_\lax}"']
    && \Operad_\cat
\end{tikzcd}
$$

\subsubsection{The lax mapping operad}

\begin{definition}
    Let $\catoperad P, \catoperad Q$ be two categorical operads.
    A $\catoperad P$-algebra in $\catoperad Q$ is the data of a
    morphism of categorical operads
    $$
    A : \catoperad P \to \catoperad Q.
    $$
\end{definition}

\begin{definition}
    Let $\catoperad P, \catoperad Q$ be two categorical operads
    and let $A_1, \ldots, A_n,B$ be $\catoperad P$-algebras in $\catoperad Q$.
    A lax operation from $(A_1, \ldots, A_n)$ to $B$
    is the data of 
    elements
    $$
        m_c \in \catoperad Q (A_1(c), \ldots, A_n(c); B(c)), \quad c \in \Ob(\catoperad P)
    $$
    together with natural transformations $l(c_1, \ldots, c_N;c)$
            $$
            \begin{tikzcd}
                \catoperad P(c_1, \ldots, c_N;c)
                \ar[rr, "B"] \ar[dd, "{(A_i)_{i=1}^n}"'] 
                &{}
                \ar[dd, Rightarrow]
                & \catoperad Q(B(c_1), \ldots, B(c_N);B(c))
                \ar[d, "{ - \triangleleft (m_{c_1}, \ldots, m_{c_N})}"]
                \\
                && \catoperad Q\left((A_i(c_j))_{i=1, j=1}^{n,N} ;B(c)\right)
                \ar[d, leftarrow, "\xi_{n,N}^\ast"]
                \\
                \prod_{i=1}^n \catoperad Q(A_i(c_1), \ldots, A_i(c_N);A_i(c))
                \ar[rr, "{m_c \triangleleft (-,\ldots, -)}"']
                &{}
                & \catoperad Q\left((A_i(c_j))_{j=1, i=1}^{N,n} ;B(c)\right)
        \end{tikzcd}
        $$
        for any colours $c_1, \ldots, c_N,c$ of $\catoperad P$,
        where $\xi_{n,N} \in \mathbb S_{n \times N}$ is the
        permutation
        $$
        (j-1) n + i \mapsto (i-1)N + j, \quad 1 \leq i \leq n, 1 \leq j \leq N.
        $$
        Such natural transformations are required to behave coherently
        with respect to the operadic structure of $\catoperad P$ and $\catoperad Q$ in the sense
        that the following diagrams commute
        $$
\begin{tikzcd}
    B\left(p \triangleleft (p_j)_{j=1}^N\right)
    \triangleleft (m_{c_{k,j}})_{k=1, j=1}^{N_j, N}
    \ar[rr, equal] \ar[dddd]
    && B(p) \triangleleft
    \left(B(p_j) \triangleleft (m_{c_{k,j}})_{k=1}^{N_j}\right)_{j=1}^N
    \ar[d]
    \\
    && B(p) \triangleleft
    \left(\left(m_{c_j}\triangleleft (A_i(p_j))_{i=1}^{n}\right)^{\xi_{n,N_j}}\right)_{j=1}^N
    \ar[d, equal]
    \\
    && \left(\left(B(p) \triangleleft (m_{c_j})_{j=1}^N \right)
    \triangleleft \left(A_i(p_j)\right)_{i=1, j=1}^{n, N}\right)^{(\xi_{n, N_j})_{j=1}^N}
    \ar[d]
    \\
    && \left(\left(m_c \triangleleft (A_i(p))_{i=1}^n \right)^{\xi_{n, N}} \triangleleft
     \left(A_i(p_j)\right)_{i=1, j=1}^{n, N}\right)^{(\xi_{n, N_j})_{j=1}^N}
    \ar[d, equal]
    \\
    \left(m_c \triangleleft \left( A_i(p  \triangleleft (p_j)_{j=1})^N\right)_{i=1}^n\right)^{\xi_{n, \sum_j N_j}}
    \ar[rr, equal]
    &&\left(m_c \triangleleft \left( A_i(p)  \triangleleft (A_i(p_j))_{j=1}^N\right)_{i=1}^n\right)^{\xi_{n, \sum_j N_j}}
\end{tikzcd}
        $$
        $$
\begin{tikzcd}
    B(p) \triangleleft \left(m_{c_{j}}\right)_{j=1}^N
    \ar[r, equal] \ar[d]
    & \left( B(p^\sigma) \triangleleft \left(m_{c_{\sigma(j)}}\right)_{j=1}^N \right)^\mu
    \ar[d]
    \\
    \left(m_c \triangleleft \left(A_i(p)\right)_{i=1}^n \right)^{\xi_{n,N}}
    \ar[r, equal]
    & \left( m_c \triangleleft \left(A_i(p^\sigma)\right)_{i=1}^n  \right)^{\xi_{n,N} \circ \mu}
\end{tikzcd}
\quad 
\begin{tikzcd}
    B(1_c) \triangleleft m_c
    \ar[r, equal] \ar[d]
    & m_c
    \ar[d, equal]
    \\
    m_c \triangleleft A(1_c)
    \ar[r, equal]
    & m_c    
\end{tikzcd}
        $$
for any operations $p \in  \catoperad P(c_1, \ldots, c_N; c)$
and $p_j \in \catoperad P((c_{k,j})_{k=1}^{N_j}; c_j)$
for $j= 1, \ldots, N$ and where $\mu \in \mathbb S_{n \times N}$
is the permutation
$$
(j-1)n + i \mapsto (\sigma(j)-1) +i, \quad 1 \leq j \leq N, 1 \leq i \leq n. 
$$
A morphism of lax operations 
    from $(m_c)_{c \in \Ob(\catoperad P)}$
    to $(m'_c)_{c \in \Ob(\catoperad P)}$
    is the data of a morphism
    $m_c \to m'_c$ in $\catoperad Q(A_1(c), \ldots, A_n(c); B(c))$
    for any colour $c \in \Ob(\catoperad P)$ so that the following
    diagram commutes
    $$
    \begin{tikzcd}
        B(p) \triangleleft (m_{c_1},  \ldots m_{c_N})
        \ar[r] \ar[d]
        & \left( m_c \triangleleft (A_i(p))_{i=1}^n\right)^{\xi_{n,N}}
        \ar[d]
        \\
        B(p) \triangleleft (m'_{c_1},  \ldots m'_{c_N})
        \ar[r]
        & \left( m'_c \triangleleft (A_i(p))_{i=1}^n\right)^{\xi_{n,N}}.
    \end{tikzcd}
        $$
        for any operation $p \in \catoperad P(c_1, \ldots, c_n;c)$.

    Such lax operations and their morphisms
    form a category.
\end{definition}

\begin{definition}
    For any two categorical operads $\catoperad P, \catoperad Q$,
    let $\mathcal A(\catoperad P,  \catoperad Q)_\lax$
    be the categorical operad whose objects are
    $\catoperad P$-algebras in $\catoperad Q$
    and so that for any $\catoperad P$-algebras in $\catoperad Q$
    $A_1, \ldots, A_n, B$,
    $\mathcal A(\catoperad P,  \catoperad Q)_\lax(A_1, \ldots, A_n; B)$
    is the category of lax operations from
    $(A_1, \ldots, A_n)$ to $B$.
    The composition,
    units
    and actions of symmetric groups in $\mathcal A(\catoperad P,  \catoperad Q)_\lax$
    are given by those of the operad $\catoperad Q$.
\end{definition}

\begin{proposition}
    The construction
    $$
    \catoperad Q, \catoperad P \mapsto \mathcal{A}(\catoperad P, \catoperad Q)_\lax
    $$
    induces a functor
    $$
    \Operad_{\mathrm{cat}} \times \Operad_{\mathrm{cat}}^{\op} \to \Operad_{\mathrm{cat}}
    $$
    that has the canonical structure of a lax monoidal structure with respect to 
    the cartesian monoidal structures on these categories. 
\end{proposition}

\begin{proof}
    Given a pair of morphism of categorical operads $f : \catoperad P'
    \to \catoperad P$ and $g: \catoperad Q \to \catoperad Q'$, one obtains a morphism
    of categorical operads
    $$
    (g^\ast, f^\ast): \mathcal{A}(\catoperad P, \catoperad Q)_\lax
    \to \mathcal{A}(\catoperad P', \catoperad Q')_\lax
    $$
    that sends
    \begin{itemize}
        \itemt a colour of $\mathcal{A}(\catoperad P, \catoperad Q)_\lax$
        that is a morphism $A :\catoperad P \to \catoperad Q$ to the morphism
        $g \circ A \circ f : \catoperad P' \to \catoperad Q'$;
        \itemt a lax operation
        $$
        \left(m_c \in \catoperad Q(A_1(c), \ldots, A_n(c); B(c))\right)_{c \in \Ob(\catoperad P)}
        $$
        to the lax operation $g(m_{f(c)})_{c \in \Ob(\catoperad P')}$;
        \itemt a morphism of lax operations to its image through $g$.
    \end{itemize}
    The lax monoidal structure of the functor
    $\mathcal{A}(-, -)_\lax$ is given by natural morphisms
    of operads
    $$
    \mathcal{A}(\catoperad P, \catoperad Q)_\lax \times
    \mathcal{A}(\catoperad P', \catoperad Q')_\lax
    \to \mathcal{A}(\catoperad P \times \catoperad P',
    \catoperad Q \times \catoperad Q')_\lax
    $$
    that sends
    \begin{itemize}
        \itemt a pair of colours $A,A'$ to the morphism
        $$
        A \times A' : \catoperad P \times \catoperad P' \to \catoperad Q \times \catoperad Q';
        $$
        \itemt a pair of lax operations
        $$
        \left(m_c \in \catoperad Q(A_1(c), \ldots, A_n(c); B(c))\right)_{c \in \Ob(\catoperad P)},
        \quad
        \left(m'_{c'} \in \catoperad Q'(A'_1(c'), \ldots, A'_n(c'); B'(c'))\right)_{c' \in \Ob(\catoperad P')}
        $$
        to the lax operation
        $$
        \left( (m_c,m'_{c'}) \in \catoperad Q(A_1(c), \ldots, A_n(c); B(c)) 
        \times \catoperad Q'(A'_1(c'), \ldots, A'_n(c'); B'(c'))\right)_{(c,c') \in \Ob(\catoperad P) \times \Ob(\catoperad P')}.
        $$
    \end{itemize}
\end{proof}

\begin{proposition}
The following diagram commutes
$$
\begin{tikzcd}
    \categ{Context}_\lax \times \Operad_\cat^\op
    \ar[rr, "{\tcatalg{-}{-}_\lax}"]\ar[d, hookrightarrow]
    && \categ{Context}_\lax
    \ar[d, hookrightarrow]
    \\
    \Operad_\cat \times \Operad_\cat^\op 
    \ar[rr, "{\mathcal A_-(-)_\lax}"']
    && \Operad_\cat
\end{tikzcd}
$$
up to a canonical natural isomorphism.
\end{proposition}

\begin{proof}
    This follows from a straightforward checking.
\end{proof}

\subsubsection{Grothendieck strict opfibrations}

\begin{definition}
    Let $f:\catoperad P \to \catoperad Q$ be a morphism of operads. An element $p \in \catoperad P(\underline{c};c)$
    is strictly $f$-cartesian if for any colour $c'$ and any two tuples $\underline{c'}, \underline{c''}$
    of colours of $\catoperad P$, the
    functor from $\catoperad P(\underline{c'}, c, \underline{c''};c)$ to the pullback
    $$
    \catoperad P(\underline{c'}, \underline{c'}, \underline{c''};c')
    \times_{\catoperad Q(f(\underline{c'}, \underline{c'}, \underline{c''});f(c'))}
    \catoperad Q(f(\underline{c'}, c, \underline{c''});f(c'))
    $$
    given by the square diagram
    $$
\begin{tikzcd}
    \catoperad P(\underline{c'}, c, \underline{c''};c)
    \ar[r, "- \triangleleft p"]
    \ar[d,"f"']
    & \catoperad P(\underline{c'}, \underline{c}, \underline{c''};c) 
    \ar[d,"f"]
    \\
    \catoperad Q(f(\underline{c'}, c, \underline{c''});f(c'))
    \ar[r, "- \triangleleft f(p)"]
    & \catoperad Q(f(\underline{c'}, \underline{c}, \underline{c''});f(c'))
\end{tikzcd}
    $$
    is an isomorphism of categories.
\end{definition}

\begin{definition}
    A morphism of operads $f:\catoperad P \to \catoperad Q$ is a strict Grothendieck
    opfibration if for any
    colours $c_1, \ldots, c_n$ of $\catoperad P$ and any operation $q \in \catoperad Q(f(c_1), \ldots, f(c_n);d)$
    there exists a colour $c$ of $\catoperad P$ so that $f(c)=d$ and a strictly $f$-cartesian element
    $p \in \catoperad P(c_1, \ldots, c_n;c)$ so that $f(p) = q$. Such an element $p$ is called a
    cartesian lift of $q$ below $(c_1, \ldots, c_n)$.
\end{definition}

\begin{definition}
    Let $\Operad_{\mathrm{cat}, \mathrm{fib}}$ be the full subcategory
    of $\Operad_{\mathrm{cat}}$ spanned by categorical operads $\catoperad P$ so
    that the morphism $\catoperad P  \to \catoperad{uCom}$ is a strict
    Grothendieck opfibration.
\end{definition}

\begin{definition}
    Let $\catoperad P$ be a categorical operad so that the morphism
    $\catoperad P  \to \catoperad{uCom}$ is a strict
    Grothendieck opfibration. The straightening of $\catoperad P$
    is the small monoidal context $s(\catoperad P)$ defined as follows:
    \begin{itemize}
        \itemt the underlying 2-category of $s(\catoperad P)$
        is the underlying 2-category of $\catoperad P$ in the sense that
        its objects are the colours of $\catoperad P$ ($\Ob(s(\catoperad P)) = \Ob(\catoperad P)$)
        and for any two objects $x,y$, $s(\catoperad P)(x,y) = \catoperad P(x;y)$; the composition
        and the units of $s(\catoperad P)$ are defined by those of $\catoperad P$;
        \itemt for any objects $x_1, \ldots, x_n$, we choose a cartesian lift of the unique
        element of $\catoperad{uCom}$ below $(x_1, \ldots, x_n)$
        that we denote $m(x_1, \ldots, x_n)$ and whose target colour is denoted
        $x_1\otimes  \cdots \otimes x_n$:
        $$
        m(x_1, \ldots, x_n)  \in \catoperad P(x_1, \ldots, x_n, x_1\otimes  \cdots \otimes x_n);
        $$
        in particular, for $n=0$, this gives us a colour of $\catoperad P$ that is the monoidal unit of $s(\catoperad P)$;
        for $n=2$, we obtain the tensor product $x_1 \otimes x_2$;
        \itemt the naturality of the tensor product is given by the composite functor
        $$
        \begin{tikzcd}
            s(\catoperad P)(x,x') \times s(\catoperad P)(y,y')
            \ar[d, equal]
            \\
            \catoperad P(x;x') \times \catoperad P(y;y')
            \ar[d]
            \\
            \catoperad P(x', y'; x' \otimes y') \times (\catoperad P(x;x') \times \catoperad P(y;y'))
            \ar[d, "{- \triangleleft (-,-)}"]
            \\
            \catoperad P(x, y; x' \otimes y')
            \ar[d, "\simeq"]
            \\
            \catoperad P(x \otimes y; x' \otimes y');
        \end{tikzcd}
        $$
        one can check that this makes $- \otimes -$
        a strict 2-functor;
        \itemt the associator is the image of the identity of $(x_1 \otimes x_2) \otimes x_3$
        through the map
        $$
        \begin{tikzcd}
            \catoperad P((x_1 \otimes x_2) \otimes x_3; (x_1 \otimes x_2) \otimes x_3)
        \ar[d, "\simeq"]
        \\
        \catoperad P(x_1 \otimes x_2, x_3; (x_1 \otimes x_2) \otimes x_3)
        \ar[d, "\simeq"]
        \\
        \catoperad P(x_1 , x_2, x_3; (x_1 \otimes x_2) \otimes x_3)
        \ar[d, "\simeq"]
        \\
        \catoperad P(x_1 , x_2 \otimes x_3; (x_1 \otimes x_2) \otimes x_3)
        \ar[d, "\simeq"]
        \\
        \catoperad P(x_1 \otimes (x_2 \otimes x_3); (x_1 \otimes x_2) \otimes x_3);
        \end{tikzcd}
        $$
        the commutator is the image of the identity of $x_1 \otimes x_2$
        through the map
        $$
        \catoperad P(x_1 \otimes x_2; x_1 \otimes x_2) \simeq 
        \catoperad P(x_1 , x_2, x_1 \otimes x_2) \xrightarrow{(1,2)^\ast}
        \catoperad P(x_2 , x_1, x_1 \otimes x_2)
        \simeq \catoperad P(x_2 \otimes x_1; x_1 \otimes x_2);
        $$
        the unitors are the images of the identity of $x$ through the maps 
        \begin{align*}
            &\catoperad P(x;x) \simeq \catoperad P(\II, x;x)  \simeq \catoperad P(\II \otimes x;x) ,
            \\
            &\catoperad P(x;x) \simeq \catoperad P(x, \II;x) \simeq \catoperad P(x \otimes \II;x).
        \end{align*}
    \end{itemize}
\end{definition}

\begin{proposition}
    The image through the functor $\catEnd$ of a small
    monoidal context belongs to the full subcategory $\Operad_{\mathrm{cat}, \mathrm{fib}}$.
    Moreover, the construction
    $$
\catoperad P \in \Operad_{\mathrm{cat}, \mathrm{fib}} \mapsto s(\catoperad P)
    $$
    induces a functor
    $$
s : \Operad_{\mathrm{cat}, \mathrm{fib}} \to \categ{Context}_\lax
    $$
    that is a pseudo-inverse of the functor $\catEnd$.
\end{proposition}

\begin{proof}
    Straightforward.
\end{proof}

\begin{corollary}
    A lax context functor $F :  \categ C \to \categ D$
    is strong if and only if the morphism of operads above $\catoperad{uCom}$
    $$
\catEnd(F) : \catEnd(\categ C)\to \catEnd(\categ D)
    $$
    sends cartesian lifts to cartesian lifts.
\end{corollary}

\begin{proof}
    Straightforward.
\end{proof}

\subsection{Opposite structure}

Let $\categ C$ be a small monoidal context and let $\catoperad P$ be a categorical
operad.

\begin{definition}
    A $\catoperad P$-coalgebra in $\categ C$ is the data of
    a $\catoperad P$-algebra in $\categ C^{\op}$. Then, we define
    $$
    \tcatcog{\categ C}{\catoperad P} = \tcatalg{\categ C^\op}{\catoperad P}^\op.
    $$
\end{definition}

\begin{definition}
    Let $A,B$ be two $\catoperad P$-algebras. They are actually morphisms
    of categorical operads
    $$
    A,B: \catoperad P \to \catEnd(\categ C).
    $$
    They, equivalently may be described as morphism from
    $\catoperad P^{co}$ to $\catEnd(\categ C^{co}) = \catEnd(\categ C)^{co}$.
    Then, an oplax morphism
    from $A$ to $B$ is a morphism in $\tcatalg{\categ C^{co}}{\catoperad P^{co}}_\lax$.
    We define
    $$
    \tcatalg{\categ C}{\catoperad P}_\oplax = \tcatalg{\categ C^{co}}{\catoperad P^{co}}_\lax^{co}
    $$
\end{definition}
    
    Then, one can notice that
    \begin{itemize}
        \itemt the opposite "op" construction swaps algebras and coalgebras;
        \itemt the "co" construction swaps lax morphisms/modules into oplax morphisms/modules.
    \end{itemize}

    One cannot compose a lax $\catoperad P$-morphism with an oplax $\catoperad P$-morphism, but one can rewrite a sequence of a lax $\catoperad P$-morphism
    followed by an oplax $\catoperad P$-morphism into a sequence of an oplax $\catoperad P$-morphism followed by a lax $\catoperad P$-morphism.
    
    \begin{definition}
    Let us consider the following square diagram in $\categ C^{\Ob(\catoperad P)}$
    $$
    \begin{tikzcd}
      A \ar[r,"f"] \ar[d,swap, "f'"]
      &   A' \ar[d,"g"] \ar[ld, Rightarrow, "a"]
      \\
      B \ar[r,swap,"g'"]
      & B'
    \end{tikzcd}
    $$
    where the objects are equipped with structures of $\catoperad P$-algebras, where $f$ and $g'$ are equipped with structures of lax $\catoperad P$-morphisms
    and where $f'$ and $g$ are equipped with structures of oplax $\catoperad P$-morphisms.
    Then, the 2-morphisms
    $$
    a_c : g_c \circ f_c \to g'_c\circ f'_c , \quad c \in \Ob(\catoperad P)
    $$
    form a rewriting if the following diagram commutes
    $$
    \begin{tikzcd}
         B'(p) \circ g^{\otimes n}\circ f^{\otimes n} \ar[d]
         & g \circ A'(p) \circ f^{\otimes n}
         \ar[l] \ar[r] 
         & g \circ f \circ A(p)
         \ar[d]
         \\
         B'(p) \circ g'^{\otimes n}\circ f'^{\otimes n} 
         \ar[r]
         & g' \circ B(p) \circ f'^{\otimes n}
         & g' \circ f' \circ A(p)
         \ar[l]
    \end{tikzcd}
    $$
    for any operation $p$ of $\catoperad P$.
    One can define in a similar way a rewriting from $g' \circ f'$ to $g \circ f$.
    \end{definition}

\subsection{Categorical operads with a fixed set of colours}

Let us restrict our attention to categorical operads with a
fixed set of colours.

\begin{definition}
    For any set of colours $O$, let
    $$
    \Operad_{\cat, O}, \Operad_{\cat, \pl, O}, \categ{Collections}_O
    $$
    be the fibers over $O$ of the $\Ob(-)$ functors
\end{definition}

\begin{lemma}
    Let $O$ be a set.
    The three monadic adjunction
    $$
\begin{tikzcd}
    \categ{Collections}
    \ar[rr, shift left, "\mathrm{T}_\pl"]
    && \categ{Operad}_{\mathrm{cat}, \pl}
    \ar[rr, shift left, "\__{\mathbb{S}}"] \ar[ll, shift left]
    && \categ{Operad}_{\mathrm{cat}} .
    \ar[ll, shift left]
\end{tikzcd}
$$
restrict to three monadic adjunction on the $O$-fibers
$\categ{Collections}_O$, $\categ{Operad}_{\cat, \pl, O}$,
and $\categ{Operad}_{\mathrm{cat}, O}$.
\end{lemma}

\begin{proof}
    Straightforward.
\end{proof}

\begin{proposition}
    Let $O$ be a set.
    The category $\categ{Collections}_O$ is cocomplete
    and the forgetful functors 
    $$
    \categ{Operad}_{\cat, O} \to \categ{Operad}_{\cat, \pl, O}
\to    \categ{Collections}_O
    $$
    create and preserve sifted colimits. Hence the categories
    $\categ{Operad}_{\cat, O}, \categ{Operad}_{\cat, \pl, O}$
    are cocomplete.
\end{proposition}

\begin{proof}
    Straightforward with Appendix \ref{appendixcomonads}.
\end{proof}

\begin{lemma}
    The functor
    $$
\Ob(-) : \Operad_\cat \to \Set
    $$
    is a Grothendieck fibration. The result still
    holds for categorical planar operads.
\end{lemma}

\begin{proof}
    Given a categorical operad $\catoperad P$ and a function
    $\phi: O \to \Ob(\catoperad P)$, the cartesian lifting of
    $\catoperad P$ along $\phi$ is the categorical operad $\phi^\ast(\catoperad P)$
    whose set of colours is $O$ and so that
    $$
    \phi^\ast(\catoperad P)(c_1, \ldots, c_n;c) =
    \catoperad P(\phi(c_1), \ldots, \phi(c_n);\phi(c)) 
    $$
    for any colours $c_1, \ldots, c_n,c \in O$.
\end{proof}

\begin{proposition}
    Let $O$ be a set of objects. The functor
    $$
    \Operad_{\cat, O} \to \Operad_\cat
    $$
    preserves colimits of connected diagrams.
    The result still holds for categorical planar operads.
\end{proposition}

\begin{proof}
    Let us consider a connected diagram 
    $$
D : \categ I \to \Operad_{\cat, O}.
    $$
    whose colimit is denoted $\catoperad Q$,
    and a categorical operad $\catoperad P$.
    Then, the following data are equivalent:
    \begin{itemize}
        \itemt a cocone of $D$ in the category $\Operad_\cat$
        whose target is $\catoperad P$;
        \itemt a function $\phi: O \to \Ob(\catoperad P)$
        and a cocone of $D$ in $\Operad_{\cat, O}$
        whose target is $\phi^\ast(\catoperad P)$ (the cartesian lifting
        of $\catoperad P$ along $\phi$);
        \itemt a function $\phi: O \to \Ob(\catoperad P)$
        and a morphism of categorical operads above $O$ from $\catoperad Q$
        to $\phi^\ast(\catoperad P)$;
        \itemt a morphism of categorical operads 
         from $\catoperad Q$
        to $\catoperad P$.
    \end{itemize}
    Hence, $\catoperad Q$ is the colimit of $D$ in the category
    $\Operad_{\cat}$.
\end{proof}

\subsection{Categorical operads by generators and relations}

\begin{proposition}\label{propcolimoper}
    The category of categorical (planar) operads is complete
    and cocomplete.
\end{proposition}

\begin{proof}
    Since the category of categorical collections is complete
    then that of categorical operads is also complete (see Appendix \ref{appendixcomonads}).

    It remains to show that $\Operad_\cat$ is cocomplete. We will actually
    show how to compute colimits of categorical operads.

On the one hand, coproducts are simple to compute; indeed, for a family
of categorical operads $(\catoperad P_i)_{i \in \set I}$, $\coprod_i \catoperad P_i$
is the categorical operad so that
$$
\Ob(\coprod_i \catoperad P_i) = \coprod_i \Ob(\catoperad P_i)
$$
and for any colours $c_0, c_1, \ldots, c_n$
$$
(\coprod_i \catoperad P_i)(c_1, \ldots, c_n; c_0)
= 
\begin{cases}
    \catoperad P_j(c_1, \ldots, c_n; c_0) \text{ if }\exists j, c_0, c_1, \ldots, c_n \in
    \Ob(\catoperad P_j);
\\
    \emptyset \text{ otherwise.}
\end{cases}
$$

On the other hand, let us consider a pair of morphism of
categorical operads $f,g: \catoperad P \to \catoperad P'$
and let $X$ be its coequaliser in the category of categorical collections.
Let $\categ{coeq}$ be the category with three objects $0,1,2$
and so that 
$$
\hom_{\categ{coeq}}(i,j) = 
\begin{cases}
    \emptyset \text{ if } i >j,
    \\
    \{a,b\} \text{ if } i=0,\  j=1,
    \\
    \ast \text{ otherwise.}
\end{cases}
$$
Let us contemplate the following diagram
that represents a functor $\categ{coeq} \times \categ{coeq} \to
\Operad_\cat$
$$
\begin{tikzcd}
    \mathrm{T} \mathrm{T} \catoperad P
    \ar[r, shift left, "TTf"] \ar[r, shift right, "TTg"']
    \ar[d, shift left] \ar[d, shift right]
    & \mathrm{T} \mathrm{T} \catoperad P'
    \ar[d, shift left] \ar[d, shift right]
    \ar[r]
    & \mathrm{T} \mathrm{T} \catoperad X
    \ar[d, shift left] \ar[d, shift right]
    \\
    \mathrm{T} \catoperad P
    \ar[r, shift left, "Tf"] \ar[r, shift right, "Tg"']
    \ar[d]
    & \mathrm{T} \catoperad P'    
    \ar[r] \ar[d]
    & \mathrm{T} \catoperad X
    \ar[d]
    \\
     \catoperad P
    \ar[r, shift left, "f"] \ar[r, shift right, "g"']
    &  \catoperad P
    \ar[r]
    &  \catoperad Q    
\end{tikzcd}
$$
where $\catoperad Q$ is the coequaliser in $\Operad_{\cat, \Ob(X)}$
(and hence in $\Operad_{\cat}$) of the pair of maps
$\mathrm{T} \mathrm{T} \catoperad X \rightrightarrows \mathrm{T} \catoperad X$.
Since the three columns of this diagram are colimiting
and since the two first rows are also colimiting, then
the last row is colimiting. Hence $\catoperad Q$
is the coequaliser of $f,g$.
\end{proof}

\begin{remark}\label{prop:ploperadsasalgebras}\label{prop:operadsasalgebras}
    One can notice that the category of categorical collections
    is presentable and that the monad 
    $\mathrm{T}$
    on categorical collections
    preserve filtered colimits.
    Hence, using \cite[Corollary 2.47 and Paragraph 2.78]{AdamekRosicky94},
    the category of categorical operads
    is presentable and the forgetful functors
    towards categorical collections preserve
    filtered colimits.
\end{remark}

\begin{corollary}\label{coroperadcolimcontinuousfun}
    Let $\categ D$ be a category. Then a functor
    $$
F : \Operad_{\cat} \to \categ D
    $$
    preserves colimits if an only if the following
    conditions are satisfied
    \begin{itemize}
        \itemt $F$ preserves coproducts;
        \itemt the composite functor 
        $$
        \categ{Collections} \to
        \Operad_{\cat} \xrightarrow{F} \categ D
        $$
        preserves coequalisers;
        \itemt for any set $O$, the composite functor
        $$
        \Operad_{\cat, O} \to
        \Operad_{\cat} \xrightarrow{F} \categ D
        $$
        preserves coequalisers.
    \end{itemize}
\end{corollary}

\begin{proof}
    On the one hand, a functor that preserves colimits clearly
    satisfies these conditions.

    On the other hand, if $F$ satisfies these conditions,
    then it preserves coproducts and coequalisers
    (this is a consequence of the way
    coequalisers are computed in $\Operad_{\cat}$ as described in 
    the proof of Proposition \ref{propcolimoper}).
    Hence it preserves colimits.
\end{proof}

\begin{corollary}
    The functor $\Ob(-)$ from categorical operads 
    to sets preserves colimits.
\end{corollary}

\subsection{Generators of operads and algebras}

Let $\categ C$ be a small monoidal context.

\begin{proposition}\label{propcollectionalg}
    Let $X$ be a categorical collection and let 
    $\catoperad P = \mathrm{T}(X)$.
    \begin{itemize}
        \itemt The $\catoperad P$-algebras in $\categ C$
        are the morphisms of categorical collections
        $$
X \to \catEnd(\categ C),
        $$
        that is the data of objects
        $$
A_c \in \categ C, \quad c \in \Ob(X)
        $$
        and functors
        $$
X(c_1, \ldots, c_n;c) \to \categ C(A_{c_1}\otimes \cdots \otimes A_{c_n}, A_c).
        $$
        \itemt The lax morphisms $f:A \to B$
        are the are the data of morphisms in
        $\categ C$
        $$
        f_c : A_c \to B_c,\quad c \in \Ob(X)
        $$
        together with natural transformations
        $$
        \begin{tikzcd}
            X(c_1, \ldots, c_n;c)
            \ar[r] \ar[d]
            & \categ C(B_{c_1} \otimes \cdots \otimes B_{c_n}, B_c)
            \ar[d] \ar[ld, Rightarrow]
            \\
            \categ C(A_{c_1} \otimes \cdots \otimes A_{c_n}, A_c)
            \ar[r]
            & \categ C(A_{c_1} \otimes \cdots \otimes A_{c_n}, B_c).
        \end{tikzcd}
        $$
        \itemt 2-morphisms $a:f \to g$ are the data of 2-morphisms in $\categ C$
        $$
        a_c : f_c \to g_c, \quad c \in \Ob(X)
        $$
        so that the following diagram commutes
        $$
\begin{tikzcd}
    B(x) \circ f^{\otimes n}
    \ar[r] \ar[d]
    & f \circ  A(x)
    \ar[d]
    \\
    B(x) \circ g^{\otimes n}
    \ar[r]
    & f \circ  A(x)
\end{tikzcd}
        $$
        for any operation $x$ of $X$.
    \end{itemize}
\end{proposition}

\begin{proof}
    Straightforward.
\end{proof}

\begin{lemma}\label{lemmasurjfullyfaithfulstrict}
    Let $f: \catoperad P \to \catoperad Q$ be a morphism
     of categorical operad that is surjective in the sense
     that the underlying function on colours 
     is surjective and for any colours $c_1, \ldots, c_n,c$
     of $\catoperad P$, the functor
     $$
    \catoperad P(c_1, \ldots, c_n;c) \to \catoperad Q(f(c_1), \ldots, f(c_n);f(c))
     $$
     is surjective on objects and morphisms.
     Then the 2-functor
     $$
     \tcatalg{\categ C}{\catoperad Q}_\lax
     \to \tcatalg{\categ C}{\catoperad P}_\lax
     $$
     is injective on objects, morphisms and 2-morphisms.
     The result still holds
     if we replace lax morphisms by oplax morphisms, strong
     morphisms or strict morphisms.
\end{lemma}

\begin{proof}
    Straightforward with the definitions.
\end{proof}

\begin{proposition}
    The functor
    $$
    \Operad_\cat^\op \xrightarrow{\tcatalg{\categ C}{-}_\lax} \categ{Context}_\strict
    $$
    preserves small limits. The result still holds
    if we replace
    lax morphisms by oplax morphisms, strong morphisms
    or strict morphisms.
\end{proposition}

\begin{proof}
    Following Corollary \ref{coroperadcolimcontinuousfun},
    we need to show that
    \begin{itemize}
        \itemt the functor $\tcatalg{\categ C}{-}_\lax$ preserves products;
        \itemt the composite functor
        $$
        \categ{Collections}^\op \to \Operad_{\cat}\xrightarrow{\tcatalg{\categ C}{-}_\lax}
        \categ{Context}_\strict
        $$
        preserves equalisers;
        \itemt for any set of colours $O$, the composite functor
        $$
        \Operad_{\cat, O}^\op \to
        \Operad_\cat^\op \xrightarrow{\tcatalg{\categ C}{-}_\lax} \categ{Context}_\strict
        $$
        preserves equalisers.
    \end{itemize}

    The two first points follow from straightforward checkings (using 
    Proposition \ref{propcollectionalg} for the second point).
    
    Now, let $O$ be a set and let us prove the third point.
    Let us consider a coequaliser
    of categorical operads over the set $O$
    $$
\catoperad P' \rightrightarrows \catoperad P \xrightarrow{\pi} \catoperad Q.
    $$
    Since the coequaliser is computed in categorical collections
    over $O$, then the map $\pi:\catoperad P \to \catoperad Q$
    is surjective.
    We also get a diagram of monoidal contexts and 
    strict context functors
    $$
    \tcatalg{\categ C}{\catoperad Q}_\lax \to
    \tcatalg{\categ C}{\catoperad P}_\lax
    \rightrightarrows
    \tcatalg{\categ C}{\catoperad P}_\lax .
    $$
    Let us denote $L$ the equaliser of the two right
    strict context functors (this is actually the equaliser
    in the category of small strict 2-categories)
    and let us contemplate
    the 2-functors
    $$
    \tcatalg{\categ C}{\catoperad Q}_\lax
    \to L
    \to \tcatalg{\categ C}{\catoperad P}_\lax.
    $$
    They are both injective on objects, morphisms
    and 2-morphisms (Lemma
    \ref{lemmasurjfullyfaithfulstrict}).
    Then, for a $\catoperad P$-algebra $A$, the following assertions
    are equivalent
    \begin{itemize}
        \itemt $A$ is in $\tcatalg{\categ C}{\catoperad Q}_\lax$;
        \itemt the map $A : \catoperad P \to \catEnd(\categ C)$
        factorises (uniquely) through the quotient $\catoperad Q$;
        \itemt both maps from $\catoperad P'$ to $\catEnd(\categ C)$
        $$
        \catoperad P' \rightrightarrows \catoperad P \xrightarrow{A} \catEnd(\categ C)
        $$
        are equal;
        \itemt $A$ is in $L$.
    \end{itemize}
    Then, for any lax $\catoperad P$-morphism $f: A \to B$ between
    algebras $A,B$ that belong to $L$, the following assertions are equivalent:
    \begin{itemize}
        \itemt $f$ is a lax $\catoperad Q$-morphism;
        \itemt the maps
        \begin{align*}
            &B(p) \circ f^{\otimes n} \to f \circ A(p)
            &B(p') \circ f^{\otimes n} \to f \circ A(p')
        \end{align*}
        are equal for any two operations $p,p'$ of $\catoperad P$
        that have the same image in $\catoperad Q$;
        \itemt the maps
        \begin{align*}
            &B(p) \circ f^{\otimes n} \to f \circ A(p)
            &B(p') \circ f^{\otimes n} \to f \circ A(p')
        \end{align*}
        are equal for any two operations $p,p'$ of $\catoperad P$
        that are the two images of the same operation of $\catoperad P'$;
        \itemt $f$ is a morphism of $L$.
    \end{itemize}
    Finally, a $\catoperad P$ 2-morphism $a:f \to g$ between morphisms
    that are in $L$ belongs to $L$ and is actually a $\catoperad Q$ 2-morphism.
    Thus, the 2-functor $\tcatalg{\categ C}{\catoperad Q}_\lax \to L$
    is an isomorphism.
    To conclude, a strict context functor whose underlying 2-functor
    is an isomorphism of strict 2-categories is an isomorphism in
    $\categ{Context}_\strict$.
\end{proof}

\begin{remark}
    Another way to understand the proposition just above is to remember that
    lax $\catoperad P$-morphisms are morphisms over a categorical
    operad $\catoperad P \otimes_{BV,lax} [1]$ (see Remark \ref{remarkbvtensorlax})
    and to notice that the construction
    $$
\catoperad P \mapsto \catoperad P \otimes_{BV,lax} [1]
    $$
    preserves colimits.
\end{remark}

\subsection{Example of monoidal structures in a monoidal context}

Let us consider a small monoidal context $\categ C$. Our goal in this subsection is to describe some types of algebras encoded by categorical operads. 

\subsubsection{Pairing}

\begin{definition}
A $n$-pairing in the monoidal context $\categ C$ is the data of $n+1$ objects $(X_1, \cdots X_n,Y)$ and a morphism
$$
    p :X_1 \otimes \cdots\otimes X_n \to Y .
$$
\end{definition}

A n-pairing will be denoted in this subsection using a corolla with $n$-input $\corollan{}$.

\subsubsection{Pseudo-monoids}

\begin{definition}
 A pseudo monoid in $\categ C$ is an object $A$ together
 with pairings
 \begin{align*}
     m:A \otimes A \to A
     \\
     u:\II \to A
 \end{align*}
 which we can represent as corollas
$m = \corollatwo{}$, $u = \corollazero{}$
and together with 2-isomorphisms
\begin{align*}
    \treeassleft{}
    &\simeq \treeassright{};
    \\
    \treeunitleft{}
    &\simeq \id_A \simeq \treeunitright{};
\end{align*}
called respectively the associator, the left unitor and the right unitor.
so 
that 
\begin{itemize}
    \itemt the associator satisfies the pentagon identity, that is the following diagram commutes
    $$
    
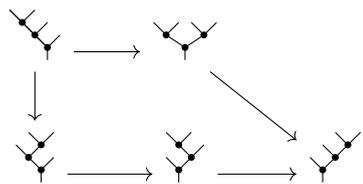
\begin{tikzcd}
      \treepentagonna{}
      \ar[r] \ar[d]
      & \treepentagonnb{}
      \ar[rd]
      \\
      \treepentagonnc{}
      \ar[r]
      &\treepentagonnd{}
      \ar[r]
      & \treepentagonne{};
    \end{tikzcd}
    $$
    \itemt the unitors satisfy the triangle identity,that is the following diagram commutes
    $$
    
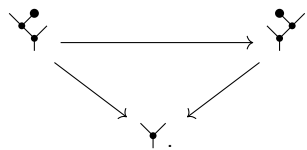
\begin{tikzcd}
      \treetrianglea{}
      \ar[rr] \ar[rd]
      && \treetriangleb{}{}
      \ar[ld]
      \\
      &\corollatwo{}.
    \end{tikzcd}
    $$
\end{itemize}
We denote $\catoperad A_{\infty}$ the categorical planar
operad that encodes pseudo monoids.
\end{definition}

\subsubsection{Pseudo commutative monoids}

Let us denote 
$$
\kappa(X,Y) : X \otimes Y \simeq Y \otimes X
$$
the commutator of the monoidal context $\categ C$.

\begin{definition}
A pseudo commutative monoid is a pseudo-monoid $A$ equipped with a 2-isomorphism
$m_A \simeq m_A \circ \kappa(A,A)$
also written
$$
 \corollatwo{} \simeq \corollatwocomm{}
$$
and called the commutator so that
the following diagrams commute
$$
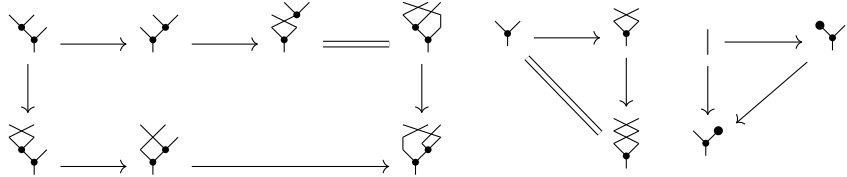
\begin{tikzcd}
     \treeassleft{}
     \ar[r]\ar[d]
     &\treeassright{}
     \ar[r]
     & \corollacommassa{}
     \ar[r,equal]
     &\corollacommassb{}
     \ar[d]
     \\
     \corollacommassc{}
     \ar[r]
     &\corollacommassd{}
     \ar[rr]
     &&\corollacommasse{}
\end{tikzcd}
\quad
\begin{tikzcd}
     \corollatwo{} \ar[r] \ar[rd, equal]
     & \corollatwocomm{} 
     \ar[d]
     \\
     & \corollatwodoublecomm{}
\end{tikzcd}
\quad
\begin{tikzcd}
     {|}
     \ar[r] \ar[d]
     &\treeunitleft{}
     \ar[ld]
     \\
     \treeunitright{}
\end{tikzcd}
$$
We denote $\catoperad E_{\infty}$ the categorical operad that
encodes pseudo commutative monoids.
\end{definition}

\subsubsection{Modules}

\begin{definition}
Given a pseudo monoid $A \in \categ C$, a lax left $A$-module
is an object $M \in \categ C$ equipped with
\begin{itemize}
    \itemt a pairing $m_M: A \otimes M \to M$ also written $\corollamod$
    \itemt two 2-morphisms
    $m_M \circ (m_A \otimes \id_M)
        \to m_M \circ (\id_A \otimes m_M)$ and 
        $m_M \circ (u_A \otimes \id_M)
        \to \id_M$
        also written
        $$
        \treeassleftmod \to \treeassrightmod ; \quad
        \treeunitmodule \to \id_M ;
        $$
\end{itemize}
so that the following diagrams commute
    $$
    
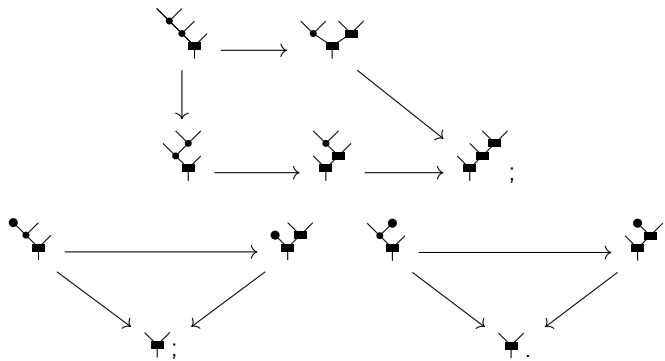
\begin{tikzcd}
      \treepentagonnmoda{}
      \ar[r] \ar[d]
      & \treepentagonnmodb{}
      \ar[rd]
      \\
      \treepentagonnmodc{}
      \ar[r]
      &\treepentagonnmodd{}
      \ar[r]
      & \treepentagonnmode{};
    \end{tikzcd}
    $$
    $$
    \begin{tikzcd}
      \treeunitmoduleasslefta{}
      \ar[rr] \ar[rd]
      && \treeunitmoduleassleftb{}{}
      \ar[ld]
      \\
      &\corollamod{};
    \end{tikzcd}
    \quad
    \begin{tikzcd}
      \treeunitmoduleassrighta{}
      \ar[rr] \ar[rd]
      && \treeunitmoduleassrightb{}
      \ar[ld]
      \\
      &\corollamod{}.
    \end{tikzcd}
    $$
\end{definition}

Pairs of a pseudo-monoid $A$ and a lax left module $M$ are
algebras over a categorical planar operad $\catoperad{LM}_{\lax}$
with two colours.

\begin{definition}
One can define similarly
\begin{itemize}
    \itemt an oplax left module with a pairing $m_M : A \otimes M \to M$ and 2-morphisms $m_M \circ (\id_A \otimes m_M) \to m_M \circ (m_A \otimes \id_M)$ and 
    $\id_M \to m_M \circ (u_A \otimes \id_M)$;
    \itemt a lax right module
    with a pairing $m_M : M \otimes A \to M$ and 2-morphisms $m_M \circ (\id_M \otimes m_A) \to m_M \circ (m_M \otimes \id_A)$ and 
    $m_M \circ (id_M \otimes u_A) \to \id_M$;
    \itemt an oplax right module with a pairing $m_M : M \otimes A \to M$ and 2-morphisms $m_M \circ (m_M \otimes \id_A) \to m_M \circ (\id_M \otimes m_A)$ and 
    $\id_M \to m_M \circ (id_M \otimes u_A)$;
\end{itemize}
where the 2-morphisms are required to satisfy mutatis mutandis the same conditions as those in the definition of a lax left module.
\end{definition}

\begin{remark}
    One can notice that 
    the transposition "tr" construction swaps left modules and right modules.
\end{remark}

\begin{definition}
    A lax left (resp. right) module $M$ over a pseudo monoid $A$
    is called strong if the structural 2-morphisms are invertible.
    Pairs of a pseudo-monoid $A$ and a strong left (resp. right)
    module $M$ are
algebras over a categorical planar operad $\catoperad{LM}_{\strong}$
(resp. $\catoperad{RM}_{\strong}$)
with two colours.
\end{definition}

\subsubsection{Mac Lane's coherence for operads}

\begin{proposition}\cite{MacLane63}
    \begin{itemize}
        \itemt The categorical planar operad $\catoperad A_{\infty}$ has one colour
        and $\catoperad A_{\infty}(n)$
        is the groupoid equivalent to the point category
        whose objects are (isomorphism classes of) planar trees
        with only arity 2 and arity 0 nodes; the composition is computed
        by grafitng trees.
        \itemt The categorical operad $\catoperad E_{\infty}$ has one colour
        and $\catoperad E_{\infty}(n)$
        is the groupoid equivalent to the point category
        whose objects are pairs $(t,\sigma)$ of a
        (isomorphism class of) planar trees
        with only arity 2 and arity 0 nodes and a permutation 
        $\sigma \in \mathbb S_n$; the composition is computed
        by grafitng trees.
    \end{itemize}
\end{proposition}

\begin{proof}
    This follows from the same arguments as those used to
    prove MacLane's coherence theorem \cite{MacLane63}    
\end{proof}

 \subsection{Doctrinal adjunction}

Let $\categ C$ be a small monoidal context and let $\catoperad P$
be a categorical operad. Such an operad induces a doctrine (that is a 2-monad)
on $\categ C$ (see \cite{Kelly74}) whose definition of (op)lax morphisms matches with
that of Definition \ref{def :laxmorphis}.
Along these lines of thought,  Kelly's results on doctrinal adjunctions apply in the framework of operads.

Remember that, using the same notation as in Definition \ref{def :laxmorphis},
    an oplax $\catoperad P$-morphism from $A$ to $B$ is the data of morphisms in $\categ C$
    $$
    f_c : A_c \to B_c
    $$
    for any colour $c \in \Ob(\catoperad P)$, and 2-morphisms
    $$
        f \circ A(p) \to B(p) \circ f^{\otimes n}
    $$
    for any element $p \in \catoperad P(c_1, \cdots, c_n ;c)$ that satisfy some conditions.

\begin{theorem}\cite{Kelly74}
Let us consider two $\catoperad P$-algebras $A, B$ in the monoidal context $\categ C$
and, for any colour $c \in \Ob(\catoperad P)$, an adjunction
$$
\begin{tikzcd}
     A_c \ar[rr, shift left, "l_c"]
     && B_c \ar[ll, shift left, "r_c"]
\end{tikzcd}
$$
in $\categ C$ whose unit and counit are denoted respectively $\eta_c$ and $\epsilon_c$. Then,
\begin{enumerate}
    \item the set of structures of an oplax $\catoperad P$-morphism
    on $l$ and the set of structures
    of a lax $\catoperad P$-morphism on $r$ are canonically related by a bijection that we call Kelly's bijection;
    \item given a structure of an oplax $\catoperad P$-morphism
    on $l$ and a structure of a lax $\catoperad P$-morphism
    on $r$, they are related to each other through Kelly's bijection if and only if the
    2-morphisms $\id \circ \id \to r_c \circ l_c$ and $l_c \circ r_c \to \id \circ \id$ are rewritings;
    \item if $l$ and $r$ are equipped with structures of lax $\catoperad P$-morphisms then the 2-morphisms $\eta_c$ and $\epsilon_c$ are $\catoperad P$ 2-morphisms if and only if $l$
    is a strong $\catoperad P$-morphism (hence and oplax $\catoperad P$-morphism) and its oplax structure is related to the lax structure on $r$ through Kelly's bijection;
    \item if $l$ and $r$ are equipped with structures of oplax $\catoperad P$-morphisms then the 2-morphisms $\eta_c$ and $\epsilon_c$ are $\catoperad P$ 2-morphisms if and only if $r$
    is a strong $\catoperad P$-morphism (hence a lax $\catoperad P$-morphism) and its lax structure is related to the oplax structure on $l$ through Kelly's bijection.
\end{enumerate}
\end{theorem}

\begin{proof}
Let us first prove (1). Given the structure of an oplax $\catoperad P$ morphism on $l$, the related structure of a lax $\catoperad P$-morphism on $r$ is given by morphisms of the form
$$
 A(p) \circ r^{\otimes n} \to r \circ l \circ  A(p) \circ r^{\otimes n}
 \to r\circ  B(p) \circ l^{\otimes n} \circ r^{\otimes n}
 \to r \circ B(p) .
$$
Conversely, given the structure of a lax $\catoperad P$ morphism on $r$, the related structure of an oplax $\catoperad P$-morphism on $l$ is given by
morphisms of the form
$$
l \circ  A(p) \to 
l \circ  A(p) \circ r^{\otimes n} \circ l^{\otimes n}
 \to l \circ r \circ  B(p) \circ l^{\otimes n} 
 \to B(p) \circ l^{\otimes n}  .
$$
A straightforward check shows that these formulas do define respectively a lax and an oplax structure and that they are inverse to each other. This defines Kelly's bijection.

Now, let us prove (2). The fact that the 2-morphisms $\id \circ \id \to r_c \circ l_c$ and $l_c \circ r_c \to \id \circ \id$ are rewritings means that the following squares are commutative
$$
\begin{tikzcd}
     l \circ A(p) \circ r^{\otimes n}
     \ar[r] \ar[d]
     & l \circ r \circ B(p)
     \ar[d]
     \\
     B(p)\circ (l \circ
     r)^{\otimes n}
     \ar[r]
     & B(p)
\end{tikzcd}
\quad
\begin{tikzcd}
     A(p)
     \ar[r] \ar[d]
     & r \circ l \circ A(p)
     \ar[d]
     \\
     A(p)\circ (r \circ l)^{\otimes n}
     \ar[r]
     & r \circ B(p) \circ l^{\otimes n}
\end{tikzcd}
$$
for any operation $p$ of the operad $\catoperad P$ (actually, if one square is commutative, then the other one is also commutative). A straightforward diagram chasing shows that the commutation of these diagrams is equivalent to the fact that the oplax structure on $l$ and the lax structure on $r$ are related through Kelly's bijection.

Then, let us prove (3). In that context, $l$ is equipped with the structure of a lax $\catoperad P$-morphism and with the structure of an oplax $\catoperad P$-morphism. Then, using the right commutative square just above, it is straightforward to check that the fact that the natural transformations $\id \to r_c l_c$ form a $\catoperad P$ 2-morphism is equivalent to the fact that the composite map
$$
l \circ A(p) \to A(p) \circ l^{\otimes n} \to  
l \circ A(p)
$$
is the identity of $l \circ A(p)$ for any $p$. Similarly, using the left commutative square just above,
it is straightforward to check that the fact that the natural transformations $l_cr_c \to \id$ form a $\catoperad P$ 2-morphism is equivalent to the fact that the composite map
$$
A(p) \circ l^{\otimes n} \to l \circ A(p) 
\to A(p) \circ l^{\otimes n}
$$
is the identity of $A(p) \circ l^{\otimes n}$ for any $p$.

Finally, (4) may be proven using the same arguments as (3).
\end{proof}

\begin{proposition}
A lax $\catoperad P$ morphism is an isomorphism (resp. an equivalence) in the 2-category of $\catoperad P$-algebras, lax $\catoperad P$-morphisms and $\catoperad P$ 2-morphisms
if and only if it is a strong $\catoperad P$-morphism and the underlying morphism in $\categ C^{\Ob(\catoperad P)}$ is an isomorphism (resp. an equivalence).
We have the same result when considering oplax morphisms instead of lax morphisms.
\end{proposition}

\begin{proof}
Let $f : A \to B$ be a lax $\catoperad P$-morphism.

Let us suppose that $f$ is an equivalence in the 2-category
of $\catoperad P$-algebras and lax morphisms, with right adjoint $g$.
It is clear that the underlying morphism of $f$ in
$\categ C^{\Ob(\catoperad P)}$ is an equivalence (with pseudo-inverse the
underlying morphism of $g$). Moreover,
by doctrinal adjunction, $f$ is a strong $\catoperad P$-morphism.
 
 Conversely, let $f :A \to B$ be a strong $\catoperad P$-morphism
 whose underlying morphism in $\categ C^{\Ob(\catoperad P)}$
 is an equivalence, with right adjoint $g : B \to A$.
 Then, again by doctrinal adjunction, $g$ inherits the
 structure of a lax $\catoperad P$-morphism so that the
 adjunction
 in $\categ C^{\Ob(\catoperad P)}$ relating $f$ and $g$
 lifts to an adjunction in the 2-category of algebras.
 Since the unit and the counit of this adjunction are
 2-isomorphisms, this is an adjoint equivalence.
\end{proof}

\subsection{Examples}

    In the small cartesian monoidal context of
    $\mathcal V$-small strict 2-categories,
    pseudo commutative monoids
    are precisely monoidal contexts and lax $\catoperad E_\infty$-morphisms are lax context functors.

    In the small cartesian monoidal context $\categ{Cats}_{\mathcal{V}-\mathrm{small}}$
     of $\mathcal V$-small categories, one can check that:
\begin{itemize}
    \itemt pseudo (commutative) monoids are (symmetric) monoidal categories;
    \itemt (op)lax morphisms of pseudo monoids are (op)lax monoidal functors;
    \itemt (op)lax morphisms of pseudo commutative monoids are (op)lax symmetric monoidal functors;
    \itemt $\catoperad A_\infty$ 2-morphisms are monoidal natural transformations;
    \itemt $\catoperad E_\infty$ 2-morphisms are symmetric monoidal natural transformations;
    \itemt lax left modules are categories tensored over a monoidal category;
    \itemt lax morphisms of lax left modules over a lax monoidal functor are functors equipped with a strength;
    \itemt a cotensorisation of a category $\categ D$ by
    a monoidal category $\categ C$ is the structure
    of a oplax right $\categ C^{\op}$-module on $\categ D$,
    or equivalently, the structure of a lax
    right $\categ C$-module on $\categ D^{\op}$.
\end{itemize}

In the small cartesian monoidal context $\categ{Functors}_{\mathcal V - \mathrm{small}}$
(that is the full sub 2-category of $\categ{Functors}$ spanned by functors
between $\mathcal{V}$-small categories, that is stable through finite products):
\begin{itemize}
    \itemt the pseudo (commutative) monoids are given by pairs of (symmetric) monoidal categories and strict monoidal functors.
    \itemt the (op)lax monoidal functors
    from $F_1 : \categ C_1 \to \categ D_1$ to $F_2 : \categ C_2 \to \categ D_2$ are pairs of (op)lax
    monoidal functors $S : \categ C_1
    \to \categ C_2$ and $T : \categ D_1
    \to \categ D_2$ so that $F_2\circ S = T \circ F_1$ as (op)lax monoidal functors;
    \itemt the lax left modules over a pseudo monoid $\categ A_1 \to \categ A_2$ are
    the data of a category $\categ M_1$ tensored
    over $\categ A_1$, a category $\categ M_2$ tensored over $\categ A_2$ and a functor
    $\categ M_1 \to \categ M_2$ that commutes with all the tensoring structures.
\end{itemize}


\section{Comonads and monoidal structures}

In this section, we describe the monoidal context of comonads and show how it is related to the monoidal context $\Functors$.

From now on, a set or a small set is a $\mathcal U$-small set 
and a large set is a $\mathcal U$-large set. Moreover, otherwise stated,
a category is a $\mathcal U$-category. Then, we will denote
\begin{itemize}
    \itemt $\Cats$ will no more denote the cartesian monoidal
    context of $\mathcal W$-small categories but that of $\mathcal{U}$-categories;
    \itemt $\Functors$ will denote the cartesian monoidal
    context
    $\TwoFun(\Mor, \categ{Cats}_{\mathcal U})$ instead of \\$\TwoFun(\Mor, \categ{Cats}_{\mathcal{W}-\mathrm{small}})$. 
\end{itemize}

\subsection{The monoidal context of comonads}

\begin{definition}
Let $\Delta_{act}$ be the sub category of $\Delta$ made up of active morphisms. More precisely, its objects are the posets
$$
    [n] \coloneqq (0 < 1 < \cdots < n)
$$
for $n \in \mathbb N$ and its morphisms from $[n]$ to $[m]$ are the morphisms of posets (that is functors) $f$ so that
$f(0)=0$ and $f(n)=m$. This is a strict monoidal category with tensor product
$$
    [n] \otimes [m] = [n+m].
$$
Moreover, let $B(\Delta_{act})$ be the delooping of the monoidal category $\Delta_{act}$, that is the strict 2-category with one object $\ast$ and so that
$$
    B(\Delta_{act}) (\ast,\ast) = \Delta_{act}.
$$
\end{definition}

\begin{definition}
A category with comonad is a pair $(\categ C, Q)$ of a category equipped with a comonad. Equivalently, this is a 2-functor from $B(\Delta_{act})$ to $\Cats$.
\end{definition}

In particular, categories with comonads are algebras over a categorical operad (this is just the 2-category $B(\Delta_{act})$ seen as an operad) in the monoidal context $\Cats$. Thus, one can then define (op)lax morphisms between categories with comonads.

\begin{definition}
A (op)lax comonad functor between two categories with comonads $(\categ C, Q)$ and $(\categ D, R)$ is a (op)lax morphism of $B(\Delta_{act})$-algebras. For instance an
oplax comonad functor is the data of a functor
$F: \categ C \to \categ D$ and a natural transformation
\[
	F \circ Q \xrightarrow{A} R \circ F
\]
so that the
the following diagrams commute
\[
\begin{tikzcd}
 	FQ \ar[rr] \ar[d]
	&& RF
	\ar[d]
	\\
	FQQ
	\ar[r]
	& RFQ
	\ar[r]
	& FRR
\end{tikzcd}
\quad
\begin{tikzcd}
 	FQ \ar[r] \ar[d]
	& RF
	\ar[d]
	\\
	F\id
	\ar[r,equal]
	& \id F .
\end{tikzcd}
\]
\end{definition}

\begin{definition}
 Let $\categ{Comonads}$ be the strict 2-category made up of $B(\Delta_{act})$-algebras, oplax $B(\Delta_{act})$-morphisms and $B(\Delta_{act})$ 2-morphisms. More precisely, its
\begin{itemize}
 \itemt object are categories with comonads $(\categ C, Q)$;
 \itemt morphisms are oplax comonad functors;
\itemt 2-morphisms between morphisms $(F,A)$ and $(F',A')$ from $(\categ C, Q)$ to $(\categ D, R)$ are natural transformations $F \to F'$
 so that the following diagram commutes
 \[
\begin{tikzcd}
 	FQ
	\ar[r] \ar[d]
	& RF
	\ar[d]
	\\
	F'Q
	\ar[r]
	& RF' .
\end{tikzcd}
 \]
\end{itemize}
\end{definition}

\begin{proposition}
 The 2-category $\categ{Comonads}$ has strict finite products and hence is a cartesian monoidal context.
\end{proposition}

\begin{proof}
 Straightforward.
\end{proof}

\begin{definition}
 Given a category $\categ C$ and a comonad $Q$ on it, let $\Cog_{\categ C}(Q)$
 be the category of $Q$-coalgebras. Equivalently, this is the mapping category
 $$
\Cog_{\categ C}(M) = \Comonads\left(\ast_{\categ{Monads}},  (\categ C, M)\right) .
 $$
 This defines a 2-functor from $\Comonads$ to $\Cats$.
\end{definition}

\begin{proposition}
 The 2- functor $\Cog$ preserves strict finite products and hence is a context functor.
\end{proposition}

\begin{proof}
 This follows from the fact that $\ast_{\categ{Monads}}$ is
 a cocommutative coalgebra.
\end{proof}

\begin{proposition}
 The forgetful 2-functor $\Comonads \to \Cats$ preserves strict finite products and is hence a context functor.
\end{proposition}

\begin{proof}
 Straigthforward.
\end{proof}

\subsection{Monoidal structures in the monoidal context of comonads}

Our goal in this subsection is to describe pairings, pseudo-monoids and their modules in the monoidal context of comonads.

\subsubsection{Pairings}

A 2-pairing in the monoidal context of comonads from the pair $(\categ C, Q), (\categ D,O)$ to $(\categ E, R)$ consists in a bifunctor
\begin{align*}
    \categ C \times \categ D & \to \categ E
    \\
    (X , Y) &\mapsto X \otimes Y;
\end{align*}
together with a 
natural transformation
 \[
 	Q(X) \otimes O(Y)\to R(X \otimes Y)
 \]
 so that the following diagrams commute
 \[
\begin{tikzcd}[column sep=small]
 	Q(X) \otimes O(Y)
	\ar[rr] \ar[d]
	&& R(X \otimes Y)
	\ar[d]
	\\
	QQ(X) \otimes OO(Y)
	\ar[r]
	& R(Q(X) \otimes O(Y))
	\ar[r]
	& RR(X \otimes Y)
\end{tikzcd}
\quad
\begin{tikzcd}[column sep=small]
 	Q(X) \otimes O(Y)
	\ar[r] \ar[d]
	& R(X \otimes Y) \ar[d]
	\\
	X \otimes Y \ar[r,equal]
	& X \otimes Y .
\end{tikzcd}
 \]

\subsubsection{Hopf comonads}

A pseudo-monoid in the monoidal context of comonads consists in a monoidal category $\categ C$ together with a Hopf comonad; this notion is dual to that of a Hopf monad (see \cite{Moerdijk02}).

\begin{definition}
 A Hopf comonad on a monoidal category $(\categ C, \otimes ,\II)$ is the data of a comonad $(Q,w, n)$
on $\categ C$ together with a structure of a lax monoidal functor on $Q$
\begin{align*}
 		&Q(X) \otimes Q(Y) \to Q(X\otimes Y);
		\\
		&\II \to Q(\II) ;
\end{align*}
so that the natural transformations  $w: Q \to QQ$ and $n: Q \to \id$
are monoidal natural transformations.
\end{definition}

Then, a pseudo commutative monoid in the monoidal context of comonads is given by a symmetric monoidal category equipped with a commutative Hopf comonad.

\begin{definition}
Let $Q$ be a Hopf comonad in a symmetric monoidal category $\categ C$. It is said to be commutative if the structure of a lax monoidal functor on $Q$ is symmetric.
\end{definition}

Then, one can notice that
\begin{itemize}
    \itemt a lax $\catoperad A_\infty$-morphism in the context of comonads from $(\categ C, Q)$
    to $(\categ D, R)$ is given by a lax monoidal functor $F : \categ C \to \categ D$ and a natural transformation
$$
    FQ \to RF
$$
that is a monoidal natural transformation and that makes $F$ an oplax comonad functor;
\itemt an oplax $\catoperad A_\infty$-morphism in the context of comonads from $(\categ C, Q)$
    to $(\categ D, R)$ is given by an oplax monoidal functor $F : \categ C \to \categ D$ and a natural transformation
$$
    FQ \to RF
$$
that makes $F$ an oplax comonad functor
and
that is also a rewriting.
\end{itemize}

\subsubsection{Comonads comodules}

A pair of a pseudo monoid together with a left lax module in the monoidal context of comonads is the data of a
Hopf comonad $(Q,w,\tau)$ on a monoidal category $\categ E$
together with a category $\categ C$ tensored over $\categ E$ and a comonad $R$ on $\categ C$ equipped with the structure of a Hopf $Q$-comodule as defined in the following definition.

\begin{definition}
A structure of Hopf comodule comonad
on the comonad $(R,w',\tau')$ is the data of a strength on the functor $R$ with respect to the lax monoidal functor $Q$
\[
	Q(X) \boxtimes B(Y) \to B( X \boxtimes Y)
\]
so that the natural transformation $w'$ is strong with respect to the monoidal natural transformation
$w$ and so that $\tau'$ is strong
with respect to $\tau$.
\end{definition}

\subsection{Monads}

We have a canonical isomorphism of monoidal contexts
$$
\Cats \simeq \Cats^{co}
$$
that sends a category $\categ C$ to its opposite category $\categ C^{\op}$. This isomorphism lifts to an isomorphism between the monoidal context $\Comonads$ and the monoidal context of monads.

\subsubsection{The monoidal context of monads}

\begin{definition}
A category with monad $(\categ C, M)$ is the data of a category $\categ C$ and a monad $M$ on $\categ C$.
Equivalently, this is a 2-functor from $B(\Delta_{act}^{\op})$ to $\Cats$.
Moreover, a lax monad functor between two categories with monads $(\categ C, M)$ and $(\categ D, N)$ is a lax morphism of algebras over $B(\Delta_{act}^{\op})$, that is the data of a functor $F :\categ C \to \categ D$ and a natural transformation
\[
	N \circ F \xrightarrow{A} F \circ M
\]
so that the following diagrams commute
\[
\begin{tikzcd}
 	NNF \ar[r] \ar[d]
	& NFM
	\ar[r]
	& FMM
	\ar[d]
	\\
	NF
	\ar[rr]
	&& FM
\end{tikzcd}
\quad
\begin{tikzcd}
 	\id F \ar[r,equal] \ar[d]
	& F \id
	\ar[d]
	\\
	NF
	\ar[r]
	& FM .
\end{tikzcd}
\] 
\end{definition}

\begin{definition}
 Let $\categ{Monads}$ be the strict 2-category
 of $B(\Delta_{act}^{\op})$-algebras, lax $B(\Delta_{act}^{\op})$-morphisms and $B(\Delta_{act}^{\op})$ 2-morphisms.
\end{definition}

\begin{proposition}
 The strict 2-category $\Monads$ form a cartesian monoidal context and the construction
 $$
  \categ C \in \Cats \mapsto \categ C^{\op}
 $$
 induces a canonical isomorphism of monoidal contexts
 $$
 \Monads \simeq \Comonads^{co}.
 $$
\end{proposition}

\begin{proof}
 Straightforward.
\end{proof}

\subsubsection{Hopf monads and module monad over a Hopf comonad}

A pseudo-monoid in the monoidal context of monads is a monoidal category equipped with a Hopf monad.

\begin{definition}\cite{Moerdijk02}
 A Hopf monad on a monoidal category $(\categ C, \otimes ,\II)$ is the data of a monad $M$ together with the structure of a Hopf comonad on the related comonad on $\categ C^{\op}$. Equivalently, this is the data of a monad $M$
 on $\categ C$ together with a structure of an oplax monoidal functor on $M$
\begin{align*}
 		&M(X \otimes Y) \to M(X) \otimes M(Y)
		\\
		&M(\II) \to \II
\end{align*}
so that the natural transformations  $m: MM \to M$ and $u: \id \to M$
are monoidal natural transformations.
\end{definition}

Let us consider a Hopf comonad $Q$ on a monoidal category $\categ C$, a category $\categ D$ cotensored over $\categ C$ through a bifunctor
\begin{align*}
\categ D \times \categ C^{\op} &\to \categ D    
\\
(X,Y) &\mapsto \langle X,Y \rangle ;
\end{align*}
and a monad $M$ on $\categ D$.
A structure of a lax right $(\categ C,Q)$-module on $(\categ D^{\op}, M)$ in the monoidal context of comonads that enhances the cotensorisation of $\categ D$ by $\categ C$ corresponds to the structure of a Hopf $Q$-module monad on $M$.

\begin{definition}
A Hopf $Q$-module monad is the data of a monad $(M,m,u)$ on $\categ D$ together
with a strength on the functor $M$ with respect to the lax monoidal functor $Q$
\[
	M(\langle X,Y \rangle)
	\to \langle M(X), Q(Y) \rangle
\]
so that the following diagrams commute
$$
\begin{tikzcd}
     MM \langle X, Y \rangle
     \ar[r] \ar[d]
     & M \langle M(X), Q(Y) \rangle
     \ar[r]
     & \langle MM(X), QQ(Y) \rangle
     \ar[d]
     \\
     M \langle X, Y \rangle
     \ar[rr]
     && \langle MM(X), QQ(Y) \rangle ;
\end{tikzcd}
$$
$$
\begin{tikzcd}
     \langle X, Y \rangle
     \ar[r,equal] \ar[d]
     & \langle X, Y \rangle
     \ar[d]
     \\
     M \langle X, Y \rangle
     \ar[r]
     & \langle M(X), Q(Y) \rangle .
\end{tikzcd}
$$
\end{definition}

\subsection{From comonads to coalgebras and back to comonads}

Let us consider two pairs $(\categ C, Q)$ and $(\categ D, R)$ of a categories
with comonads (that is objects in $\Comonads$).

\begin{proposition}\label{prop : naturcomonad}
 Let $F: \categ C \to \categ D$ be a functor. Then there is a canonical bijection between
 \begin{enumerate}
     \item the set of functors $F_{cog} : \Cog_{\categ C}(Q) \to \Cog_{\categ D}(R)$
 that lifts $F: \categ C \to \categ D$ (that is $FU_Q = U_RF_{cog}$);
 \item the set of oplax comonad structures on $F$
 \[
	F \circ Q \xrightarrow{\beta} R \circ F ,
\]
with respect to $Q$ and $R$.
 \end{enumerate}
\end{proposition}

\begin{proof}
Given a functor $F_{cog} : \Cog_{\categ C}(Q) \to \Cog_{\categ D}(R)$ that lifts $F$, the equality $FU_Q  = U_R F_{cog}$ gives us by adjunction a morphism
$$
    F_{cog} L^Q \to L^R F
$$
and then a morphism
$$
    FQ = F U_QL^Q = U_R F_{cog} L^Q
    \to U_R L^R F = RF .
$$
One can check that the resulting map
$FQ \to RF$ is an oplax comonad structure on $F$. 

Conversely, given an oplax comonad structure $FQ \to RF$ on $F$, then for any $Q$-algebra $V$, the object $F(V)$ has the structure of
a $R$-coalgebra given by the map
\[
	F(V) \to FQ(V) \to RF(V).
\]
This construction is natural and defines the expected lifting functor.

A straightforward check shows that the two constructions are inverse to each other.
\end{proof}

\begin{corollary}\label{prop : naturmonad}
Given a monad $M$ on $\categ C$ and a monad $N$ on $\categ{D}$,
there is a canonical bijection between
\begin{enumerate}
 \item the set of functors $F_{alg} : \catalg{\categ C}(M) \to \catalg{\categ D}(N)$
 that lifts $F: \categ C \to \categ D$ (that is $FU^M= U^NF_{alg}$);
 \item the set of lax monad structures on $F$
 \[
	N \circ F \xrightarrow{\alpha} F \circ M ,
\]
with respect to $M$ and $N$.
\end{enumerate}
\end{corollary}

\begin{proof}
This follows from the same arguments as those used to prove Proposition \ref{prop : naturcomonad}.
In particular, for any $M$-algebra $A$, the object $F(A)$ has the structure of
a $N$ algebra given by the map
\[
	NF(A) \to FM(A) \to F(A) .
\]
\end{proof}

\begin{proposition}\label{prop : naturrestric}
 Let us consider two functors $F,G : \categ C \to \categ D$ together with liftings $F_{cog}, G_{cog} : \Cog_{\categ C}(Q) \to \Cog_{\categ D}(R)$ to the categories of coalgebras, that correspond to oplax comonad structures on $F$ respectively denoted $\alpha$ and $\beta$. Moreover, let $A : F \to G$ be a natural transformation. Then, the following assertions are equivalent:
 \begin{enumerate}
     \item the natural transformation $A$ lifts to a 2-morphism in $\Comonads$ from $(F,\alpha)$ to $(G,\beta)$;
     \item the natural transformation $A$ lifts to a 2-morphism in $\Functors$ from $(F,F_{cog})$ to $(G,G_{cog})$;
     \item for any $Q$-coalgebra $V$, the map
     $$
        A(V) : F(V) \to G(V)
     $$
     is a morphism of $R$-coalgebras.
 \end{enumerate}
\end{proposition}

\begin{proof}
The assertion (2) is clearly equivalent to (3). Let us prove that (3) is equivalent to (1).

On the one hand, let us assume (1). Then for any
$Q$-coalgebra $V$, the following diagram commutes
$$
\begin{tikzcd}
     F(V)
     \ar[d, "A(V)"']
     \ar[r]
     & FQ(V)
     \ar[d,"A(Q(V))"]
     \ar[r]
     & RF(V)
     \ar[d,"R(A(V))"']
     \\
     G(V)
     \ar[r]
     & GQ(V)
     \ar[r]
     &RG(V)
\end{tikzcd}
$$
Thus, the map $A(V) : F(V) \to G(V)$ is a morphism of $R$-coalgebra.

Conversely, let us assume (3). Then, the square diagram
$$
\begin{tikzcd}
     FQ(X)
     \ar[d,"A(Q(X))"] \ar[r]
     & RF(X)
     \ar[d,"R(A(X))"']
     \\
     GQ(X)
     \ar[r]
     &RG(X)
\end{tikzcd}
$$
decomposes as 
$$
\begin{tikzcd}
     FQ(X)
     \ar[d,"A(Q(X))"'] \ar[r]
     & FQQ(X) \ar[d,"A(QQ(X))"] \ar[r]
     & RFQ(X) \ar[d,"R(A(Q(X)))"] \ar[r]
     & RF(X)
     \ar[d,"R(A(X))"]
     \\
     GQ(X)
     \ar[r]
     & GQQ(X)
     \ar[r]
     & RGQ(X)
     \ar[r]
     & RG(X)
\end{tikzcd}
$$
The left square and the right square are commutative by naturality. The middle square is commutative since $Q(X)$ is a $Q$-coalgebra.
Hence, the whole square is commutative, which shows (1).
\end{proof}

\begin{theorem}\label{main}
The construction that sends a category with a comonad
$(\categ C,Q)$ to the functor $\Cog_{\categ C}(Q) \to \categ C$ canonically induces a 2-functor from $\Comonads$ to $\Functors$ that is strictly fully faithful and that preserves strict finite products.
\end{theorem}

\begin{proof}
Such a 2-functor sends a morphism (that is an oplax comonad functor) $(F,A)$ to the pair of functors $(F,F_{cog}) $ defined in Proposition \ref{prop : naturcomonad}, and a 2-morphism (that is a natural transformation) $A' : (F,A) \to (G,B)$ to the pair of natural transformation $(A'',A')$ whose first component is defined in Proposition \ref{prop : naturrestric}.

It is strictly fully faithful by Proposition \ref{prop : naturcomonad}
and Proposition \ref{prop : naturrestric} and it preserves strict finite products because both 2-functors $(\categ C,Q) \mapsto \catcog{\categ C} Q$ and
$(\categ C,Q) \mapsto \categ C$ do.
\end{proof}

Hence, any algebraic structure inside the 2-category $\categ{Comonads}$
may equivalently be described using forgetful functors from categories of coalgebras to the ground category.

\begin{corollary}
 The construction that sends a category with a monad
$(\categ C,M)$ to the functor $\catalg{\categ C} M \to \categ C$ canonically induces a 2-functor from $\Monads$ to $\Functors$ that is strictly fully faithful and that preserves strict finite products.
\end{corollary}

\subsection{Consequences}
One can draw several consequences from Theorem \ref{main}.

\subsubsection{Monoidal categories}

Given a monoidal category $(\categ C, \otimes, \II)$ and a comonad $(Q,w,n)$ (resp. a monad $(M,m,u)$),
there is a canonical bijection between
\begin{enumerate}
 \item the set of structures of a monoidal category on $Q$-algebras (resp. $M$-algebras)
 that lift that of $\categ C$ (that is the forgetful functor $\Cog_{\categ C}(Q)\to \categ C$ is strict monoidal);
 \item the set of structures of a Hopf comonad on $Q$ (resp. structures of a Hopf monad on $M$).
\end{enumerate}

Indeed, given a structure of a Hopf comonad on $Q$, the tensor product of two $Q$-coalgebras $V,W$ and the unit $\II$ inherit
structures of $Q$-coalgebras through the formulas
\begin{align*}
    V \otimes W &\to Q(V) \otimes Q(W) \to Q(V \otimes W);
    \II \to Q(\II).
\end{align*}
Conversely, from a structure of a monoidal category on $Q$-coalgebras that lifts that of $\categ C$, one obtain the structure of a Hopf comonad on $Q$ by lifting the natural map
 \[
 	Q(X) \otimes B(Y) \xrightarrow{\tau(X),\tau(Y)} X \otimes Y
 \]
to $Q(X \otimes Y)$

\subsubsection{Lax monoidal functors}

Now, let us consider Hopf comonads $Q$ and $O$ on monoidal categories respectively $\categ C$ and $\categ D$ and a lax monoidal functor $F: \categ C \to \categ D$. The two following assertions
are equivalent
\begin{enumerate}
 \item the natural  transformation $FQ \to OF$ is monoidal;
 \item the natural map in $\categ F$
\[
	F(V) \otimes F(W) \to  F(V \otimes W)
\]
induced by the structure of a lax monoidal functor on $F$
is a morphism of $O$-coalgebras for any two $Q$-coalgebras $V,W$.
\end{enumerate}
If these assertions are true, then the structure of a lax monoidal functor on $F: \categ C \to \categ D$
induces a structure of a lax monoidal functor on $F_{cog}: \catcog{\categ C} Q \to \catcog{\categ O} O$.

\subsubsection{Modules}

Let $(\categ C, Q)$ be monoidal category and a Hopf comonad
and let $R$ be a comonad on a category $\categ D$ tensored by $\categ C$.
Then, there is a canonical bijection
between
\begin{enumerate}
 \item the set of tensorisations of the category
of $R$-coalgebras by the monoidal category of $Q$-coalgebras that lifts the tensorisation
of $\categ D$ by $\categ C$;
\item the set of structures of a Hopf $Q$-module comonad
on $R$.
\end{enumerate}

Similarly, if $M$ is a monad on $\categ E$ which is cotensored by $\categ C$, then there is a canonical bijection
between
\begin{enumerate}
 \item the set of cotensorisations of the category
of $M$-algebras by the monoidal category of $Q$-coalgebras that lifts the cotensorisation
of $\categ E$ by $\categ C$;
\item the set of structures of a Hopf $Q$-module monad
on $M$.
\end{enumerate}

\subsection{The adjoint lifting theorem}
\label{sec adjiliftthm}
In this subsection, we recall the adjoint lifting theorem and its link with (op)lax (co)monad functors.

Let us consider an oplax comonad functor $(L,A): (\categ C,Q) \to (\categ D,O)$.
Let us assume that the functor $L$ has a right adjoint $R$. Then, the structure of an oplax comonad functor $LQ \to OL$ on $L$ induces
by doctrinal adjunction the structure of a lax comonad functor
on the right adjoint $R$
$$
    QR \to RLQR \to ROLR \to RO . 
$$
Thus, for any $O$-coalgebra $W$, let us consider the two following morphisms of $Q$-coalgebras from $L^Q R(W)$ to $L^QRO(W)$:
\begin{itemize}
    \itemt  on the one hand, the morphism induced by the map $W \to O(W)$;
    \itemt on the other hand, the composite morphism
    $$
    L^Q R(W) \to L^QQR(W) \to L^QRO(W).
    $$
\end{itemize}
This gives us a coreflexive pair of maps
\begin{equation}\label{eqadjlift}
    L^Q R(W) \rightrightarrows L^QRO(W)
\end{equation}
with common left inverse induced by the map $O(W) \to W$.

\begin{theorem}[Adjoint lifting theorem, \cite{Johnstone75}]\label{them adjliftthm}
 The functor $L_{cog}$ has a right adjoint if and only if the pair of maps just above in diagram \ref{eqadjlift} has an equaliser for any $O$-coalgebra $W$. Then, such a limit defines the value of this right adjoint functor on $W$.
\end{theorem}

\begin{proof}
 Straightforward.
\end{proof}

Besides, one can factorise the oplax comonad functor $(L,A)$ from $(\categ C, Q)$ to $(\categ D, O)$ as follows.

\begin{proposition}
 The endofunctor $LQR$ of $\categ D$ has the canonical structure of a comonad. Moreover, the oplax comonad functor $(L,A)$ factorises as
 $$
 (\categ C, Q) \xrightarrow{(L,A')} (\categ D, LQR)
 \xrightarrow{(\id,A'')} (\categ D, O)
 $$
 where $A'$ and $A''$ are respectively the natural maps
 \begin{align*}
     &LQ \xrightarrow{LQ\eta} LQRL ;
     \\
     &LQR \xrightarrow{A} OLR \xrightarrow{O\epsilon} O.
 \end{align*}
\end{proposition}

\begin{proof}
 The structure of a comonad on $LQR$ is given by the maps
 \begin{align*}
     &LQR \to LQQR \to LQRLQL;
     \\
     & LQR \to LR \to \id .
 \end{align*}
 Proving that these maps do define a comonad and the rest of the proposition follow from a straightforward checking.
\end{proof}

\subsection{More on comonad functors}

This subsection deals with the subset of oplax comonad morphisms spanned by strong morphisms.

\begin{definition}
A strong comonad functor between categories with comonads
is an oplax comonad functor $(F,A) : (\categ C, Q) \to (\categ D, R)$ so that $A$ is a natural isomorphism.
\end{definition}

\begin{remark}
 By the result of Kelly on doctrinal adjunctions, for any adjunction in the 2-category $\Comonads$, the right adjoint is a strong comonad functor.
\end{remark}

\begin{proposition}
An oplax comonad functor $(F,A) : (\categ C, Q) \to (\categ D, R)$ is a comonad functor if and only if the induced natural transformation
$$
    F_{cog} L^Q \to L^R F
$$
is an isomorphism.
\end{proposition}

\begin{proof}
This is a direct consequence of the way the morphisms
$F_{cog} L^Q \to L^R F$ and $FQ \to RF$ are related in the proof of Proposition \ref{prop : naturcomonad}.
\end{proof}

Let us consider an oplax comonad functor
$$
    (F,A) : (\categ C, Q) \to (\categ D,R) .
$$
We suppose that the categories $\categ C$ and $\categ D$
are complete and that the comonads $Q,R$
preserve coreflexive equalisers. Hence, the categories of coalgebras
over these comonads are complete (see Appendix \ref{appendixcomonads}).

\begin{proposition}\label{prop : Fcog preserves (co)limits}
Suppose that $F$ preserves limits and that $A$ is a natural isomorphism (hence, $(F,A)$ is a strong comonad functor). Then, $F_{cog}$ preserves limits.
\end{proposition}

\begin{proof}
Since $U_R \circ F_{cog} = F \circ U_Q$, $U_Q,F$ preserve coreflexive equalisers and $U_R$ create coreflexive equalisers, then $F_{cog}$ also preserve coreflexive equalisers.

Let us consider a family of $Q$-coalgebras $(V_i)_{i \in I}$ and the following diagram
$$
\begin{tikzcd}
     F_{cog}(\prod_i V_i)
     \ar[r] \ar[d]
     & \prod_i F_{cog}(V_i)
     \ar[d]
     \\
     F_{cog}L^Q (\prod_i U_Q(V_i))
     \ar[d, shift left] \ar[d, shift right]
     \ar[r]
     & L^R(\prod_i U_R F_{cog}(V_i))
     \ar[d, shift left] \ar[d, shift right]
     \\
     F_{cog}L^Q (\prod_i QU_Q(V_i))
     \ar[r]
     & L^R(\prod_i R U_R F_{cog}(V_i))
\end{tikzcd}
$$
which represents a natural transformation between the two vertical subdiagrams. The middle horizontal arrow and the bottom horizontal arrow decompose respectively as
\begin{align*}
     & F_{cog}L^Q (\prod_i U_Q(V_i))
     \to L^R F(\prod_i U_Q(V_i))
     \to L^R (\prod_i F U_Q(V_i))
     = L^R(\prod_i U_R F_{cog}(V_i)) ;
     \\
     &F_{cog}L^Q (\prod_i QU_Q(V_i))
     \to L^R F(\prod_i Q U_Q(V_i))
     \to L^R (\prod_i F Q U_Q(V_i))
     \to L^R (\prod_i R U_R F_{cog}(V_i)) ;
\end{align*}
all these maps are isomorphisms since $F$ preserves products and the morphism $FQ \to RF$ is an isomorphism. Hence, the middle horizontal arrow and the bottom horizontal arrow are also isomorphisms. Since the two vertical subdiagrams are limiting, then the top horizontal arrow is also an isomorphism.

To conclude, $F_{cog}$ preserves coreflexive equalisers and products and hence preserves all limits.
\end{proof}

\begin{corollary}
 Let us consider a lax monad functor between categories with monads $(G,B) : (\categ C, M) \to (\categ D, N)$ where $\categ C$ and $\categ D$ are cocomplete and $M$ and $N$ preserve reflexive coequalisers. If $F$ preserves colimits and $B$ is a natural isomorphism, then the induced functor between algebras $F_{alg}$ preserves colimits.
\end{corollary}


\section{Mapping coalgebras}

In this section, we use the adjoint lifting theorem to describe contexts where some categories of (co)algebras over a (co)monad are enriched tensored and cotensored over the category of coalgebras over another comonad.

\subsection{Pairing adjoints}

\subsubsection{The situation}

Let us consider a 2-pairing in the monoidal context of comonads
$$
    - \boxtimes - : (\categ C, Q) \times (\categ D, R)
    \to (\categ E, O).
$$
It is given by a bifunctor $ - \boxtimes - : \categ C \times \categ D \to \categ E$ together with a natural map
$Q(X) \boxtimes R(Y) \to O(X \boxtimes Y)$ that satisfy some commutation conditions with respect to the counits and decompositions of comonads. We know that such a natural map satisfying such conditions is equivalent to the data of a bifunctor
$$
    -\boxtimes_{cog}- : \catcog{C}{Q} \times \catcog{D}{R}
    \to \catcog{E}{O}
$$
that lifts $- \boxtimes -$.

Let us suppose that for any $X \in \categ C$, the functor $X \boxtimes - : \categ D \to \categ E$
has a right adjoint that is denoted
$$
    \langle -, X \rangle : \categ E \to \categ D.
$$
Then, by naturality, we obtain a bifunctor
$$
    \langle -, - \rangle: \categ E \times \categ C^{\op} \to \categ D;
$$
together with a natural isomorphism
$$
 \Hom{\categ E}{X \boxtimes Y}{Z}
 \simeq 
 \Hom{\categ D}{Y}{\langle Z, X\rangle}
$$
for any $(X,Y,Z) \in \categ C \times\categ D \times \categ E$.

\subsubsection{The adjoint}

Let $V$ be a $Q$-coalgebra. The functor $V \boxtimes_{cog} -$ from $R$-coalgebras to $O$-coalgebras lifts the functor $U_Q(V) \boxtimes -$ from $\categ D$ to $\categ E$. This corresponds to the structure of an oplax comonad functor on $U_Q(V) \boxtimes -$ with respect to $R$ and $O$
given by the map
$$
    U_Q(V) \boxtimes R(-)
    \to QU_Q(V) \boxtimes R(-)
    \to O( U_Q(V) \boxtimes -) .
$$
By doctrinal adjunction, the adjoint functor $\langle -,U_Q(V)\rangle$ inherits the structure of a lax comonad functor given by a natural map
$$
    R(\langle -,U_Q(V)\rangle) \to \langle O(-),U_Q(V)\rangle .
$$
(See Subsection \ref{sec adjiliftthm})
For any $O$-coalgebra Z,
let us consider the following two $R$-coalgebra morphisms from $L^R(\langle Z, V\rangle)$
to $L^R(O(\langle Z), V\rangle)$:
\begin{enumerate}
    \item on the one hand, the morphism induced by the structural morphism $Z \to O(Z)$;
    \item on the other hand the composite morphism
    $$
    L^R(\langle Z, V\rangle)
    \to L^RR(\langle Z, V\rangle)
    \to L^R(\langle O(Z), V\rangle).
    $$
\end{enumerate}
They also share a left inverse induced by the counit map $O(Z) \to Z$. We thus obtain a coreflexive pair of morphisms
\begin{equation}\label{eq1}
    L^R(\langle Z, V\rangle)
    \rightrightarrows
    L^R(\langle O(Z), V\rangle).
\end{equation}

\begin{proposition}
 The functor
 $$
    V \boxtimes_{cog} - : \catcog{D}{R}
    \to \catcog{E}{O}
 $$
 admits a right adjoint, that
 we denote $\langle -, V\rangle_R$,
 if and only if the category of $R$-coalgebras admits limits of the diagram (\ref{eq1}) for any $O$-coalgebra $Z$.
\end{proposition}

\begin{proof}
 This is an application of the adjoint lifting theorem (Theorem \ref{them adjliftthm}).
\end{proof}

\begin{corollary}
  If such an adjoint $\langle -, V\rangle_R$ exists
  for any $Q$-coalgebra $V$, then it yields a bifunctor
  $$
\langle - , - \rangle_R : \catcog{E}{O} \times \catcog{C}{Q}^{\op} \to \catcog{D}{R} .
  $$
  and a natural isomorphism
  $$
  \Hom{\catcog{E}{O}}{V \boxtimes_{cog} W}{ Z}
  \simeq \Hom{\catcog{D}{R}}{W}{\langle Z, V \rangle_R } .
  $$
\end{corollary}

\begin{corollary}
 If the category $\categ D$ admits coreflexive equalisers and if they are preserved by $R$, then the functor $V \boxtimes_{cog} -$ has a right adjoint for any $Q$-coalgebra $V$.
\end{corollary}

\begin{remark}
If the functor $- \boxtimes Y$ has a right adjoint
$[Y,-]$ for any $Y \in \categ D$, one gets the same phenomenon, that is the functor $- \boxtimes_{cog} W$ has a right adjoint if and only if equalisers of pairs of maps of the form
$$
    L^Q([W,Z]) \rightrightarrows L^Q([W,O(Z)])
$$
exist in $Q$-coalgebras. This just follows from considering the composite pairing
$$
    (\categ D,R) \times (\categ C, Q) \simeq 
    (\categ C, Q) \times (\categ D, R)
    \xrightarrow{-\boxtimes -} (\categ E,O)
$$
and applying the same results.
\end{remark}

\subsection{Enrichment}

Let $\categ C$ be a monoidal category and let $\categ D$ be a category enriched, tensored and cotensored over $\categ C$. Let us suppose that $\categ C$ have all coreflexive equalisers and that $\categ D$ have all coreflexive equalisers and all reflexive coequalisers.

Let us consider a Hopf comonad $Q$ on $\categ C$ that preserves coreflexive equalisers, a comonad $R$ on $\categ D$ that preserves coreflexive equalisers
and a monad $M$ on $\categ D$ that preserves reflexive coequalisers.

\begin{theorem}
 Given a tensorisation
 $$
 \catcog{C}{Q} \times \catcog{D}{R} \to \catcog{D}{R}
 $$
that lifts that of $\categ C$ on $\categ D$,
then $R$-coalgebras are enriched, tensored and cotensored over $Q$-coalgebras.
\end{theorem}

\begin{theorem}
 Given a cotensorisation
 $$
 \catalg{D}{M} \times \catcog{C}{Q}^{\op} \to \catcog{D}{M}
 $$
that lifts that of $\categ C$ on $\categ D$, then $M$-algebras are enriched, tensored and cotensored over $Q$-coalgebras.
\end{theorem}

\subsection{Pairing transfer}

Let us consider two 2-pairings in the monoidal context of comonads
\begin{align*}
   - \boxtimes - : (\categ C, Q) \times (\categ D, R)
    &\to (\categ E, O)
    \\
   - \boxtimes' - : (\categ C', Q') \times (\categ D', R')
    &\to (\categ E', O')
\end{align*}
together with a lax morphism of pairings, that is the data of functors
\begin{align*}
    F_C &: \categ C \to \categ C' ;
    \\
    F_D &: \categ D \to \categ D' ;
    \\
    F_E &: \categ E \to \categ E' ;
\end{align*}
that are lifted to the level of coalgebras by functors respectively denoted $F_{C,cog}, F_{D,cog}$ and $F_{E,cog}$
and
together with a natural morphism
$$
    F_E (X \boxtimes Y) \to
    F_C X \boxtimes' F_D Y
$$
that also lifts to the level of coalgebras. 

Let us suppose that the functors 
\begin{align*}
    & X \boxtimes - : \categ D \to \categ E ;
    \quad X' \boxtimes' - : \categ D' \to \categ E' ;
    \\
    &V \boxtimes_{cog} - : \catcog{\categ D}{R} \to \catcog{\categ E}{O};
    \quad V' \boxtimes'_{cog} - : \catcog{\categ D'}{R'} \to \catcog{\categ E'}{O'};
\end{align*}
all have right adjoints for any objects $X,X' \in \categ C \times \categ C'$, any $Q$-coalgebras $V$ and any $Q'$-coalgebra $V'$. We denote these right adjoints $\langle -, X\rangle$, $\langle -, X'\rangle'$, $\langle - , V \rangle_R$ and $\langle - , V' \rangle_{R'}$.

\begin{proposition}\label{prop transferleft}
 Let us suppose that
 \begin{itemize}
     \itemt the natural transformations $F_D R \to R' F_D$ and
     $F_E O \to O' F_E$ are isomorphisms;
     \itemt the canonical morphism
 $$
    F_D(\langle Z, X \rangle) \to
    \langle F_E(Z), F_C(X) \rangle'
 $$
 is an isomorphism for any objects $X,Z \in \categ C \times \categ E$, ;
 \itemt the image through the functor $F_{D,cog}$ of the limiting diagram
 $$
  \langle Z, V\rangle_R \to \langle L^O Z , V\rangle_R \rightrightarrows \langle L^O O Z , V\rangle_R
 $$
 is limiting for any $O$-coalgebra $Z$ and any $Q$-coalgebra $V$.
 \end{itemize}
 Then, for any two coalgebras $Z, V \in \catcog{E}{O} \times \catcog{C}{Q}$, the
 canonical morphism
 $$
    F_{D,cog} (\langle Z, V \rangle_R) \to
    \langle F_{E,cog}(Z), F_{C,cog}(V) \rangle_{R'}
 $$
 is an isomorphism.
\end{proposition}

\begin{proof}
 First, let us prove the result in the case where $Z$ is cofree, that is
 $$
    Z = L^O K.
 $$
 From the lifting property of the functor $F_C,F_D,F_E$ and of the natural transformation $
F_E (X \boxtimes Y) \to
F_C X \boxtimes' F_D Y
$ we get a commutative diagram of functors
$$
\begin{tikzcd}
     U_{O'} \circ (F_{C,cog} (V)
     \boxtimes'_{cog} -) \circ F_{D,cog}
     \ar[r, equal] \ar[d]
     & (F_{C}U_Q(V) \boxtimes' -) \circ U_{R'} \circ  F_{D,cog}
     \ar[d, equal]
     \\
     U_{O'} \circ F_{E,cog} \circ
     (V \boxtimes_{cog} -)
     \ar[d, equal]
     & (F_{C}U_Q(V) \boxtimes' -)  \circ
     F_{R} 
     \circ U_{R}
     \ar[d]
     \\
     F_{E} \circ U_{O} \circ
     (V \boxtimes_{cog} -)
     \ar[r, equal]
     & F_{E}  \circ
     (U_Q(V) \boxtimes -)
     \circ U_{R} .
\end{tikzcd}
$$
By adjunction, we thus get a commutative diagram of functors
$$
\begin{tikzcd}
     F_{D,cog} \circ L^R \circ \langle - , U_Q(V)\rangle
     \ar[r] \ar[d]
     & F_{D,cog} \circ \langle - , V\rangle_R \circ L^O
     \ar[d]
     \\
     L^{R'} \circ F_{D} \circ \langle - , U_Q(V)\rangle
     \ar[d]
     & \langle - , F_{C,cog}(V)\rangle_{R'} \circ F_{E,cog} \circ L^O
     \ar[d]
     \\
     L^{R'} \circ \langle - ,F_C U_Q(V)\rangle' \circ F_{E}
     \ar[r]
     & \langle - , F_{C,cog}(V)\rangle_{R'} \circ L^{O'} \circ F_{E}
\end{tikzcd}
$$
 The horizontal arrows are isomorphisms since for any isomorphism of left adjoint functors the induced morphism of right adjoint functors is also an isomorphism. Then, by the hypothesis, the two left vertical arrows and the bottom right vertical arrow are all isomorphisms.
 Hence the map
 $$
 F_{D,cog} \circ \langle - , V\rangle_R \circ L^O
     \to
     \langle - , F_{C,cog}(V)\rangle_{R'} \circ F_{E,cog} \circ L^O
 $$
 is also an isomorphism. This proves the result for $Z = L^O (K)$.
 
 In the general case, let us consider the following diagram
 $$
 \begin{tikzcd}
      F_{D,cog} (\langle L^OO(Z), V \rangle_R)
      \ar[r]
      & \langle F_{E,cog} L^OO(Z), F_{C,cog}(V) \rangle_{R'}
      \\
      F_{D,cog} (\langle L^O(Z), V \rangle_R)
      \ar[r] \ar[u, shift left]
      \ar[u, shift right]
      & \langle F_{E,cog} L^O(Z), F_{C,cog}(V) \rangle_{R'}
      \ar[u, shift left]
      \ar[u, shift right]
      \\
      F_{D,cog} (\langle Z, V \rangle_R)
      \ar[r] \ar[u]
      & \langle F_{E,cog}(Z), F_{C,cog}(V) \rangle_{R'} .
      \ar[u]
 \end{tikzcd}
 $$
 The left vertical part is limiting as well as the left vertical part (by hypothesis). Moreover, the two first horizontal arrows are isomorphisms. Hence, the bottom horizontal map is also an isomorphism. This proves the result.
\end{proof}

\begin{remark}
The third condition of Proposition \ref{prop transferleft} is true if in particular, the categories $\categ D$ and $\categ D'$ have coreflexive equalisers that are preserved by $F_D$, $R$ and $R'$.
\end{remark}

\begin{remark}
One has the same result for right adjoints of the functors $- \boxtimes Y$ and $- \boxtimes_{cog} W$. It suffices to apply the result to the opposite pairing
$$
 Y \boxtimes^{\op} X = X \boxtimes Y .
$$
\end{remark}

\subsection{The example of chain complexes}

Let $\categ{Ch}$ be the category of chain complexes of modules over a ring $\mathbb K$ and let $\categ{Mod}_{\mathbb K}^{gr}$ be the category of $\mathbb Z$-graded $\mbk$-modules. Let us denote by $U_d$ the forgetful functor from chain complexes to graded modules. Let us notice that the categories $\categ{Ch}$ and $\categ{Mod}_{\mathbb K}^{gr}$
are complete and cocomplete and that $U_d$ preserves limits and colimits.

\begin{definition}
We call a monad $M$ (resp. a comonad $Q$) on chain complexes "computed on graded modules" if there exists a monad $M^{gr}$ (resp. a comonad $Q^{gr}$) on graded modules so that we have equalities
\begin{align*}
    U_d \circ M &= M^{gr} \circ U_d
    \\
    U_d \circ Q &= Q^{gr} \circ U_d
\end{align*}
that are consistent with respect to the structural morphisms of monads and comonads. This determines uniquely $M^{gr}$ and $Q^{gr}$.
\end{definition}

The following diagrams of categories are commutative
$$
\begin{tikzcd}
     \categ{Ch} \times \categ{Ch}
     \ar[r, "\otimes"] \ar[d]
     & \categ{Ch} \ar[d]
     \\
     \categ{Mod}_{\mathbb K}^{gr} \times \categ{Mod}_{\mathbb K}^{gr}
     \ar[r, swap, "\otimes"]
     & \categ{Mod}_{\mathbb K}^{gr} ;
\end{tikzcd}
\quad\begin{tikzcd}
     \categ{Ch}^{\op} \times \categ{Ch}
     \ar[r,"{[-,-]}"]
     \ar[d]
     & \categ{Ch}
     \ar[d]
     \\
     (\categ{Mod}_{\mathbb K}^{gr})^{\op} \times \categ{Mod}_{\mathbb K}^{gr}
     \ar[r,"{[-,-]}"]
     & \categ{Mod}_{\mathbb K}^{gr} .
\end{tikzcd}
$$
Moreover, the functors $X\otimes -$, $-\otimes X$, have the same right adjoint given by $[X,-]$ and the functor $[-,X]$ from $\categ{Ch}^{\op}$ to $\categ{Ch}$ has a left adjoint given actually by the same formula $[-,X]$.

Let us consider three comonads on chain complexes $Q,R,O$ and two monads $M,N$. We suppose that these monads and comonads are computed on graded modules, that the comonads preserve coreflexive equalisers and that the monads preserve reflexive coequalisers.

\subsubsection{Coalgebras in chain complexes}

Let us consider a bifunctor
$$
-\otimes - :\catcog{\categ{Ch}} Q \times \catcog{\categ{Ch}} R \to \catcog{\categ{Ch}} O
$$
that lifts the tensor product of chain complexes. Since these comonads are computed at the level of graded modules, the bifunctor between categories of coalgebras described above lifts another bifunctor
$$
-\otimes - :\catcog{\categ{Mod}_{\mathbb K}^{gr}}{Q^{gr}} \times \catcog{\categ{Mod}_{\mathbb K}^{gr}}{R^{gr}} \to \catcog{\categ{Mod}_{\mathbb K}^{gr}}{O^{gr}} ,
$$
that lifts itself the tensor product of graded modules.
By the adjoint lifting theorem as used in this section, we obtain four bifunctors
\begin{align*}
    &\langle -, - \rangle : \catcog{\categ{Ch}} O \times 
    \catcog{\categ{Ch}} Q^{\op}
    \to
    \catcog{\categ{Ch}} R;
    \\
    &\{ -, - \} : \catcog{\categ{Ch}} R^{\op} \times 
    \catcog{\categ{Ch}} O
    \to
    \catcog{\categ{Ch}} Q;
    \\
    &\langle -, - \rangle_{gr} : \catcog{\categ{Mod}_{\mathbb K}^{gr}}{O^{gr}}
    \times 
    \catcog{\categ{Mod}_{\mathbb K}^{gr}}{Q^{gr}}^{\op}
    \to
    \catcog{\categ{Mod}_{\mathbb K}^{gr}}{R^{gr}};
    \\
    &\{ -, - \}_{gr} : \catcog{\categ{Mod}_{\mathbb K}^{gr}}{R^{gr}}^{\op} \times 
    \catcog{\categ{Mod}_{\mathbb K}^{gr}}{R^{gr}}
    \to
    \catcog{\categ{Mod}_{\mathbb K}^{gr}} Q^{gr};
\end{align*}
together with natural isomorphisms
\begin{align*}
    &\Hom{}{W}{\langle Z, V \rangle }
    \simeq \Hom{}{V\otimes W}{Z}
    \simeq
    \Hom{}{V}{\{ W,Z \} } ;
    \\
    &\Hom{}{W'}{\langle Z', V' \rangle_{gr} }
    \simeq \Hom{}{V'\otimes W'}{Z'}
    \simeq
    \Hom{}{V'}{\{ W',Z' \}_{gr} } ;
\end{align*}
for any $Q,R,O,Q^{gr},R^{gr},O^{gr}$-coalgebras respectively $V,W,Z,V',W',Z'$.
From Proposition \ref{prop transferleft}, we get that the natural morphisms
\begin{align*}
    &U_{d,cog} (\langle Z, V\rangle) \to \langle U_{d,cog}(Z), U_{d,cog}(V)\rangle_{gr}
    \\
    &U_{d,cog}( \{ W, Z\}) \to \{ U_{d,cog}(W), U_{d,cog}(Z)\}_{gr}
\end{align*}
are isomorphisms.

\subsubsection{Algebras in chain complexes}

Let us consider a bifunctor
$$
[-,-] :\catcog{\categ{Ch}}{Q}^{\op} \times \catalg{\categ{Ch}} M \to \catalg{\categ{Ch}} N
$$
that lifts the internal hom of chain complexes. Since $M,N,Q$ are computed at the level of graded modules, the bifunctor between categories of algebras and coalgebras described above lifts another bifunctor
$$
[-,-] :\catcog{\categ{Mod}_{\mathbb K}^{gr}}{Q^{gr}}^{\op} \times \catalg{\categ{Mod}_{\mathbb K}^{gr}}{M^{gr}} \to \catcog{\categ{Mod}_{\mathbb K}^{gr}}{N^{gr}} ,
$$
that lifts itself the internal hom of graded modules.
Thus, from the adjoint lifting theorem, we obtain four bifunctors
\begin{align*}
    & -\boxtimes - : \catcog{\categ{Ch}} Q \times 
    \catcog{\categ{Ch}} N
    \to
    \catalg{\categ{Ch}} M;
    \\
    &\{ -, - \} : \catalg{\categ{Ch}} N^{\op} \times 
    \catalg{\categ{Ch}} M
    \to
    \catcog{\categ{Ch}} Q;
    \\
    & -\boxtimes_{gr} - : : \catcog{\categ{Mod}_{\mathbb K}^{gr}}{Q^{gr}}
    \times
    \catalg{\categ{Mod}_{\mathbb K}^{gr}}{N^{gr}}
    \to
    \catcog{\categ{Mod}_{\mathbb K}^{gr}}{M^{gr}};
    \\
    &\{ -, - \}_{gr} : \catalg{\categ{Mod}_{\mathbb K}^{gr}}{N^{gr}}^{\op}
    \times 
    \catalg{\categ{Mod}_{\mathbb K}^{gr}}{M^{gr}}
    \to
    \catcog{\categ{Mod}_{\mathbb K}^{gr}}{Q^{gr}};
\end{align*}
together with natural isomorphisms
\begin{align*}
    &\Hom{}{V \boxtimes B}{A}
    \simeq \Hom{}{B}{[V,A]}
    \simeq
    \Hom{}{V}{\{ B,A \} } ;
    \\
    &\Hom{}{V' \boxtimes_{gr} B'}{A'}
    \simeq \Hom{}{B'}{[V',A']}
    \simeq
    \Hom{}{V'}{\{ B',A' \}_{gr} } ;
\end{align*}
for any $Q,Q^{gr}$-coalgebras $V,V'$ and any $M, N,M^{gr}, N^{gr}$-algebras $A,B,A',B'$.
From Proposition \ref{prop transferleft}, we get that the natural morphisms
\begin{align*}
    & U_{d,cog}(V) \boxtimes_{gr} U_{d,alg}(B) \to  U_{d,alg}(V \boxtimes B)
    \\
    &U_{d,cog}( \{ B, A\}) \to \{ U_{d,alg}(B), U_{d,alg}(A)\}_{gr}
\end{align*}
are isomorphisms.


\appendix

\section{Symmetric monoidal categories}

In this first appendix, we just recall the main definitions related to symmetric monoidal categories.

\subsection{Monoidal categories}

\begin{definition}
A monoidal category is the data of a category $\categ C$ equipped with a bifunctor
 $$
 - \otimes - : \categ C \times \categ C \to \categ C
 $$
 an object $\II$ and natural isomorphisms
 $$
 X \otimes (Y \otimes Z) \simeq (X \otimes Y) \otimes Z,\quad \II \otimes X \simeq X,\quad X \otimes \II \simeq X
 $$
 called respectively the associator, the left unitor and the right unitor and so that the following diagrams commute
 $$
 \begin{tikzcd}
      ((U \otimes X ) \otimes Y ) \otimes Z
      \ar[r] \ar[d]
      & (U \otimes X ) \otimes (Y  \otimes Z)
      \ar[dd]
      \\
      (U \otimes (X  \otimes Y )) \otimes Z
      \ar[d]
      \\
      U \otimes ((X  \otimes Y ) \otimes Z)
      \ar[r]
      & U \otimes (X  \otimes (Y  \otimes Z))
 \end{tikzcd}
 \quad
 \begin{tikzcd}
      (X \otimes \II) \otimes Y
      \ar[r] \ar[rd]
      & X \otimes (\II \otimes Y)
      \ar[d]
      \\
      & X \otimes Y .
 \end{tikzcd}
 $$
\end{definition}

\begin{definition}
 Given two monoidal categories $\categ C, \categ D$, a lax monoidal functor from $\categ C$ to $\categ D$ is the data of a functor $f : \categ C \to \categ D$ together with a natural transformation
 $$
 f(X) \otimes f(Y) \to f(X \otimes Y)
 $$
 and a map $\II \to f(\II)$ so that the following diagrams commute
 $$
 \begin{tikzcd}
      (f(X) \otimes f(Y)) \otimes f(Z)
      \ar[r] \ar[d]
      & f(X) \otimes (f(Y) \otimes f(Z))
      \ar[d]
      \\
      f(X\otimes Y) \otimes f(Z)
      \ar[d]
      & f(X) \otimes f(Y \otimes Z)
      \ar[d]
      \\
      f((X\otimes Y) \otimes Z)
      \ar[r]
      & f(X \otimes (Y \otimes Z))
 \end{tikzcd}
 $$
 $$
 \begin{tikzcd}
      \II \otimes f(X)
      \ar[r] \ar[d]
      & f(\II) \otimes f(X)
      \ar[d]
      \\
      f(X)
      \ar[r]
      & f(\II \otimes X) ,
 \end{tikzcd}
 \quad
 \begin{tikzcd}
      f(X) \otimes \II
      \ar[r] \ar[d]
      & f(X) \otimes f(\II)
      \ar[d]
      \\
      f(X)
      \ar[r]
      & f(X \otimes \II) .
 \end{tikzcd}
 $$
\end{definition}

\begin{definition}
 Given lax monoidal functors between monoidal categories $f,g : \categ C \to \categ D$, a monoidal natural transformation between them is a natural transformation between the functor $f,g$ so that the following diagrams commute
 $$
 \begin{tikzcd}
      f(X) \otimes f(Y)
      \ar[r] \ar[d]
      & f(X \otimes Y)
      \ar[d]
      \\
      g(X) \otimes g(Y)
      \ar[r]
      & g(X \otimes Y)
 \end{tikzcd}
 \quad
 \begin{tikzcd}
      \II
      \ar[r] \ar[rd]
      &f(\II)
      \ar[d]
      \\
      & g(\II).
 \end{tikzcd}
 $$
\end{definition}

\subsection{Symmetric monoidal categories}

\begin{definition}
 A symmetric monoidal category is the data of a category $\categ C$ equipped with a bifunctor
 $$
 - \otimes - : \categ C \times \categ C \to \categ C
 $$
 an object $\II$ and natural isomorphisms
 $$
 X \otimes (Y \otimes Z) \simeq (X \otimes Y) \otimes Z,\quad X \otimes Y \simeq Y \otimes X,\quad \II \otimes X \simeq X,\quad X \otimes \II \simeq X
 $$
 called respectively the associator, the commutator, the left unitor and the right unitor and so that the following diagrams commute
 $$
 \begin{tikzcd}
      ((U \otimes X ) \otimes Y ) \otimes Z
      \ar[r] \ar[d]
      & (U \otimes X ) \otimes (Y  \otimes Z)
      \ar[dd]
      \\
      (U \otimes (X  \otimes Y )) \otimes Z
      \ar[d]
      \\
      U \otimes ((X  \otimes Y ) \otimes Z)
      \ar[r]
      & U \otimes (X  \otimes (Y  \otimes Z))
 \end{tikzcd}
 \quad
 \begin{tikzcd}
      ( X \otimes Y ) \otimes Z
      \ar[r] \ar[d]
      & ( Y \otimes X ) \otimes Z
      \ar[d]
      \\
      Z \otimes ( X \otimes Y )
      \ar[d]
      & Y \otimes ( X \otimes Z )
      \ar[d]
      \\
      ( Z \otimes X ) \otimes Y
      \ar[r]
      & Y \otimes ( Z \otimes X )
 \end{tikzcd}
 $$
 $$
 \begin{tikzcd}
      X \otimes Y
      \ar[r] \ar[rd, equal]
      & Y \otimes X
      \ar[d]
      \\
      & X \otimes Y
 \end{tikzcd}
 \quad
 \begin{tikzcd}
      \II \otimes X
      \ar[r] \ar[rd]
      & X \otimes \II
      \ar[d]
      \\
      & X
 \end{tikzcd}
 \quad
 \begin{tikzcd}
      (X \otimes \II) \otimes Y
      \ar[r] \ar[rd]
      & X \otimes (\II \otimes Y)
      \ar[d]
      \\
      & X \otimes Y .
 \end{tikzcd}
 $$
\end{definition}

\begin{definition}
 Given two symmetric monoidal categories $\categ C, \categ D$, a lax symmetric monoidal functor from $\categ C$ to $\categ D$ is the data of a functor $f : \categ C \to \categ D$ together with a natural transformation
 $$
 f(X) \otimes f(Y) \to f(X \otimes Y)
 $$
 and a map $\II \to f(\II)$ so that the following diagrams commute
 $$
 \begin{tikzcd}
      (f(X) \otimes f(Y)) \otimes f(Z)
      \ar[r] \ar[d]
      & f(X) \otimes (f(Y) \otimes f(Z))
      \ar[d]
      \\
      f(X\otimes Y) \otimes f(Z)
      \ar[d]
      & f(X) \otimes f(Y \otimes Z)
      \ar[d]
      \\
      f((X\otimes Y) \otimes Z)
      \ar[r]
      & f(X \otimes (Y \otimes Z))
 \end{tikzcd}
 $$
 $$
 \begin{tikzcd}
      f(X) \otimes f(Y)
      \ar[r] \ar[d]
      & f(Y) \otimes f(X)
      \ar[d]
      \\
      f(X\otimes Y)
      \ar[r]
      & f(Y \otimes X)
 \end{tikzcd}
 \quad
 \begin{tikzcd}
      f(X) \otimes \II
      \ar[r] \ar[d]
      & f(X) \otimes f(\II)
      \ar[d]
      \\
      f(X)
      \ar[r]
      & f(X \otimes \II) .
 \end{tikzcd}
 $$
\end{definition}

\begin{definition}
 Given lax symmetric monoidal functors between symmetric monoidal categories $f,g : \categ C \to \categ D$, a monoidal natural transformation between them is a natural transformation between the functor $f,g$ so that the following diagrams commute
 $$
 \begin{tikzcd}
      f(X) \otimes f(Y)
      \ar[r] \ar[d]
      & f(X \otimes Y)
      \ar[d]
      \\
      g(X) \otimes g(Y)
      \ar[r]
      & g(X \otimes Y)
 \end{tikzcd}
 \quad
 \begin{tikzcd}
      \II
      \ar[r] \ar[rd]
      &f(\II)
      \ar[d]
      \\
      & g(\II).
 \end{tikzcd}
 $$
\end{definition}

\section{Category enriched tensored and cotensored over a monoidal category}

In this appendix, we recall some notions related to categories enriched, tensored and cotensored over a monoidal category.

\subsection{Category tensored over a monoidal category}

\begin{definition}
A category $\categ C$ is said to be tensored over a monoidal category $(\categ{E}, \otimes, \II)$
if it is equipped with the structure of a lax $\categ E$-module, that is there exists an
oplax monoidal functor
 \[
 	\categ E \to \Fun{}{\categ C}{\categ C}.
 \]
Equivalently, there exists
a bifunctor $
 - \boxtimes - :  \categ{E} \times \categ{C} \ra \categ{C}$
 together with natural transformations
\begin{align*}
 	\II \boxtimes X &\to X \ ,\\
	\left({\coalgebra V}' \otimes {\coalgebra V}\right)\boxtimes X &\to {\coalgebra V}' \boxtimes \left( {\coalgebra V} \boxtimes X\right) \ ,
\end{align*}
that satisfy coherences.
\end{definition}

Let $\categ E$ and $\categ F$ be two monoidal categories and let $G : \categ E \to \categ F$ be a lax monoidal functor.

\begin{definition}
Let $\categ C$ and $\categ D$ be respectively a lax $\categ E$-module a lax $\categ F$-module.
A $G$-strength on $F$ is the additional structure of
a lax morphism of lax modules that is a natural morphism $G(\coalgebra V) \boxtimes F( X) \to F(\coalgebra V \boxtimes X)$
 for any $\coalgebra V \in \categ E$ and any $X \in \categ C$,
 so that the following diagrams commute
 \[
\begin{tikzcd}
	(G(\coalgebra V) \otimes G(\coalgebra V')) \boxtimes  F(X)
	\arrow[r] \arrow[d]
	&
	G(\coalgebra V) \boxtimes (G(\coalgebra V') \boxtimes F(X))
	\arrow[r]
	&
	G(\coalgebra V) \boxtimes F(\coalgebra V' \boxtimes X)
	\arrow[d]
	\\
	G(\coalgebra V \otimes \coalgebra V') \boxtimes  F(X)
	\arrow[r]
	&
	F((\coalgebra V \otimes \coalgebra V') \boxtimes X)
	\arrow[r]
	& F(\coalgebra V \boxtimes (\coalgebra V' \boxtimes X)) ;
\end{tikzcd}
 \]
 \[
 \begin{tikzcd}
	{\II \boxtimes F(X )}
	\arrow[r] \arrow[d]
	& {G(\II) \boxtimes F (X)}
	\arrow[d]
	\\
	F(X)
	\arrow[r]
	& F(\II \boxtimes X) .
\end{tikzcd}
 \]
 In the case where all the structural morphisms $G(\coalgebra V) \boxtimes F( X) \to F(\coalgebra V \boxtimes X)$ are isomorphisms,
 one talks about $G$-tensorial strength isomorphism. In the case where $G$ is the identity functor of $\categ E$, one talks about a $\categ E$-tensorial strength for $F$.
\end{definition}

Let us consider two lax monoidal functors $G,G': \categ E \to \categ F$ between monoidal categories, a monoidal natural transformation $G \to G'$, two categories
$\categ C$ and $\categ D$ tensored over respectively $\categ E$ and $\categ F$ and two functors $F,F': \categ C \to \categ D$. Let us suppose that
$F$ is equipped with a $G$-strength and that $F'$ is equipped with a $G'$-strength.

\begin{definition}\label{definistrongnat}
A strong natural transformation from $F$ to $F'$, with respect to their strength and with respect to the monoidal natural transfomation $G \to G'$ is a natural transformation
so that the following diagram commutes
 \[
\begin{tikzcd}
	G(\coalgebra V) \boxtimes F(X)
	\arrow[r] \arrow[d]
	& F(\coalgebra V\boxtimes X)
	\arrow[d]
	\\
	G'(\coalgebra V) \boxtimes F'(X)
	\arrow[r]
	&F'(\coalgebra V\boxtimes X) 
\end{tikzcd}
 \]
for any $X \in \categ C, \coalgebra V \in \categ E$.
\end{definition}

 \subsection{Tensorisation, cotensorisation and enrichment}

\begin{definition}
 For any monoidal category $(\categ{E}, \otimes, \mathbb 1)$, we denote by $\categ E^{tr} = (\categ{E}, \otimes, \mathbb 1)$ the transposed monoidal category of $\categ E$, that is the same category with the opposite monoidal structure
 \[
 	X \otimes^{tr} Y \coloneqq Y \otimes X.
 \]
\end{definition}

\begin{definition}[Category cotensored over a monoidal category]\label{def:tce}
Let $(\categ{E}, \otimes, \mathbb 1)$ be a monoidal category and let $\categ{C}$ be a category.
We say that $\categ{C}$ is \textit{cotensored} over $\categ{E}$ if there exists a bifunctor 
 \[
 	\langle - , - \rangle :  \categ{C} \times \categ{E}^{\op} \ra \categ{C}\\	
 \]
 together with natural transformations
\begin{align*}
 	X &\to \langle X , \II \rangle\ ,\\
	 \langle \langle X,  {\coalgebra V} \rangle , {\coalgebra V}' \rangle &\to\langle X , {\coalgebra V} \otimes {\coalgebra V}' \rangle\ ,
\end{align*}
that makes the category $\categ C$ tensored over $\categ E^{optr} = (\categ E^\op, \otimes^{tr}, \II)$.
\end{definition}
 
\begin{definition}
 A category enriched over a monoidal category $\categ E$, (or $\categ E$-category) $\categ C$ is the data of a set (possibly large)
 called the set of objects or the set of colours, an object $\{X,Y\}\in \categ E$ for any two colours $X,Y$ and morphisms
\begin{align*}
 	\gamma:\{Y,X\} \otimes \{X,Y\} &\to \{X,Z\};
	\\
	\eta_X: \II & \to \{X,X\}
\end{align*}
for any three colours $X,Y,Z$ that define a unital associative composition. Moreover, a morphism of $\categ E$-categories
 from $\categ C$ to $\categ D$ is the data of a function on colours $\phi$ and morphisms
 \[
 	\{X,Y\} \to \{\phi(X), \phi(Y)\}
 \]
 for any two colours $X,Y$ of $\categ C$ that commute with units and compositions. This defines the category $\categ{Cat}_{\categ E}$ of $\categ E$-categories.
\end{definition}

\begin{definition}\label{definitiontce}
A category \textit{tensored-cotensored-enriched} (TCE for short) over $\categ{E}$ is the data of a category $\categ{C}$ equipped with three bifunctors:
  $$
\begin{cases}
- \boxtimes -: \categ E \times \categ C \to \categ C\\
 \{-,-\}: \categ{C}^{\op} \times \categ{C} \to \categ{E}\\
 \langle -, - \rangle: \categ{C} \times \categ{E}^{\op} \to \categ{C} ,
\end{cases}
  $$
 with natural isomorphisms,
 \[
 \hom_\categ{C} ({\coalgebra V} \boxtimes X , Y) \simeq \hom_\categ{E} ({\coalgebra V}, \{X,Y\}) \simeq \hom_\categ{C} (X, \langle Y,  {\coalgebra V} \rangle)\ ,
 \]
 and with 
\begin{itemize}
 \itemt  a structure of a tensorisation on $-\boxtimes -$,
 \itemt or equivalently a structure of a cotensorisation on $\langle -, - \rangle$,
 \itemt or equivalently a structure of an enrichment on $\{-,-\}$.
\end{itemize}
\end{definition}

Given a tensorisation $-\boxtimes -$, the composition and the unit on $\{-,-\}$ are the adjoints maps of the morphisms
\begin{align*}
	(\{Y,Z\} \otimes  \{ X, Y\}) \boxtimes X\to \{Y,Z\} \boxtimes ( \{ X, Y\} \boxtimes X) \to \{Y,Z\} \boxtimes Y \to Z\ ,
	\\
	\II \boxtimes X \to X .
\end{align*}
Given a cotensorisation $\langle-, -\rangle$, the composition and the unit on $\{-,-\}$ are the adjoints maps of the morphisms of the map
\begin{align*}
	X \to \langle Y, \{X,Y\}\rangle \to \langle \langle Z,  \{Y,Z\} \rangle, \{X,Y\}\rangle
	\to \langle Y, \{Y,Z\} \otimes  \{X,Y\}\rangle\ ,
	\\
	X \to \langle X, \II\rangle .
\end{align*}
Conversely, given an enrichment $\{-,-\}$, the structural tensorisation morphisms for $- \boxtimes -$ are adjoints to the maps
\begin{align*}
	\coalgebra V \otimes \coalgebra V'
	\to \{\coalgebra V' \boxtimes X, \coalgebra V\boxtimes (\coalgebra V' \boxtimes  X)\}
	\otimes \{X, \coalgebra V' \boxtimes X\}
	\to \{X, \coalgebra V\boxtimes (\coalgebra V' \boxtimes  X)\}\ ,
	\\
	\II \to \{X,X\},
\end{align*}
and the structural cotensorisation morphisms for $\langle-, -\rangle$ are adjoints to the maps
\begin{align*}
	\coalgebra V \otimes \coalgebra V'
	\to \{ \langle X, \coalgebra V\rangle, X\} 
	\otimes \{\langle \langle X,\coalgebra V\rangle, \coalgebra V'\rangle, \langle X, \coalgebra V\rangle\}
	\to \{\langle \langle X,\coalgebra V\rangle, \coalgebra V'\rangle, X\}\ ,
	\\
	\II \to \{X,X\}.
\end{align*}

\begin{remark}\label{lemmachangeenrich}
Given an adjunction between monoidal categories
\[
\begin{tikzcd}
 	\categ E
	\arrow[rr, shift left,"L"]
	&&
	 \categ F
	\arrow[ll, shift left,"R"]
\end{tikzcd}
\]
where $L$ is monoidal (and then $R$ is lax monoidal), then
any category $\categ C$ TCE over $\categ F$
 is also TCE over $\categ E$. 
\end{remark}

\begin{definition}
A biclosed monoidal category is a monoidal category $(\categ{E}, \otimes, \mathbb 1)$ TCE over itself in a way so that the tensorisation
is the tensor product of $\categ E$. If it is symmetric, one says simply that it is closed.
\end{definition}

\begin{notation}
When dealing with a biclosed (or a closed) monoidal category, the enrichment will usually be denoted $[-,-]$. 
\end{notation}

 \subsection{Strength in the context of categories enriched-tensored and cotensored}
\label{sectionstrength}

 Let $G : \categ E \to \categ F$ be
lax monoidal functor between monoidal categories.  Moreover, let $\categ C$ and $\categ D$ be two categories TCE respectively over $\categ E$ and $\categ F$ and let $F: \categ C \to \categ D$ be a functor.

A $G$-strength on $F$ is equivalent to the data of natural morphisms $F(\langle X, \coalgebra V\rangle) \to \langle F(X), G(\coalgebra V)\rangle$ for any $\coalgebra V \in \categ E$ and any $X \in \categ C$,
 so that the following diagrams commute
 \[
\begin{tikzcd}
	F(\langle \langle X, \coalgebra V\rangle, \coalgebra W \rangle)
	\arrow[r] \arrow[d]
	&\langle F(\langle X, \coalgebra V\rangle), G(\coalgebra W) \rangle
	\arrow[r]
	&
	\langle \langle F(X), G(\coalgebra V)\rangle, G(\coalgebra W) \rangle
	\arrow[d]
	\\
	F(\langle X, \coalgebra V\otimes \coalgebra W \rangle)
	\arrow[r]
	&
	\langle F(X), G(\coalgebra V\otimes \coalgebra W) \rangle
	\arrow[r]
	&  \langle F(X), G(\coalgebra V)\otimes G(\coalgebra W) \rangle;
\end{tikzcd}
 \]
 \[
 \begin{tikzcd}
	F(\langle X, \II \rangle)
	\arrow[r] \arrow[d]
	& \langle F(X), G(\II)\rangle
	\arrow[d]
	\\
	F(X)
	\arrow[r]
	& \langle F(X), \II \rangle.
\end{tikzcd}
 \]
 It is also equivalent to the data of natural morphisms $G\{X,Y\} \to \{FX,FY\}$ for any $X,Y \in \categ C$,
 so that the following diagrams commute
 \[
\begin{tikzcd}
 	G(\{Y,Z\})\otimes G(\{X,Y\})
	\arrow[r] \arrow[d]
	&\{FY,FZ\} \otimes \{FX,FY\}
	\arrow[dd]
	\\
	G(\{Y,Z\}\otimes \{X,Y\})
	\arrow[d]
	\\
	\{X,Z\}
	\arrow[r]
	&G(\{FX,FZ\});
\end{tikzcd}
\begin{tikzcd}
 	\II
	\arrow[r] \arrow[d]
	&G(\II)
	\arrow[d]
	\\
	\{FX,FX\}
	& G(\{X,X\}) .
	\arrow[l] 
\end{tikzcd}
 \]

Moreover, let us consider another lax monoidal functor $G': \categ E \to \categ F$ a monoidal natural transformation $G \to G'$, and another functor $F': \categ C \to \categ D$ equipped with a $G'$-strength. Then, a natural transformation $F \to F'$ is strong if and only if the following diagram commutes
 \[
\begin{tikzcd}
 	F(\langle X, \coalgebra V\rangle)
	\arrow[r] \arrow[d]
	& \langle F(X), G(\coalgebra V)\rangle
	\arrow[rd]
	\\
	F'(\langle X, \coalgebra V\rangle)
	\arrow[r]
	&\langle F'(X), G'(\coalgebra V)\rangle
	\arrow[r]
	& \langle F'(X), G(\coalgebra V)\rangle.
\end{tikzcd}
 \]
 for any $X \in \categ C, \coalgebra V \in \categ E$, if and only if the following diagram commutes
 \[
\begin{tikzcd}
 	G(\{X,Y\}) 
	\arrow[r] \arrow[d]
	& \{F(X),F(Y)\}
	\arrow[rd]
	\\
	G'(\{X,Y\}) 
	\arrow[r]
	&\{F'(X),F'(Y)\}
	\arrow[r]
	& \{F(X),F'(Y)\}.
\end{tikzcd}
 \]
 for any $X,Y \in \categ C$.

 \subsection{Strong adjunctions}
\label{sectionstrongadj}

Let us consider an adjunction
\[
\begin{tikzcd}
 	\categ C
	\arrow[rr, shift left, "L"]
	&& \categ D .
	\arrow[ll, shift left, "R"]
\end{tikzcd}
\]
We suppose that $\categ C$ and $\categ D$ are TCE over a monoidal category $\categ E$ and $\categ F$.

\begin{definition}\label{definitionstrongadj}
The adjunction $L \dashv R$ is said to be strong if $L$ and $R$ are equipped with strengths with respect to $\categ E$ so that the natural transformations
$\id \to RL$ and $LR \to \id$ are strong.
\end{definition}

Equivalently, by doctrinal adjunction, $L \dashv R$ is strong if $L$ is equipped with a strength so that the natural morphism
$$
    V \boxtimes L(X) \to L(V \boxtimes X)
$$
is an isomorphism for any $V,X \in \categ E \times \categ C$. Then, the strength on $R$ is given by the formula
$$
 V \boxtimes R(Y) \to RL(V \boxtimes R(Y))
 \simeq R(V \boxtimes LR(Y)) \to 
 R(V \boxtimes Y) .
$$

One has also an adjunction relating the opposite categories
\[
\begin{tikzcd}
 	\categ D^{\op}
	\arrow[rr, shift left, "R^{\op}"]
	&& \categ C^{\op} .
	\arrow[ll, shift left, "L^{\op}"]
\end{tikzcd}
\]
The cotensorisation of $\categ C$ and $\categ D$ over $\categ E$ are actually tensorisation of $\categ C^{\op}$ and $\categ D^{\op}$ over $\categ E^{tr}$.

\begin{proposition}
 The structure of a strong adjunction on $L \dashv R$ with respect the tensorisation of $\categ C$ and $\categ D$ over $\categ E$ is equivalent to the structure on a strong adjunction on $R^{\op} \dashv L^{\op}$ with respect the tensorisation of $\categ D^{\op}$ and $\categ C^{\op}$ over $\categ E^{tr}$.
\end{proposition}

\begin{proof}
In Subsection \ref{sectionstrength}, we saw that we have a one to one correspondence between the strengths on $L$ (resp. $R$) with respect to $\categ E$ and the strengths on $L^{op}$ (resp. $R^{op}$) with respect to $\categ E^{tr}$.
Then, given a pair of strength on $L$ and $R$, the natural transformations $\id \to RL$ and $LR \to \id$ are strong if and only if the natural transformations $\id \to L^{\op}R^{\op}$ and $R^{\op}L^{\op} \to \id$ are strong
with respect to the induced strength on $L^{op}$ and $R^{op}$.
\end{proof}

\begin{proposition}
 Let us suppose that the functors $L$ and $R$ are equipped with strengths. Then, the following assertions are equivalent.
 \begin{enumerate}
     \item \label{ptone} the natural transformations $\id \to RL$ and $LR \to \id$ are strong (hence, the adjunction $L \dashv R$ is strong);
     \item \label{pttwo} the natural maps
     \begin{align*}
     &\{LX , Y\} \to \{RLX , RY\} \to \{X , RY\};
     \\
     &\{X , RY\} \to \{LX , LRY\} \to \{LX , Y\};
     \end{align*}
     are isomorphisms inverse to each other for any $X,Y \in \categ C \times \categ D$.
 \end{enumerate}
\end{proposition}

\begin{proof}
We know from Subsection \ref{sectionstrength} that the fact that the natural transformations $\id \to RL$ and $LR \to \id$ are strong is equivalent to the commutation of the following two diagrams
$$
\begin{tikzcd}
     \{X,Y\} 
     \ar[r] \ar[d]
     & \{ RX,RY\}
     \ar[d]
     \\
     \{LR X ,Y\}
     & \{ LRX,LRY\}
     \ar[l]
\end{tikzcd}
\quad
\begin{tikzcd}
     \{X',Y'\} 
     \ar[r] \ar[d]
     & \{ LX',LY'\}
     \ar[d]
     \\
     \{X' , RL Y'\}
     & \{ RLX',RLY'\} 
     \ar[l]
\end{tikzcd}
$$
for any $X,X',Y,Y'$.

Let us suppose (\ref{ptone}). Using the commutation of the diagram just above, it is straightforward to prove that the composite maps
 \begin{align*}
 &\{LX,Y\} \to \{X,RY\} \to \{LX,Y\} ,
 \\
 &\{X,RY\} \to \{LX,Y\} \to \{X,RY\} ,
 \end{align*}
 are identities.
 
Conversely, let us suppose (\ref{pttwo}). By naturality, the two following diagrams are commutative
$$
\begin{tikzcd}
     \{X,Y\} 
     \ar[r] \ar[d]
     & \{ RX,RY\}
     \\
     \{LR X ,Y\}
     \ar[r]
     & \{ RLRX,RY\}
     \ar[u]
\end{tikzcd}
\quad
\begin{tikzcd}
     \{X',Y'\} 
     \ar[r] \ar[d]
     & \{ LX',LY'\}
     \\
     \{X' , RL Y'\}
     \ar[r]
     & \{ LX',LRLY'\}
     \ar[u]
\end{tikzcd}
$$
for any $X,X',Y,Y'$. Then, the commutation of the two previous squares follows from the fact that the maps
\begin{align*}
    &\{LR X ,Y\} \to \{ RX,RY\}
    \\
    &\{X' , RL Y'\} \to \{ LX',LY'\}
\end{align*}
are respective inverses of the maps that appear in these previous diagrams.
\end{proof}

\subsection{Homotopical enrichment}

\begin{definition}[Homotopical enrichment]\label{defin:almosthomotop}\leavevmode
Let $\categ M$ be a model category and let $\categ{E}$ be a monoidal model category. We say that  $\categ M$ is homotopically TCE over $\categ{E}$ if it TCE over $\categ{E}$ and if for any cofibration $f: X \ra X'$ in $\categ M$ and any fibration $g : Y \ra Y'$ in $\categ M$, the morphism in $\categ{E}$:
 $$
 \{X',Y\} \ra \{X',Y\} \times_{\{X,Y'\}} \{X,Y\}
 $$ 
 is a fibration. Moreover, we require this morphism to be a weak equivalence whenever $f$ or $g$ is a weak equivalence.
\end{definition}

This is equivalent to the fact that for any cofibration $f:X \to Y$ in $\categ M$ and any cofibration $g:\coalgebra V \to \coalgebra W$ in $\categ E$,
the morphism
\[
	\coalgebra V \boxtimes Y \coprod_{\coalgebra V \boxtimes X} \coalgebra W \boxtimes X \to \coalgebra W \boxtimes Y
\]
is a cofibration and it is acyclic whenever $f$ or $g$ is acyclic. This is also equivalent to the fact that for any 
fibration $f:X \to Y$ in $\categ M$ and any cofibration $g:\coalgebra V \to \coalgebra W$ in $\categ E$, the morphism
\[
	\langle X, \coalgebra W \rangle \to \langle X, \coalgebra V \rangle \times_{\langle Y, \coalgebra V \rangle} \langle Y, \coalgebra W \rangle
\]
is a fibration and it is acyclic whenever $f$ or $g$ is acyclic.


\section{Monads, comonads, limits and colimits}
\label{appendixcomonads}

\begin{proposition}
Let $M$ be a monad on a category $\categ C$. Then, the functor $U^M$ preserves and creates limits
that may exist in $\categ C$.
\end{proposition}

\begin{proof}
The functor $U^M$ preserves limits since it is right adjoint. Moreover, it reflects limits since it is conservative. Finally, for any diagram $D : \categ I \to \catalg{\categ C}{M}$, if $U^M \circ D$ has a limit, then this limit has the structure of a $M$-algebra given by
 $$
    M(\varprojlim U^MD) \to \varprojlim MU^MD
    = \varprojlim U^MT_MU^MD\to \varprojlim U^MD .
 $$
 This $M$-algebra is the limit of the diagram $D$.
\end{proof}

\begin{corollary}
 Let $Q$ be a comonad on a category $\categ C$. Then, the functor $U_Q$ preserves and creates colimits
 that may exist in $\categ C$.
\end{corollary}

\begin{proposition}\label{prop: alg reflex}
 Let $M$ be a monad on a category $\categ C$. Let us suppose $\categ C$ has all reflexive coequalisers and that $M$
 preserves these reflexive coequalisers. Then the category of $M$-algebras has all reflexive coequalisers and these are preserved by
 the functor $U^M$.
\end{proposition}

\begin{proof}
 Let $I$ be the category generated by
\begin{itemize}
 \itemt two objects $0$ and $1$;
 \itemt three nontrivial morphisms $f,g: 0 \to 1$ and $h: 1 \to 0$
\end{itemize}
with the relation $fs =gs =\id_1$.
For any diagram $D : I \to \catalg{\categ C}{M}$, let $\algebra A$ be the colimit of the functor $U^M \circ D$.
Then, $\algebra A$ has the structure of a $M$ algebra as follows
\[
	M(\algebra A) = M (\colim U^M \circ D) \simeq  \colim (M U^M \circ D) = \colim(U^M T_M U^M \circ D)
	\to  \colim  U^M \circ D = \algebra A .
\]
A straightforward check shows that this defines the structure of a $M$-algebra
on $A$ which gives the colimit of
the diagram $D$. It is then clear that $U^M$ preserves reflexive coequalisers.
\end{proof}

\begin{proposition}\label{prop: alg cocomp}
In the context of Proposition \ref{prop: alg reflex}, let us suppose that $\categ C$ is cocomplete. Then the category of $M$-algebras is cocomplete.
\end{proposition}

\begin{proof}
Given a family of $M$-algebras $(\algebra A_i)_{i \in J}$, their coproduct is given by the
reflexive coequaliser of the following diagram
\[
\begin{tikzcd}
 	T_M(\coprod_i M(\algebra A_i))
	\ar[r, bend left] \ar[r, bend right]
	& T_M(\coprod_i \algebra A_i) .
	\ar[l]
\end{tikzcd}
\]
It is enough to have all reflexive coequalisers and all coproducts to be cocomplete.
\end{proof}

\begin{corollary}\label{prop: cog coreflex}
 Let $Q$ be a comonad on a category $\categ C$. Let us suppose $\categ C$ has all coreflexive equalisers and that $Q$
 preserves these coreflexive equalisers. Then the category of $Q$-coalgebras has all coreflexive equalisers and these are preserved by
 the functor $U_Q$.
\end{corollary}

\begin{corollary}
 In the context of Corollary \ref{prop: cog coreflex}, let us suppose that $\categ C$ is complete. Then the category of $Q$-coalgebras is complete.
\end{corollary}


\section{Trees}

\label{appendixtree}

\begin{definition}
    A tree (also called planar tree) $t = (\mathrm{edges}(t), \mathrm{leaves}(t))$
    is the data of a non empty finite set
    $\mathrm{edges}(t)$ called the set of edges of $t$t and
    a subset $\mathrm{leaves}(t)$ called the set of leaves of $t$
    together with
    two partial order on $\mathrm{edges}(t)$ 
    \begin{itemize}
        \itemt the height order $\leq$, that has a minimal element called the
        root, so that leaves are maximal elements,
        and so that for any edge $e$ the poset
        $$
        \mathrm{edges}(t)_{\leq e} = \{e' \leq e\}
        $$
        is linear;
        \itemt the planar order $\leq_\pl$ that is a linear order and so that
        $$
        e' \leq e \implies e' \leq_\pl e .
        $$
    \end{itemize}
   \end{definition}

   \begin{definition}
        Given an edge $e$ of a tree $t$, the descendants of
        $e$ are the edges $e'$ so that $e' > e$. The children
        of $e$ are its immediate descendants, that is the edges
        $e' > e$ so that for any other edge $e''$:
        $$
        e' \geq e'' > e \implies e' = e''.
        $$
   \end{definition}

   \begin{definition}
    A subtree $t'$ of a tree $t$ is the data of
    \begin{itemize}
        \itemt non empty subset
        $\mathrm{edges}(t') \subseteq \mathrm{edges}(t)$
        that has a minimal element (not necessarily the root of $t$)
        and so that for any of its element $e$, if it contains
        a descendant of $e$, then it
        contains all its children;
        \itemt a subset $\mathrm{leaves}(t')$ of the set of
        maximal elements 
        of $\mathrm{edges}(t')$ that contains the intersection
        $\mathrm{edges}(t') \cap \mathrm{leaves}(t')$ and that contains
        all the maximal elements of $\mathrm{leaves}(t')$
        that are not maximal in $\mathrm{leaves}(t)$.
    \end{itemize}
    A subtree is in particular a tree.
    
    A node of a tree $t$ is a subtree that contains just a non leaf
    edge $e$ and 
    its children which become the leaves of the node.
   \end{definition}

The following picture represents a tree whose set of edges
is 
   $$
    \{r, e_1, e_2, e_3, l_1, l_2, l_3, l_4\}
   $$
whose root is $r$ and whose set of leaves is $\{l_1, l_2, l_3, l_4\}$.
The height order is given by the relations
\begin{align*}
    &l_1, l_2 > e_1;
    \\
    &e_3, l_4 > e_2;
    \\
    & e_1, l_3, e_2 > r.
\end{align*}
The planar order is
$$
r< e_1<l_1<l_2<l_3<e_2<e_3<l_4 .
$$
The nodes are represented using dots.
$$
\tikz{
\draw (2.5,0) -- (2.5,2);
\fill (2.5,1) circle (3pt);
\draw (2.5,1) -- (1,2);
\draw (2.5,1) -- (4,2);
\draw (1,2) -- (0,3);
\fill (1,2) circle (3pt);
\draw (1,2) -- (2,3);
\draw (4,2) -- (3,3);
\fill (4,2) circle (3pt);
\draw (4,2) -- (5,3);
\fill (3,3) circle (3pt);
\draw (0.5,2.5) node[below left] {$l_1$};
\draw (1.5,2.5) node[below right] {$l_2$};
\draw (2.5,2) node[above] {$l_3$};
\draw (4.5,2.5) node[below right] {$l_4$};
\draw (1.75,1.5) node[below left] {$e_1$};
\draw (3.25,1.5) node[below right] {$e_2$};
\draw (3.5,2.5) node[below left] {$e_3$};
\draw (2.5,0.5) node[left] {$r$};
}
$$

   \begin{definition}
       Let $n$ be a natural integer. An $n$ corolla
       is a tree whose set of edges contains exactly the root and $n$
       leaves.
   \end{definition}

   \begin{definition}
       Let $t$ be a tree and let $e$ be an edge. The height of $e$ in
       $t$ is the cardinal of the set
       $$
        \{e' < e\} .
       $$
       In particular the height of the root is zero.
       Then, the height of a node in $t$ is the height of its root.
       Finally, the height of the tree $t$ is the maximal height
       of its nodes plus one
       $$
        \mathrm{height}(t) = \mathrm{max}_{n \text{ node}} \mathrm{height}(n)+1.
       $$
       If $t$ has no node, then its height is zero.
   \end{definition}

   \begin{definition}
       An isomorphism of trees from $t$ to $t'$ is the data
       of an isomorphism 
       $$
       \phi: \mathrm{edges}(t) \simeq \mathrm{edges}(t')
       $$
       that sends leaves to leaves and that preserves the height order.
       This defines the groupoid of trees.
   \end{definition}

   \begin{definition}
    A planar isomorphism of trees from $t$ to $t'$ is an isomorphism
    of trees that also preserves the planar order.
    This defines the groupoid of planar trees.
\end{definition}

\bibliographystyle{amsalpha}
\bibliography{bib-mapping-coalgebra}

\end{document}